\documentclass[aihp]{imsart}

\RequirePackage{amsthm,amsmath,amsfonts,amssymb}
\RequirePackage[numbers]{natbib}
\usepackage{enumitem}
\usepackage[mathcal]{euscript}
\usepackage{xcolor}
\usepackage{url}
\usepackage{hyperref}
\usepackage{physics}

\startlocaldefs
\numberwithin{equation}{section}
\theoremstyle{plain}
\newtheorem{theorem}{Theorem}[section]
\newtheorem{lemma}[theorem]{Lemma}
\newtheorem{corollary}[theorem]{Corollary}
\newtheorem{proposition}[theorem]{Proposition}

\theoremstyle{definition}
\newtheorem{remark}[theorem]{Remark}
\newtheorem{definition}[theorem]{Definition}

\DeclareMathOperator{\Var}{Var}
\DeclareMathOperator{\cW}{\mathcal{W}}

\DeclareMathOperator{\cV}{\mathcal{V}}
\DeclareMathOperator{\utr}{\theta}
\DeclareMathOperator{\ntr}{\hat{\theta}}

\DeclareMathOperator{\cU}{\mathcal{U}}

\newcommand{\clf}{\mathcal{F}}
\newcommand{\clp}{\mathcal{P}}
\newcommand{\PP}{\mathbb{P}}
\newcommand{\RR}{\mathbb{R}}
\newcommand{\NN}{\mathbb{N}}

\endlocaldefs

\begin{document}

\begin{frontmatter}

\title{Strong existence, pathwise uniqueness and chains of collisions in infinite Brownian particle systems}
\runtitle{Wellposedness of Singular SDE}

\begin{aug}
\author[A]{\inits{SB}\fnms{Sayan}~\snm{Banerjee}\ead[label=e1]{sayan@email.unc.edu}},
\author[A]{\inits{AB}\fnms{Amarjit}~\snm{Budhiraja}\ead[label=e2]{budhiraj@email.unc.edu}}
\and
\author[B]{\inits{PR}\fnms{Peter}~\snm{Rudzis}\ead[label=e3]{peter.rudzis@gu.se}}
\address[A]{Department of Statistics and Operations Research, UNC Chapel Hill\printead[presep={,\ }]{e1,e2}}

\address[B]{Department of Mathematical Sciences, University of Gothenburg\printead[presep={,\ }]{e3}}
\end{aug}

\begin{abstract}
We study strong existence and pathwise uniqueness for a class of  infinite-dimensional singular stochastic differential equations (SDE), with state space as the  cone $\{x \in \RR^{\NN}: -\infty < x_1 \le x_2 \le \cdots\}$, referred to as an {\em infinite system of  competing Brownian particles}. A `mass' parameter taking values in the interval $[0,1]$ governs the splitting proportions of the singular collision local time between adjacent state coordinates. Solutions in the case when this parameter is $1/2$ correspond to the well-studied rank-based diffusions, while the general case arises from scaling limits of interacting particle systems on the lattice with asymmetric interactions and the study of the KPZ equation. Under conditions on the initial configuration, the drift vector, and the growth of the local time terms, we establish pathwise uniqueness and strong existence of solutions to the SDE. A key observation is  the connection between pathwise uniqueness and the finiteness of `chains of collisions' between adjacent particles influencing a tagged particle in the system. Ingredients in the proofs include classical comparison and monotonicity arguments for reflected Brownian motions, techniques from Brownian last-passage percolation, large deviation bounds for random matrix eigenvalues, and concentration estimates for extrema of Gaussian processes.
\end{abstract}

\begin{abstract}[language=french]
Nous étudions l’existence forte et l’unicité trajectorielle pour une classe d’équations différentielles stochastiques (EDS) singulières en dimension infinie, dont l’espace d’état est le cône $\{x \in \RR^{\NN} : -\infty < x_1 \le x_2 \le \cdots\}$, appelé un \emph{système infini de particules browniennes en compétition}. Un paramètre de «~masse~» prenant ses valeurs dans l’intervalle \([0,1]\) gouverne les proportions de partage du temps local de collision singulier entre coordonnées adjacentes de l’état. Les solutions dans le cas où ce paramètre vaut \(1/2\) correspondent aux diffusions bien connues basées sur le rang, tandis que le cas général apparaît comme limite d’échelle de systèmes de particules en interaction sur le réseau avec des interactions asymétriques, ainsi que dans l’étude de l’équation KPZ. Sous des conditions portant sur la configuration initiale, le vecteur de dérive et la croissance des termes de temps local, nous établissons l’unicité trajectorielle et l’existence forte des solutions de l’EDS. Une observation clé est le lien entre l’unicité trajectorielle et la finitude des «~chaînes de collisions~» entre particules adjacentes influençant une particule marquée du système. Les preuves font intervenir des arguments classiques de comparaison et de monotonie pour les mouvements browniens réfléchis, des techniques issues de la percolation brownienne de dernier passage, des bornes de grandes déviations pour les valeurs propres de matrices aléatoires, ainsi que des estimations de concentration pour les extrema de processus gaussiens.
\end{abstract}



\end{frontmatter}


\section{Introduction}

In this article, we investigate the strong existence and pathwise uniqueness of solutions to the infinite system of stochastic differential equations
\begin{equation} \label{eq:loctimeeq0}
X_j(t) = x_j + g_jt + B_j(t) + p L_{(j-1,j)}(t) - q L_{(j, j+1)}(t), \quad t \in [0,\infty), \quad j \in \mathbb{N}.
\end{equation}
Here, $\bar{x} = \{x_j\}_{j \in  \mathbb{N}}$ is a non-decreasing sequence in $\mathbb{R}$, $p = 1-q \in [0,1]$, $\bar{g} = \{g_j\}_{j \in \mathbb{N}} \in \mathbb{R}^{\mathbb{N}}$, $\{B_j\}_{j \in \mathbb{N}}$ is a sequence of independent one-dimensional standard Brownian motions, and $L_{(j,j+1)}$ denotes the local time of collisions between the $j$th and $(j+1)$st particles (with $L_{(0,1)} \equiv 0$), which preserves the particle order: $X_j(t) \le X_{j+1}(t)$ for all $t \ge 0, j \in \mathbb{N}$. These equations describe the collective motion of an infinite system of interacting Brownian particles, where two particles, upon collision, maintain their order on the real line but undergo a possibly unequal exchange of `momentum' (which can be heuristically thought of as the particles having unequal `masses'). 

The study of the well-posedness of the system \eqref{eq:loctimeeq0} is motivated in part by its connections to rank-based diffusions, asymmetric exclusion processes, and other well-studied models in statistical physics, reviewed below. Our results pave the way for direct analyses of infinite systems of Brownian particles interacting via collision local times, without recourse to particular solution frameworks such as finite-particle approximations. They also lay the groundwork for a forthcoming paper showing that, in the weakly asymmetric regime, the appropriately rescaled fluctuation field of the corresponding doubly infinite (i.e.\ $\mathbb{Z}$-indexed) particle system converges to a solution to the Kardar–Parisi–Zhang (KPZ) equation.

\subsection{Background}
Systems described by \eqref{eq:loctimeeq0} first arose in the case $p=1/2$ in the context of \emph{rank-based diffusions} in stochastic portfolio theory \cite{fernholz2002stochastic,banner2005atlas,fernholz2009stochastic, banner2011hybrid}. In such systems, the drift and diffusivity of each particle (representing the {market capitalization of a stock}) depend on the order of the particles on the real line. It is described by a system of equations of the following form
\begin{equation}\label{RBD}
Y_j(t) = x_j + \int_0^t\sum_{\ell=1}^{N}1_{[Y_{(\ell)}(s)=Y_j(s)]}\left(g_{\ell}ds + \sigma_{\ell}\dd W_j(s)\right),\quad t \in [0,\infty), \quad 1 \leq j \leq N.
\end{equation}
Here $N$ can be any finite positive integer or infinite (where in the latter case, it is understood that $j$ ranges over $\mathbb{N}$), $Y_{(1)}(s) \leq Y_{(2)}(s) \leq Y_{(3)}(s) \leq \cdots$ denote the order statistics at time $s$ of the particle positions $\{Y_j\}$ (provided such an ordering exists)
, $\{W_j\}$ are independent standard Brownian motions, and $g_\ell \in \mathbb{R}$ and $\sigma_{\ell}>0$ denote the drift and diffusivity, respectively, of the $\ell$th order statistic. We refer to $\{Y_j\}$ (resp. $\{Y_{(j)}\}$) as the state process for the {\em labeled} (resp. {\em ranked}) particle system. When $\sigma_\ell=1$ for all $\ell$, it can be verified \cite[Lemma 3.5]{sarantsev2017infinite}, under conditions on $(\bar x, \bar g)$, that the dynamics of the ranked trajectories $\{X_j = Y_{(j)}\}$ can be expressed in the form \eqref{eq:loctimeeq0} with $p=1/2$, $$B_j(t) = \int_0^t\sum_{\ell=1}^{\infty}1_{[Y_{(j)}(s)=Y_\ell(s)]}\dd W_\ell(s), \;\; t \geq 0,$$ and $L_{(j,j+1)}$ taken as the semimartingale local time of $X_{j+1} - X_{j}$ at the origin. Equation \eqref{RBD} is non-standard and interesting from the point of view of the classical theory of stochastic differential equations because the drift and diffusivity of a particle vary in a discontinuous manner as its rank changes in the system. Consequently, rank-based diffusions have been a subject of extensive research. Results on existence and uniqueness of solutions to \eqref{RBD} (in both the weak and strong sense, in finite and infinite particle settings) and on absence of triple collisions (three particles at the same place at the same time) have been studied in \cite{bass1987uniqueness,PP,sarantsev2017infinite,shkol2011levy, ichiba2013strong, ichiba2010collisions,sarantsev2015triple,ichiba2017yet,fernholz2013planar}.


In the finite particle setting, the `order statistics description' \eqref{eq:loctimeeq0} has proven more tractable than \eqref{RBD}, as one can in this case appeal to the extensive theory of reflected Brownian motions (RBM). In particular, when the number of particles is $N < \infty$, the \emph{gap process} $Z_j := X_{j+1} - X_j, 1 \le j \le N-1,$ is a RBM in the orthant $\RR_+^{N-1}$, and the powerful techniques of \cite{HR,harrison1987multidimensional,harrison1987brownian} can be used to analyze the existence, uniqueness and detailed pathwise behavior of solutions to \eqref{eq:loctimeeq0}. Moreover, under the so-called Harrison-Williams stability conditions, such RBM have a unique stationary distribution which, under additional {\em skew-symmetry} conditions on the coefficients, is given by a product of exponential random variables \cite{harrison1987multidimensional}. Rates of convergence to stationarity have been investigated in \cite{budlee, ichiba2013convergence,blanchet2020rates,BanBudh,banerjee2020dimension}. By careful comparison with appropriately coupled finite particle systems, significant progress has also been made in identifying and studying the stationary distributions of the gap process of solutions to the infinite particle system \eqref{eq:loctimeeq0}  (constructed in an appropriate sense) {\cite{PP,sarantsev2017stationary,tsai_stat} (see also \cite{RuzAiz})}. For example, when $p=1/2$, under conditions on $\bar g$, the gap process in the $N = \infty$ case has a one-parameter family of product form stationary distributions \cite{sarantsev2017stationary} that are extremal \cite{banerjee2024extremal} and have an intricate local stability structure and whose associated `domains of attraction' have been studied in \cite{DJO,banerjee2022domains}.  

The case $p \neq 1/2$ has gained more attention recently. It was shown in \cite{karatzas2016systems} that \eqref{eq:loctimeeq0}, in the setting of finitely many particles, arises as a scaling limit of \emph{asymmetric simple exclusion processes (ASEP)} on the lattice.  Similar results are also expected to  hold, in a suitable sense, for infinite particle systems. These models are conjectured to be in the \emph{Kardar-Parisi-Zhang (KPZ)} universality class \cite{WFSbook,SS}, an intensive area of modern research in probability {(cf. \cite{corwin2012kardar,dembo2016weakly,
tsai2016weak} and references therein)}. Roughly speaking, systems in the KPZ universality class exhibit non-Gaussian temporal fluctuations with scaling exponent $1/3$ (compared to $1/2$ for Gaussian). The appropriately rescaled fluctuations are conjectured (and in a few cases proven) to converge to a non-linear stochastic partial differential equation known as the KPZ equation \cite[Equation (1.0.1)]{WFSbook}. When $p=0$, this connection was rigorously established in \cite{WFSbook} with specific initial particle configurations. Among other achievements, \cite{WFSbook} introduced techniques from integrable systems in statistical physics, last-passage percolation and random matrix theory to the world of RBM. When $p \in (0,1), p \neq 1/2$, the model is still so-called Bethe integrable, but less tractable. It has been partially investigated in \cite{SS}, and asymptotic analysis has been carried out for the half-Poisson initial condition. See Remark \ref{SSrem} for some comments on the relation of this paper with our work.
Another connection with our model may be found in the \textit{multilevel Dyson model} with parameter $\beta = 2$, whose left edge evolves (up to a sign change) according to the finite version of the system \eqref{eq:loctimeeq0} with $p = 0$ \cite{warren2007, gorshkol2016MD}. A consequence is that the marginal distribution at any time $t > 0$ of this system, started from packed initial conditions (all particles at the same site), coincides with that of a GUE-corners process, a model which shows up as a universal scaling limit for many two-dimensional statistical mechanics models (see references in \cite[Section 1.1]{gorshkol2014Jack}). For the finite particle system, when $p \neq 1/2$, `dimension-free' convergence phenomena were established in \cite{banerjee2020dimension}, which showed that a fixed collection of gaps converge to their stationary marginals at a rate independent of the total number of particles. Asymptotics of high moments of solutions to the Stochastic Heat Equation with multiplicative noise, and consequent upper tail large deviations of the KPZ equation, were recently connected with rank-based diffusions (i.e. \eqref{eq:loctimeeq0} with $p=1/2$) in \cite{tsai2025high}.

\subsection{Goals and Challenges}

Somewhat surprisingly, despite having detailed results on pathwise properties of solutions (if they exist), the basic wellposedness of solutions to the infinite particle system \eqref{eq:loctimeeq0} is not well-understood. Strong existence and pathwise uniqueness for the finite particle version follows from classical results for the Skorohod problem (see Remark \ref{rem:veryunique}). When $p \ge 1/2$, one can appeal to monotonicity properties of finite particle systems \cite{AS2} to obtain strong solutions to \eqref{eq:loctimeeq0}, known as \emph{approximative solutions (versions)}, by taking a pathwise limit of the finite particle systems uniformly on compact time intervals \cite{sarantsev2017infinite}. However, to date, it was not known whether there is uniqueness (even in law) for systems of the form \eqref{eq:loctimeeq0}.
Moreover, approximative solutions in \cite{sarantsev2017infinite} were constructed through comparison techniques applied to the model with $p=1/2$, yielding results only when $p \ge 1/2$. When $p \in (0,1/2)$, even existence of solutions to \eqref{eq:loctimeeq0} has been an open problem. Strong existence has been established when $p=0$ in \cite{WFSbook} using last-passage percolation techniques.
Making progress on the above open questions is the main goal of this work. 

A key technical challenge involved in addressing such questions lies in quantifying the influence of the high-ranked particles on the low-ranked particles. When the particles are initially closely packed, there could possibly be infinite \emph{chains of collisions} in a compact time interval, making it impossible to `decouple' (in a suitable sense) the dynamics of a finite number of particles from the infinite-dimensional system on this interval. 
For the case $p=1/2$, when one considers the evolution of the labeled particle system given by \eqref{RBD},  such a decoupling can be analyzed under the assumption that the drift and diffusivity of all particles agree after a certain rank (see \cite[Assumption 1]{ichiba2013strong} and \cite[Theorem 3.1]{sarantsev2017infinite}), or the diffusivity coefficients are all constant and drift coefficients decay sufficiently quickly as the particle ranks increase (\cite[Theorem 3.2]{sarantsev2017infinite}). This has been the key to establishing existence of strong solutions and pathwise uniqueness for solutions to \eqref{RBD} in \cite{ichiba2013strong}
, and initial conditions that are not too closely packed, specifically, when $x_i \ge \gamma_1 i + \gamma_2, i \in \NN,$ for some $\gamma_1>0, \gamma_2 \in \RR$. 
However,  this analysis is restricted to the  setting of the labeled particle system given by  \eqref{RBD} (corresponding to the case $p=1/2$ in the associated system of ordered particles \eqref{eq:loctimeeq0}), in which the absence of collision local times enables one to analyze such a decoupling by using the basic observation that for $Y_j$ to influence $Y_1$ on a compact time interval, at least one of the driving Brownian motions $W_j$ and $W_1$ must exhibit a large deviation from its typical behavior on this interval, the probability of which can be explicitly quantified. Furthermore, for existence when $p=1/2$, the correspondence between labeled and ranked particle systems    provides existence of weak solutions to \eqref{eq:loctimeeq0} using results of \cite{ichiba2013strong} in an easy manner. However, these results are not immediately useful for proving existence of strong solutions to \eqref{eq:loctimeeq0}, since given a sequence of Brownian motions $\{B_j\}$  for which  a solution to  \eqref{eq:loctimeeq0} is sought, it is not straightforward to construct Brownian motions $\{W_j\}$ driving the evolution of the corresponding labeled particle system.

\subsection{Our contributions} In this paper, we establish the strong existence and pathwise uniqueness of solutions to \eqref{eq:loctimeeq0}, under appropriate assumptions, for all $p \in [0,1]$. Our results on pathwise uniqueness are stated in Theorems \ref{thm:unique1}, \ref{thm:unique2}, and \ref{thm:unique3}. When $p<1/2$, we show in Theorem \ref{thm:unique1} that solutions to \eqref{eq:loctimeeq0} are essentially always pathwise unique. Although, surprisingly, we do not require any assumption on the initial configuration, in Lemma \ref{lem:biginf2}, we observe that for a strong solution to exist when $p<1/2$, the initial data must satisfy $
\liminf_{M \to \infty} \frac{x_M}{\sqrt{M}} = \infty
$ almost surely (see Remark \ref{plinitial}). When $p \ge 1/2$, we establish pathwise uniqueness in Theorem \ref{thm:unique2} under the assumption $\liminf_{M \to \infty} \frac{x_M}{\sqrt{M}} >0$ almost surely and an additional assumption \eqref{eq:cncstar} on the local times. This condition requires, roughly speaking, that the net effect of the local times on the $j$th particle does not grow too rapidly as $j$ grows, uniformly in $t$ ranging in a compact interval. Proposition \ref{prop:lcond}(i) shows \eqref{eq:cncstar} holds for approximative solutions constructed in \cite{sarantsev2017infinite}, which suggests that this growth condition on the local times  is natural for the problem. For the case $p>1/2$, we also formulate a somewhat different growth condition \eqref{eq:cncstar2} on local times that is easier to interpret, simply requiring that the local times $L_{(M, M+1)}$ grow at most at a sufficiently small exponential rate in $M$. This condition is even simpler to verify for approximative versions (see the proof of Proposition \ref{prop:lcond}(ii)). We show in Theorem \ref{thm:unique3} that if the initial conditions satisfy $\limsup_{M \to \infty} \frac{x_M}{\sqrt{M}} =\infty$, then pathwise uniqueness holds for $p>1/2$ among solutions satisfying the growth condition \eqref{eq:cncstar2} on local times. Note that the two types of conditions on the initial data mentioned above are not directly comparable.

We show in Theorem \ref{thm:exist} that, when $p < 1/2$,  strong existence of solutions to \eqref{eq:loctimeeq0} holds if the distribution $\gamma$ of the initial data $\bar x =  \{x_j\}_{j \in  \mathbb{N}}$ satisfies $\sum_{k = 1}^\infty \gamma(x_k \leq k^\chi) < \infty$ for some $\chi>1/2$. 
In fact our results show that this solution is an approximative solution of \eqref{eq:loctimeeq0} in the sense of \cite{sarantsev2017infinite}.
As was noted earlier, approximative solutions were previously known to hold only for $p\ge 1/2$ and in fact there are no prior existence results on strong solutions (approximative or any other type) in the case $0<p < 1/2$. 
The proof of Theorem \ref{thm:exist} relies on the explicit form of a strong solution in the $p=0$ case given in \cite{WFSbook}  along with comparison techniques of \cite{sarantsev2017infinite} and Girsanov's theorem.  For the case $p\ge 1/2$, strong solutions are known to exist from \cite{sarantsev2017infinite} which shows that  solutions  to \eqref{eq:loctimeeq0} can be constructed as approximative versions   when the initial distribution $\gamma$ satisfies  $\sum_{j = 1}^\infty e^{-c(x_j \vee 0)^2} < \infty$ $\gamma$-almost surely, for some $c>0$. 

The above existence and uniqueness results are combined in Corollary \ref{cor:eusummary} to give succinct conditions for strong existence and pathwise uniqueness. In particular, our results characterize the strong solutions of \eqref{eq:loctimeeq0} with certain properties as the approximative versions of \cite{sarantsev2017infinite}, and thus, when $p \geq 1/2$, previous local finiteness and comparison results obtained for approximative versions  (e.g. Lemma 3.8 and Corollaries 3.10-3.12 of \cite{sarantsev2017infinite}) hold for any strong solution whose initial configuration and local times satisfy appropriate conditions.

{In Remark \ref{SSrem} we will discuss connections with the model studied in \cite{SS}, where, under the assumption of uniqueness of solutions of an equation as in \eqref{eq:loctimeeq0}, contour integral and Fredholm determinant formulas were derived for particle counting statistics and used to establish KPZ type asymptotics. In particular, we show that the half-Poisson initial data considered in \cite{SS} falls within our framework, and our results provide a proof of wellposedness for the stochastic equation studied there.}

\subsection{Proof techniques} 
Beyond new existence and uniqueness results for a class of infinite-dimensional singular diffusions, the techniques developed in this paper connect classical methods for analyzing RBM using monotonicity arguments with those arising from the study in last-passage percolation and random matrix theory in novel ways. As discussed before, the key quantities of  interest in the proofs are certain chains of collisions. For $i \in \NN$ and $0 \leq u < v < \infty$, the \emph{maximal length of a collision chain} affecting the $i$th particle from above on the time interval $[u,v]$ is captured by the quantity
\[
\begin{split}
K^*(i,[u,v]) = \sup\{k \geq 0 & : \exists u \leq s_{i + k - 1} \leq \cdots \leq s_{i + 1} \leq s_i \leq v \\
& \text{ such that, for } i \leq j \leq i + k - 1, X_{j+1}(s_j) = X_j(s_j)\}.
\end{split}
\]
As shown in Lemma \ref{lem:unique}, almost sure finiteness of the above quantity leads to pathwise uniqueness, as in such a case one can appropriately decouple the dynamics of a finite number of particles from the infinite system. This fact has an intuitive explanation, namely that when $K^*$ is finite, the causal structure of the particle system is such that the collision-mediated effects of only finitely many particles can reach a tagged particle in a given time interval. Proving finiteness of $K^*$ is therefore one of the key technical ingredients in our proofs.
In Section \ref{ssec:ptwise}, we obtain upper and lower bounds on particle trajectories in terms of \emph{Brownian last-passage percolation quantities} of the form defined in \eqref{eq:VMdef}, \eqref{eq:VMplusdef} and \eqref{eq:WMdef}. Probability estimates on these quantities are obtained in Section \ref{ssec:lpp}. The main ingredients here are large deviation bounds on largest eigenvalues of random matrices, probability concentration estimates on infima of Gaussian processes, and other Gaussian process estimates.
In Section \ref{sec:pfthm2.6}, the probability estimates of Section \ref{ssec:lpp} are applied to prove finiteness of $K^*$, which is then used to complete the proof of Theorems \ref{thm:unique1}, \ref{thm:unique2}, and \ref{thm:unique3}. This section is the technical heart of the paper which brings together the various monotonicity estimates from Section \ref{ssec:ptwise} and Brownian last-passage percolation estimates from Section \ref{ssec:lpp}. Section \ref{ssec:approximative} shows that the approximative versions  satisfy the conditions \eqref{eq:cncstar} and \eqref{eq:cncstar2} on the local times. Finally, in Section \ref{proofexist}, we prove our main result on strong existence (Theorem \ref{thm:exist}).

\section{Strong existence and pathwise uniqueness}
\label{sec:sepu}
In examining the system \eqref{eq:loctimeeq0}, we will take the collision parameter $p$ to lie in the half-open interval $[0,1)$. Note that if $p = 1$, then the $k$th particle does not influence the $j$th particle for any $k > j$ and so the system can be analyzed in an elementary recursive fashion (see e.g.\ Sections 2.1-2.2 of \cite{WFSbook}). Consequently, the system when $p = 1$ is well understood and considerably simpler than when $p \neq 1$ and will not be further discussed in this work.


Fix parameters $p \in [0,1)$, $N \in \mathbb{N} \cup \{\infty\}$, and $\bar{g} = \{g_j\}_{j \in \mathbb{N}} \in \mathbb{R}^N$. Let $q=1-p$. Let $\{B_j(t), t \in [0,\infty), j \in \mathbb{N}\}$ be a collection of independent standard Brownian motions given on a filtered probability space $(\Omega, \clf, \{\clf_t\}_{t\ge 0}, \PP)$  such that $\{B_i(t+\cdot)- B_i(t), i \in \NN\}$ is independent of $\clf_t$ for all $t\ge 0$. 
Let $\bar{x} = \{x_j\}_{j \in \mathbb{N}}$ be a non-decreasing sequence of $\RR$-valued $\clf_0$-measurable random variables.

If $N < \infty$, then we let $[N]$ denote the set of all integers 1 through $N$, inclusive. If $N = \infty$, we adopt the convention that $[N] = \mathbb{N}$.

\begin{definition}[Systems of competing Brownian particles] \label{def:cbp} \normalfont

Suppose $\{X_j(t), t \in [0,\infty)\}_{j \in [N]}$ and $\{L_{(j,j+1)}(t), t \in [0,\infty)\}_{j \in \{0\} \cup [N]}$ are collections of continuous, real-valued processes adapted to $\{\mathcal{F}_t\}$ and having the following properties:
\begin{enumerate}[label=P\arabic*,ref=P\arabic*]
\item \label{pr:p1} For all $t \in [0,\infty), j \in [N]$,
\begin{equation} \label{eq:loctimeeq}
X_j(t) = x_j + g_jt + B_j(t) + p L_{(j-1,j)}(t) - q L_{(j, j+1)}(t).
\end{equation} 
\item \label{pr:p2} For all $t \in [0,\infty)$, the sequence $\{X_j(t)\}_{j \in [N]}$ is non-decreasing. 
\item \label{pr:p3} $L_{(0,1)} \equiv 0$, $L_{(N,N+1)} \equiv 0$ if $N < \infty$, and for $1 \leq j < N$, $L_{(j,j+1)}$ is a continuous, non-decreasing process, with 
$L_{(j,j+1)}(0)=0$,
which can increase only at times $t \in [0,\infty)$ such that $X_{j+1}(t) = X_j(t)$, i.e. 
\begin{equation} \label{eq:restricted_inc} 
\int_0^\infty (X_{j+1}(t) - X_j(t)) \dd L_{(j,j+1)}(t) = 0.
\end{equation} 
\end{enumerate}
If $N < \infty$ (resp.\ $N = \infty$), we refer to the collection $\{X_j\}_{j \in [N]}$ as a \textit{a system of $N$ (resp.\ an infinite system of) competing Brownian particles with parameters $p, \bar{g}$, initial conditions $\bar{x}$, and driving Brownian motions $\{B_j\}_{j \in \mathbb{N}}$.} The collection $\{L_{(j,j+1)}\}$ will be referred to as the local times associated with the solution. Somewhat more informally, when we refer to a \textit{solution} to \eqref{eq:loctimeeq}, we mean a process $\{X_j(t),t \in [0,\infty), j \in [N]\}$ given on a filtered probability space as above, with driving noises $\{B_j(t), t \in [0,\infty), j \in \mathbb{N}\}$ and initial data   $\bar{x}$, such that properties \ref{pr:p1}-\ref{pr:p3} hold.

\end{definition}

We also introduce the processes
\begin{equation} \label{eq:Vjdef}
V_j(t) = g_j t + B_j(t), \quad j \in [N], \quad t \in [0,\infty).
\end{equation}
In the terminology of \cite{AS2}, a solution $\{X_j\}_{j \in \mathbb{N}}$ to \eqref{eq:loctimeeq} is a system of $N$ (or an infinite system of) competing particles with \textit{driving function} $\overline{V} = \{x_j + V_j\}_{j \in \mathbb{N}}$ and \textit{parameters of collision} $p,q$.

Let $\clp_0(\RR^{\infty})$ be the collection of all probability measures $\gamma$ on $\RR^{\infty}$ such that, for $\gamma$-a.e. $\bar z = \{z_j\}_{j \in \NN}$, $z_k \le z_{k+1}$ for every $k \in \NN$.
\begin{definition}[Strong Existence] \label{def:strong_exist} \normalfont
For a given selection of parameters $p,\bar{g}$, and a probability measure $\gamma \in \clp_0(\RR^{\infty})$, we say that there \textit{exists a strong solution} to \eqref{eq:loctimeeq} with initial distribution $\gamma$, if, for any choice of a filtered probability space,   Brownian motions $\{B_j\}$, and
$\bar x$ as above, with $\bar x$ distributed as $\gamma$, there exist collections of continuous $\clf_t$-adapted processes $\{X_j\}$ and $\{L_{(j,j+1)}\}$ satisfying \ref{pr:p1}-\ref{pr:p3}. 
\end{definition}

\begin{definition}[Pathwise Uniqueness] \label{def:pathwise_unique} \normalfont
We say that solutions to \eqref{eq:loctimeeq} are \textit{pathwise unique} with initial distribution $\gamma \in \clp_0(\RR^{\infty})$, if, whenever $\{X_j\}$ and $\{X'_j\}$ are two strong solutions to \eqref{eq:loctimeeq} with the same driving Brownian motions $\{B_j\}$, same initial distribution $\gamma$ and $X_j(0) = X'_j(0)$ for $j \in \NN$,   then, almost surely, for all $j \in \mathbb{N}$ and $t \in [0,\infty)$, $X_j(t) = X_j'(t)$. More generally, we will talk about pathwise uniqueness among solutions satisfying some property (P), by which we mean that if, whenever $\{X_j\}$ and $\{X'_j\}$ are two strong solutions to \eqref{eq:loctimeeq} with initial distribution $\gamma$, satisfying property (P), and $X_j(0) = X'_j(0)$ for $j \in \NN$,  then, almost surely, for all $j \in \mathbb{N}$ and $t \in [0,\infty)$, $X_j(t) = X_j'(t)$.
\end{definition}

\begin{remark} \label{rem:veryunique} \normalfont

If $N < \infty$, then strong solutions to the system of $N$ competing Brownian particles are known to exist and to be pathwise unique for any choice of  $p \in [0,1]$, $\bar{g}$ and initial distribution \cite[Section 2.1]{karatzas2016systems}. In this case, the solution is adapted to the filtration generated by the Brownian motion and the initial condition (which is contained in the filtration $\{\clf_t\}$). In fact, existence and uniqueness holds in an even more general sense, namely, if $\tau$ is any $[0,\infty]$-valued random variable, then there is an a.s. unique process $\{X_j\}_{j \in [N]}$ defined on $[0, \tau)$ with continuous sample paths 
which satisfies P1-P3 of  Definition \ref{def:cbp}, with associated local times $\{L_{(j,j+1)}\}$, where each instance of the time interval $[0,\infty)$ is replaced by $[0,\tau)$. This  follows from the pathwise construction of the solution from the solution to a Skorokhod problem in the positive orthant $\mathbb{R}_+^N$ (see \cite[Lemma 2.4]{AS2}). 

\end{remark}

For the rest of this work, we focus  on the case $N = \infty$. In the statements below, for sequences $\bar{g} = \{g_j\}_{j \in \mathbb{N}} \in \mathbb{R}^\infty$, we let $|\bar{g}|_2 = (\sum_{j \in \mathbb{N}} g_j^2)^{1/2}$ and $|\bar{g}|_\infty = \sup_{j \in \mathbb{N}} |g_j|$. We write $\bar{g} \in \ell^2(\mathbb{N})$ if $|\bar{g}|_2 < \infty$, and we write $\bar{g} \in \ell^\infty(\mathbb{N})$ if $|\bar{g}|_\infty < \infty$.
We first state our main results on pathwise uniqueness of solutions to \eqref{eq:loctimeeq}. For clarity, we state three separate theorems for different choices of the parameters $p$ and $q$, although parts of the proofs overlap substantially. When $p < q$, pathwise uniqueness holds in great generality.
\begin{theorem} \label{thm:unique1}
If $p = 1 - q \in [0,1/2)$ and $\bar{g} \in \ell^\infty(\mathbb{N})$, then solutions to \eqref{eq:loctimeeq} are pathwise unique for any initial distribution $\gamma \in \clp_0(\mathbb{R}^\infty)$.
\end{theorem}

When $p \geq q$, we obtain pathwise uniqueness among solutions whose initial data and collision local times satisfy certain conditions. Note that the second result only applies when $p > q$.

\begin{theorem} \label{thm:unique2}
Assume $p = 1 - q \in [1/2,1)$ and $\bar{g} \in \ell^\infty(\mathbb{N})$. Suppose $\gamma \in \clp_0(\mathbb{R}^\infty)$ has the property that for $\gamma$-a.e. $\bar{x} = \{x_j\}_{j \in \mathbb{N}}$, 
\begin{equation} \label{eq:init_cond1} \tag{D.1}
\liminf_{M \to \infty} \frac{x_M}{\sqrt{M}} > 0.
\end{equation}
Then solutions to \eqref{eq:loctimeeq} are pathwise unique for the initial distribution $\gamma$, among all solutions whose associated collection of local times $\{L_{(j,j+1)}\}$ satisfies the following condition: For all $T \in [0,\infty)$,
\begin{equation} \label{eq:cncstar} \tag{L.1}
\limsup_{M \to \infty} \sup_{s \in [0,T]} \frac{q L_{(M,M+1)}(s) - p L_{(M-1,M)}(s)}{1 \vee X_M(0)} \leq 0 \text{ a.s.}
\end{equation}
\end{theorem}

\begin{theorem} \label{thm:unique3}
Assume $p = 1 - q \in (1/2,1)$ and $\bar{g} \in \ell^\infty(\mathbb{N})$. Suppose that $\gamma \in \clp_0(\mathbb{R}^\infty)$ has the property that for $\gamma$-a.e. $\bar{x} = \{x_j\}_{j \in \mathbb{N}}$, 
\begin{equation} \label{eq:init_cond2} \tag{D.2}
\limsup_{M \to \infty} \frac{x_M}{\sqrt{M}} = \infty.
\end{equation}
Then solutions to \eqref{eq:loctimeeq} are pathwise unique for the initial distribution $\gamma$, among all solutions whose associated collection of local times $\{L_{(j,j+1)}\}$ satisfies the following condition: For all $T \in [0,\infty)$, 
\begin{equation} \label{eq:cncstar2} \tag{L.2}
\lim_{M \to \infty} \ \left( \frac{q}{p} \right)^M L_{(M,M+1)}(T) = 0 \text{ a.s.}
\end{equation}
\end{theorem}

\begin{remark} \normalfont 
The conditions \eqref{eq:cncstar} and \eqref{eq:cncstar2} needed when $p \geq q$ have a certain intuitive meaning and are not particularly restrictive to apply in practice. One can interpret both of these conditions as preventing rapid clustering of particles `at infinity'. In particular, viewing $L = \{L_{(i,i+1)}\}_{i \in \mathbb{Z}}$ as part of the solution data, \eqref{eq:cncstar2} can be regarded as a state space condition, analogous to the growth conditions encountered in, for example, PDE theory when the spatial domain is non-compact. If such rapid clustering does occur and the conditions fail to hold, there may be persisting effects on lower particles from far away particles. Hence, it may be possible to construct additional solutions to \eqref{eq:loctimeeq} by other approximation strategies via introducing a `heavy particle at infinity', which forces the rest of the particles to cluster rapidly in finite time. We leave the exploration of this non-uniqueness behavior as an open problem.

In practice, the conditions \eqref{eq:cncstar} and \eqref{eq:cncstar2} can be verified by showing that the particle trajectories $X_j$ satisfy certain mild bounds. We illustrate this approach in \S\ref{ssec:approximative}, where we verify the conditions for \textit{approximative solutions}, the most studied type of strong solution to \eqref{eq:loctimeeq0}. We will see that if $p \geq q$, then \eqref{eq:cncstar} holds for approximative solutions whose initial data satisfy \eqref{eq:init_cond1}, and if $p > q$, then \eqref{eq:cncstar2} holds for approximative solutions for any choice of initial data.

One may ask what the benefit of Theorem \ref{thm:unique3} is when Theorem \ref{thm:unique2} already covers the case where $p \geq q$. In addition to having the convenient interpretation as a state space solution noted above, \eqref{eq:cncstar2} seems to be easier to verify in practice. This can be seen in the proof of Proposition \ref{prop:lcond}, where part (i) of the proposition depends crucially on the growth rate of the initial data, as well as on estimates from rank-based diffusions allowing control of the system in the $p = q$ case, while part (ii) of the same proposition has a quite short and elementary proof that requires no growth condition on the initial data at all.
\end{remark}

\begin{remark} \normalfont
Our proof of Theorem \ref{thm:unique3} actually shows the following more general result: Suppose $p>q$ and $\{X_j\}$, $\{X'_j\}$ are two strong solutions with the same driving Brownian motions $\{B_j\}$, same initial distribution $\gamma \in \clp_0(\RR^{\infty})$ and $X_j(0) = X'_j(0)$ for all $j \in \NN$ and both solutions satisfy condition \eqref{eq:cncstar2}. Then, almost surely, for all $j \in \mathbb{N}$ and $t \in [0,T)$, $X_j(t) = X_j'(t)$, where
$$
T = 4^{-1}\left(1 + \sum_{j = 1}^\infty (q/p)^j(1 + \sqrt{j+1}) \right)^{-2} \left( \limsup_{M \to \infty} \frac{X_M(0)}{\sqrt{M}} \right)^2.
$$
When the right-most factor is infinite, we recover Theorem \ref{thm:unique3}.
\end{remark}

The following result gives existence of strong solutions. We emphasize that part (ii) of the theorem is not new and has been established in \cite{sarantsev2017infinite}; we include this in the theorem below in order to provide a complete picture. Part (i) is proved by combining results from \cite{WFSbook} with the monotonicity properties derived in \cite{sarantsev2017infinite}. The proof is in Section \ref{proofexist}.

\begin{theorem} \label{thm:exist}
Let $\gamma \in \clp_0(\RR^{\infty})$, and assume that $\bar{g} \in \ell^2(\mathbb{N})$.
\begin{enumerate}[label = (\roman*)]
\item  Suppose $p = 1 - q \in [0,1/2)$ and for some $\chi > 1/2$, $\gamma$ satisfies
\begin{equation} \label{eq:admissible}
\sum_{k = 1}^\infty \gamma(x_k \leq k^\chi) < \infty.
\end{equation}
Then there exists a strong solution to \eqref{eq:loctimeeq} with initial distribution $\gamma$.

\item (See \cite{sarantsev2017infinite}.) Suppose $p = 1 - q \in (1/2,1)$ and for some $c > 0$, $\gamma$ satisfies
\begin{equation} \label{eq:scon}
\sum_{j = 1}^\infty e^{-c(x_j \vee 0)^2} < \infty, \; \gamma\text{-a.s.}
\end{equation}
Then there exists a strong solution to \eqref{eq:loctimeeq} with initial distribution $\gamma$.
\end{enumerate}
\end{theorem}

\begin{remark} \label{plinitial} \normalfont
Although Theorem \ref{thm:unique1} on pathwise uniqueness of solutions when $p < 1/2$ imposes no restrictions on the initial distribution $\gamma$, we will see from its proof (Lemma \ref{lem:biginf2}) that for a strong solution to exist in this setting, $\gamma$ must satisfy
$$
\liminf_{M \to \infty} \frac{x_M}{\sqrt{M}} = \infty \text{ for  } \gamma \mbox{-a.e } \bar x = \{x_j\}_{j \in \NN}.
$$ 
Thus the condition \eqref{eq:admissible} in Theorem \ref{thm:exist}(i) comes close to the above necessary condition for existence.
\end{remark}

\begin{remark} \normalfont 
{Weak existence and uniqueness of solutions to the unordered system 
\eqref{RBD}, when the initial configuration satisfies the condition
$x_j \ge aj +b$, $j \in \NN$, for some $a>0$ was established in \cite{PP} (see also \cite{shkol2011levy}) and the  paper \cite{sarantsev2017infinite} relaxed this condition substantially to the requirement that, for every $c>0$, $\sum_{k=1}^{\infty}\exp\{-c (x_k)_+^2\}<\infty$. Note that by passing to the order statistics, such a weak solution produces a weak solution of \eqref{eq:loctimeeq0} with $p=1/2$. Here we are concerned with the more general setting of $p \in (0,1)$ and with wellposedness of strong sense solutions of \eqref{eq:loctimeeq0}.
Conditions (D.1) and (D.2) are weaker than the condition on the initial data in \cite{PP} noted above while stronger than the condition on initial data in \cite{sarantsev2017infinite}.}
\end{remark}

We note the following consequence of Theorems \ref{thm:unique1}, \ref{thm:unique2} and \ref{thm:unique3} on uniqueness of solutions, and Theorem \ref{thm:exist} on existence of solutions. Approximative versions mentioned in the statement of this corollary were introduced in \cite{sarantsev2017infinite}, and a precise definition of these is given in Section \ref{ssec:approximative}.

\begin{corollary}\label{cor:eusummary}
Fix parameters $p \in [0,1)$  and $\bar{g} = \{g_j\}_{j \in \mathbb{N}} \in \mathbb{R}^{\infty}$. Let $\gamma \in \clp_0(\RR^{\infty})$.
\begin{enumerate}[label = (\roman*)]
\item Suppose $p = 1 - q \in [0,1/2)$ and that for some $\chi > 1/2$, $\gamma$  satisfies \eqref{eq:admissible}.
Then there exists a strong solution to \eqref{eq:loctimeeq} with initial distribution $\gamma$, and this solution must be the approximative version solution of \eqref{eq:loctimeeq}.

\item Suppose $p = 1 - q \in [1/2,1)$ and that, for some $c > 0$, the initial distribution $\gamma$ satisfies \eqref{eq:scon}.
Then there exists a strong solution to \eqref{eq:loctimeeq} with initial distribution $\gamma$. Suppose in addition that \eqref{eq:init_cond1} holds and that the local times associated with a given strong solution satisfy \eqref{eq:cncstar}. Then this strong solution must be the approximative version solution of \eqref{eq:loctimeeq}.

\item Suppose $p = 1 - q \in (1/2,1)$,
suppose  that \eqref{eq:init_cond2} holds, and that the local times associated with the strong solution 
of \eqref{eq:loctimeeq} satisfy \eqref{eq:cncstar2}. Then this solution must be the approximative version solution of \eqref{eq:loctimeeq}.
\end{enumerate}
\end{corollary}
{\begin{proof}The first part follows from the proof of Theorem \ref{thm:exist}(i), which constructs the solution as an approximative version. The last two parts of the corollary follow from Proposition \ref{prop:lcond} which shows that when $p = 1 - q \in [1/2,1)$, under \eqref{eq:init_cond2}, the local times associated with the strong approximative solution satisfy \eqref{eq:cncstar}, and when $p = 1 - q \in (1/2,1)$,  the local times  satisfy \eqref{eq:cncstar2}.
\end{proof}}

\begin{remark}[Generalizations] \normalfont
One can consider more general versions of the system of stochastic differential equations \eqref{eq:loctimeeq}. Let $\{D_j(t), t \in [0,\infty), j \in \mathbb{N}\}$ be a collection of continuous processes, given on a filtered probability space $\{\Omega, \clf, \{\clf_t\}_{t \geq 0}, \PP\}$, and suppose $\bar{x} = \{x_j\}_{j \in \mathbb{N}}$ is a non-decreasing sequence of real-valued $\mathcal{F}_0$-measurable random variables. Consider the system of equations 
\begin{equation} \label{eq:genloctimeeq}
X_j(t) = x_j + D_j(t) + p_{j-1} L_{(j-1,j)}(t) - q_j L_{(j,j+1)}(t), \quad t \in [0,\infty), j \in \mathbb{N}.
\end{equation}
where $p_0 = 0$, and for all $j \in \mathbb{N}$, $p_j = 1 - q_j \in [0,1)$. The collision local times $\{L_{(j,j+1)}\}$ are defined in the usual way (see Definition \ref{def:cbp}\ref{pr:p3}).

We expect our methods for establishing pathwise uniqueness of solutions to extend to the system \eqref{eq:genloctimeeq} with more general driving processes $\{D_j\}$ and parameters $p_j$ and $q_j$. When the driving processes $\{D_j\}$ are still taken to be i.i.d.\ Brownian motions with drifts, as in \eqref{eq:loctimeeq}, straightforward analogues of Theorems \ref{thm:unique1}, \ref{thm:unique2}, and \ref{thm:unique3} should hold under the respective assumptions that (a) $\limsup_{j \to \infty} p_j \in [0,1/2)$, (b) $\liminf_{j \to \infty} p_j \in [1/2,1)$, and (c) $\liminf_{j \to \infty} p_j \in (1/2,1)$, with minimal modifications to the proofs. We have chosen to not consider these more general cases to keep the notation and proofs cleaner.

Many of our arguments do not depend on the precise distribution of the collection of driving processes $\{D_j\}$ at all. In particular, the estimates from \S\ref{ssec:ptwise} hold for any collection of continuous functions $\{D_j\}$, including deterministic ones. On the other hand, pathwise uniqueness proofs require the moderate and large deviation bounds for the quantities from Brownian last-passage percolation given in Lemma \ref{lem:lpp}, which should hopefully extend to other random driving processes possessing analogous deviation properties.

Similarly, Theorem \ref{thm:exist}(i) on strong existence of solutions should extend straightforwardly to the system \eqref{eq:genloctimeeq}, for any choice of parameters $p_j = 1-q_j \in [0,1)$, $j \in \mathbb{N}$, provided that the driving diffusions are still i.i.d.\ Brownian motions with drifts. However, obtaining existence results for more general driving processes will require additional work since our proofs use \cite[\S 2]{WFSbook} which takes the driving processes to be i.i.d.\ Brownian motions.

\end{remark}

\begin{remark} 
\label{SSrem}
\normalfont
The paper \cite{SS} considers equation \eqref{eq:loctimeeq0} for the case where $\bar g$ is the zero vector  with  a half-Poisson initial data, namely the case where $\{x_j - x_{j-1}, j \in \mathbb{N}\}$ (with $x_0=0$) is an i.i.d.\ sequence of mean $1$ Exponential random variables. Under the assumption that \eqref{eq:loctimeeq0} has a unique solution, their work studies the collection of random variables $N(u,t) := \sum_{i=1}^{\infty} \mathbf{1}\{X_i(t) \le u\}$, giving the number of particles below the level $u$ at the time instant $t$. The paper \cite{SS} gives contour integration formulas for the multipoint generating functions of the collection $\{N(x,t), x \in \mathbb{R}\}$ for each fixed $t$, and based on this they compute a single point Fredholm determinant which is then used to analyze the long time asymptotics of $N(u,t)$, making connections with the KPZ universality class. We note that the half-Poisson initial data (and $\bar g= 0$) satisfies assumptions of Theorem \ref{thm:exist}(i). Thus for the model considered in \cite{SS}, Theorem \ref{thm:exist} gives strong existence of solutions and Theorems \ref{thm:unique1} - \ref{thm:unique3} provide pathwise uniqueness of solutions for this model and together these results give a proof of wellposedness of the equation studied in \cite{SS}.
\end{remark}
\section{Finite propagation of collision events}
\label{sec:finprop}
In this section, we present some key lemmas for proving Theorems \ref{thm:unique1}, \ref{thm:unique2}, and \ref{thm:unique3}. Proofs of these lemmas are given in Section \ref{sec:pfssec2}.

The first result gives a certain ``translation invariance'' property 
of solutions to \eqref{eq:loctimeeq}. Let $C([0,\infty),\mathbb{R})$ denote the space of continuous functions from $[0,\infty)$ into $\mathbb{R}$, and for $t \in [0,\infty)$, define maps $\utr_t, \ntr_t : C([0,\infty),\mathbb{R}) \to C([0,\infty),\mathbb{R})$ by 
\begin{equation} \label{eq:translate}
\utr_t\omega(s) = \omega(t + s), \quad \ntr_t\omega(s) = \omega(s + t) - \omega(t), \quad s \in [0,\infty).
\end{equation}
In particular, for a stochastic process $\{Y(t)\}$ with sample paths in $C([0,\infty),\mathbb{R})$, we write
$$\utr_tY(s) = Y(t + s), \quad \ntr_tY(s) = Y(s + t) - Y(t), \quad s, t \in [0,\infty).$$
Also, for $s,t \geq 0$, let $\theta_t \mathcal{F}_s = \mathcal{F}_{s + t}$, where $\{\clf_t\}$ is a filtration given on some probability space.

\begin{lemma} \label{lem:translate}
Fix parameters $p \in [0,1)$  and $\bar{g} = \{g_j\}_{j \in \mathbb{N}} \in \mathbb{R}^{\infty}$. Fix $N \in \mathbb{N} \cup \{\infty\}$, and let $\{X_j\}_{j \in [N]}$ be a system of competing Brownian particles with parameters $p, \bar{g}$, initial conditions $\bar{x}$, and driving Brownian motions $\{B_j\}_{j \in \mathbb{N}}$, given on some filtered probability space $(\Omega, \clf, \{\clf_t\}_{t\ge 0}, \PP)$. Let $\{L_{(j,j+1)}\}$ be the associated collection of local times.
Suppose $\tau$ is a finite $\{\mathcal{F}_t\}$-stopping time.  Then $\{\utr_\tau X_j\}_{j \in [N]}$ is a solution to \eqref{eq:loctimeeq}, on the  filtered probability space $(\Omega, \clf, \{\theta_{\tau}\clf_t\}_{t\ge 0}, \PP)$, with parameters $p,\bar g$, initial conditions $\{X_j(\tau)\}_{j \in [N]}$, and driving Brownian motion $\{\ntr_\tau B_j\}_{j \in [N]}$; with the associated collection of local times  given by $\{\ntr_\tau L_{(j,j+1)}\}$.
\end{lemma}

The next result gives conditions under which the first few coordinates of a solution of an infinite system of competing Brownian motions agree with the corresponding coordinates of the solution of a finite system. Fix a strong solution $\{X_j\}$ of \eqref{eq:loctimeeq} with some choice of parameters $p, \bar g$ and initial distribution $\gamma$.
For $i \in \mathbb{N}$ and $0 \leq u < v < \infty$, recall the quantity
\[
\begin{split}
K^*(i,[u,v]) = \sup\{k \geq 0 & : \exists u \leq s_{i + k - 1} \leq \cdots \leq s_{i + 1} \leq s_i \leq v \\
& \text{ such that, for } i \leq j \leq i + k - 1, X_{j+1}(s_j) = X_j(s_j)\}.
\end{split}
\]
To emphasize the dependence on the choice of the solution of \eqref{eq:loctimeeq}, we will occasionally denote this quantity as $K^*(i,[u,v], \{X_j\})$. To keep notation simpler, for $i \in \mathbb{N}$ and $T \in (0,\infty)$, we will also write
\[
K^*(i,T) = K^*(i,T,\{X_j\}) := K^*(i,[0,T], \{X_j\}).
\]
The main use we have for the above quantity is the following lemma.

\begin{lemma} \label{lem:finK-unique}  
Fix parameters $p \in [0,1)$  and $\bar{g} = \{g_j\}_{j \in \mathbb{N}} \in \mathbb{R}^{\infty}$. Let $\gamma \in \clp_0(\RR^{\infty})$. Let
$\{X_j\}$ be a strong solution to \eqref{eq:loctimeeq} with initial distribution $\gamma$.
Denote the driving Brownian motions by $\{B_j\}$ and let $x_j = X_j(0)$, $j \in \NN$. For $N \in \mathbb{N}$, let $$\{X_j^N(t), t \in [0,\infty), 1 \leq j \leq N\}$$ denote the unique (pathwise) solution to the system of $N$ competing Brownian motions with parameters $p,\bar{g}^{(N)} := (g_1, \dots, g_N)$, initial data $\{x_j\}_{1 \leq j \leq N}$, and driving Brownian motions $\{B_j\}_{1 \leq j \leq N}$. Then, for $i,M \in \mathbb{N}$ and $T \in (0,\infty)$, on the event $K^*(i,T, \{X_j\}) \leq M$,
$$
X_j(t) = X_j^{i + M}(t), \text{ for all $1 \le j \leq i$ and $t \in [0,T]$}. 
$$ 
\end{lemma}

The lemmas above give us the following criterion for pathwise uniqueness of solutions to \eqref{eq:loctimeeq}. 

\begin{lemma} \label{lem:unique}
Fix parameters $p,\bar{g}$ and let $\{X_j^{(1)}\}$, $\{X_j^{(2)}\}$ be two solutions to \eqref{eq:loctimeeq} with a common initial condition $\bar x$ and driving Brownian motions $\{B_j\}$ given on some filtered probability space $(\Omega, \clf, \{\clf_t\}_{t\ge 0}, \PP)$. Suppose that, either

\begin{enumerate}[label = (\alph*)]
\item \label{h:a} for all $i \in \mathbb{N}$ and $T \in (0,\infty)$, $K^*(i,T,\{X^{(l)}_j\}) < \infty$ a.s., for $l=1,2$, or
\item \label{h:b} there exists a strictly positive 
$\clf_0$-measurable random variable $\Delta$ such that, for all $i \in \mathbb{N}$ and $T \in [0,\infty)$, $K^*(i,[T,T+\Delta],\{X^{(l)}_j\}) < \infty$ a.s. for $l=1,2$.
\end{enumerate}
Then, a.s., $X_j^{(1)}(t) = X_j^{(2)}(t)$ for all $t\ge 0$ and $j \in \NN$.
\end{lemma}

\subsection{Proofs}
\label{sec:pfssec2}
In this section we give the proofs of Lemmas \ref{lem:translate}-\ref{lem:unique}.

\begin{proof}[Proof of Lemma \ref{lem:translate}]
With notation as in the statement of the lemma, let $\tilde{X}_j = \utr_\tau X_j$ and $\tilde{B}_j = \ntr_\tau B_j$, for $j \in \mathbb{N}$, and set $\tilde{L}_{(j,j+1)} = \ntr_\tau L_{(j,j+1)}$ for $j \in \{0\} \cup \mathbb{N}$. It is easy to check that $\{\tilde{X}_j\}_{j \in \mathbb{N}}$ and $\{\tilde{L}_{(j,j+1)}\}_{j \in \{0\} \cup \mathbb{N}}$ are collections of $\{\theta_{\tau}\clf_t\}$-adapted, continuous processes. {Furthermore, for $t\ge 0$ and $j \in \NN$,
\begin{align*}
\tilde{X}_j(t) &= X_j(t+\tau) = x_j + (t+\tau)g_j + B_j(t+\tau) + p L_{(j-1,j)}(t+\tau)
- q L_{(j, j+1)}(t+\tau)\\
&= x_j+  g_j \tau + B_j(\tau) + g_j t + \tilde B_j(t) + p L_{(j-1,j)}(\tau)
- q L_{(j, j+1)}(\tau) + p \tilde L_{(j-1,j)}(t)
- q \tilde L_{(j, j+1)}(t)\\
&= X_j(\tau) + \tilde B_j(t) + p \tilde L_{(j-1,j)}(t)
- q \tilde L_{(j, j+1)}(t).
\end{align*}
Thus $\{\tilde X_j\}_{j \in \NN},  \{\tilde L_{(j, j+1)}\}_{j \in \{0\} \cup \mathbb{N}}$ satisfy P1 in Definition \ref{def:cbp} with
$x_j$ replaced with $X_j(\tau)$ and $B_j$ replaced with $\tilde B_j$. Also, since $\{X_j\}$ satisfy P2, so do $\{\tilde X_j\}$.
Finally, since P3 is satisfied for $\{X_j\}_{j \in \NN},  \{ L_{(j, j+1)}\}_{j \in \NN_0}$, we see that
$$\int_0^\infty (\tilde X_{j+1}(t) - \tilde X_j(t)) \dd \tilde L_{(j,j+1)}(t) = 
\int_{\tau}^\infty (\tilde X_{j+1}(t) - \tilde X_j(t)) \dd \tilde L_{(j,j+1)}(t) = 0.
$$
Thus $\{\tilde X_j\}_{j \in \NN},  \{\tilde L_{(j, j+1)}\}_{j \in \{0\} \cup \mathbb{N}}$ satisfy P3 in Definition \ref{def:cbp}  as well. The result follows.}
\end{proof}

For a number of subsequent arguments, including for the proof of Lemma \ref{lem:finK-unique}, we will make use of certain comparison results which we now present.

Fix parameters $p,\bar{g}$ as before and let a filtered probability space, Brownian motions $\{B_j\}$ and sequence $\bar x = \{x_j\}$ be as at the beginning of Section \ref{sec:sepu}.
Let $\{X_j(t), j \in \mathbb{N}, t \in [0,\infty)\}$ be a solution to \eqref{eq:loctimeeq}, with $N = \infty$, with driving noises $\{B_j\}$ and initial condition $\bar x$. Associated local times are denoted as $\{L_{(j,j+1)}\}$.
For each $M \in \mathbb{N}$, let $\{X_j^{M}(t), t \in [0,\infty), 1 \leq j \leq M\}$ be the unique solution to \eqref{eq:loctimeeq} for the case $N = M$, once more with driving noises
$\{B_j\}$, and initial condition $x^M = \{x_j\}_{1 \leq j \leq M}$. The associated local times  are denoted as $\{L^{M}_{(j,j+1)}(t), t \in [0,\infty), 0 \leq j \leq M\}$, where as usual $L^{M}_{(0,1)} = L^{M}_{(M,M+1)} = 0$. 
In addition, let $\{\tilde{X}_j^M(t), t \in [0,\infty), 1 \leq j \leq M\}$ denote the solution to \eqref{eq:loctimeeq}, with $N=M$, the same parameters and driving Brownian motions as the system $\{X_j^M\}$, but with {\em packed initial conditions} $\tilde x^M = \{\tilde x_j\}_{1\le j \le M}$, where  $\tilde{x}_j = 0$ for $1 \leq j \leq M$. We denote the associated collection of local times by $\{\tilde{L}^{M}_{(j,j+1)}(t), t \in [0,\infty), 0 \leq j \leq M\}$. The proof of the following lemma 
is obtained, in part, as a consequence of results from \cite{AS2}.
\begin{lemma} \label{lem:bd_by_fin}
Let $M,M' \in \mathbb{N}$ with $M' \leq M$. With notation as introduced above, for all $t \in [0,\infty)$, the following bounds hold: 
\begin{enumerate}[label = (\roman*)]
\item \label{it:bf1} $X_j^M(t) \leq X_j^{M'}(t)$ for $1 \leq j \leq M'$. 
\item \label{it:bf2} $X_j(t) \leq X_j^M(t)$ for $1 \leq j \leq M$.
\item \label{it:bf3} $x_1 + \tilde{X}_i^M(t) \leq X_j^M(t) 
\leq 
x_M + \tilde{X}_k^M(t) \text{ for } 1 \leq i \leq j \leq k \leq M$.
\end{enumerate}
\end{lemma}

\begin{proof}
We will prove part \ref{it:bf2} and \ref{it:bf3}. We omit the proof of \ref{it:bf1} which is almost identical to the proof of \ref{it:bf2}, with that the processes $\{X_j^M\}, \{L_{(j,j+1)}^M\}$ playing the roles of $\{X_j\}, \{L_{(j,j+1)}\}$, respectively, and the processes $\{X_j^{M'}\}, \{L_{(j,j+1)}^{M'}\}$ playing the roles of $\{X_j^M\}, \{L_{(j,j+1)}^M\}$, respectively.

\textit{Proof of \ref{it:bf2}.} If $p = 0$, then for $j \in \mathbb{N}$, $X_j$ and $X_{j + 1}$ are related via the Skorokhod reflection representation (see e.g.\ \cite[Lemma 6.14]{KSbook}), 
\begin{equation} \label{eq:skor1}
X_j(t) = x_j + V_j(t) + \inf_{0 \leq s \leq t} \left(X_{j+1}(s) - x_j - V_j(s) \right) \wedge 0,
\end{equation}
where $V_j$ are as in \eqref{eq:Vjdef}.
Similarly, $X^M_M(t) = x_M + V_M(t)$, and for $1 \leq j \leq M - 1$,
\begin{equation} \label{eq:skor2}
X_j^M(t) = x_j + V_j(t) + \inf_{0 \leq s \leq t} \left(X_{j+1}^M(s) - x_j - V_j(s) \right) \wedge 0.
\end{equation}
We see that $X_M(t) \leq x_M + V_M(t) = X_M^M(t)$, and we obtain \ref{it:bf2} by a backward induction on $j$ (starting with $j=M$), comparing the formulas \eqref{eq:skor1} and \eqref{eq:skor2}. 

If $0 < p < 1$, we proceed as follows. We note that the first $M$ particle positions $\{X_j, 1 \leq j \leq M\}$ satisfy the system of equations 
$$
X_j(t) = U_j(t) + p\hat{L}_{(j - 1,j)}(t) - q\hat{L}_{(j,j+1)}(t), \text{ for $1 \leq j \leq M$, $t \geq 0$},
$$
where 
$$
U_j(t) := \begin{cases}
x_j + V_j(t) & \text{ if } 1 \leq j \leq M-1, \\
x_M + V_M(t) - qL_{(M,M+1)}(t) & \text{ if } j = M,
\end{cases}
$$
and 
$$
\hat{L}_{(j,j+1)}(t) = \begin{cases}
L_{(j,j+1)}(t) & \text{ if } 0 \leq j \leq M-1 \\
0 & \text{ if } j = M.
\end{cases}
$$
We see that $\{X_j, 1 \leq j \leq M\}$ is a system of $M$ competing particles with driving function $\overline{U} := (U_1, U_2, \dots, U_M)$ and parameters of collision $p,q$, in the sense of Definition 2 in \cite{AS2}. Observe that for each $1\le j\le M$, $U_j(0) = x_j = x_j + V_j(0)$. Also, since ${L}_{(M,M+1)}$ is non-decreasing, $U_j(t) - U_j(s) \leq V_j(t) - V_j(s)$ for $0 \leq s \leq t$. Hence, by \cite[Theorem 3.2]{AS2}, $X_j(t) \leq X_j^M(t)$ for $1 \leq j \leq M$ and $t \geq 0$, as desired.

\textit{Proof of \ref{it:bf3}.} If $p = 0$, we proceed by proving the following claim. The result will follow by monotonicity of the sequences  $\{x_j\}$ and $\{\tilde{X}_j^M(t)\}$. 
\vspace{0.2in}

\textbf{Claim.} $x_j + \tilde{X}_j^{M}(t) \leq X_j^M(t) \leq x_M + \tilde X_j^M(t)$, \text{ for } $1 \leq j \leq M$.

\vspace{0.2in}

We prove the claim by backwards induction on $j$. The base case $j=M$ is clear since
$$
X_M^M(t) = x_M + V_M(t) = x_M + \tilde{X}_M^M(t).
$$
To carry out the induction step $j+1 \rightarrow j$, by the Skorokhod representation of the trajectories, we can write
\begin{equation} \label{eq:packed_skor}
\begin{split}
\tilde{X}_j^M(t) & = V_j(t) + \inf_{0 \leq s \leq t} (\tilde X^M_{j+1}(s) - V_j(s)) \wedge 0 \\
& = V_j(t) + \inf_{0 \leq s \leq t} (\tilde X^M_{j+1}(s) - V_j(s)), \quad \text{ for } 1 \leq j \leq M-1.
\end{split}
\end{equation}
Note that we can remove the ``$\wedge 0$'' from the quantity inside the infimum since all of the trajectories $\{\tilde X_j\}$ and $V_j$ start from zero. By the induction assumption and \eqref{eq:skor2}, the quantity \eqref{eq:packed_skor} is bounded above by 
\[
\begin{split}
& V_j(t) + \inf_{0 \leq s \leq t} \left(X_{j+1}^{M}(s) - x_{j+1} - V_j(s) \right) \wedge 0 \\
& \leq V_j(t) + \inf_{0 \leq s \leq t} \left(X_{j+1}^{M}(s) - x_j - V_j(s) \right) \wedge 0 \\
& = X_j^M(t) - x_j
\end{split}
\]
and is bounded below by 
\[
\begin{split}
& V_j(t) + \inf_{0 \leq s \leq t} ( X_{j+1}^M(s) - x_M - V_j(s)) \\
& = x_j + V_j(t) + \inf_{0 \leq s \leq t} ( X_{j+1}^M(s) - x_j - V_j(s)) - x_M \\
& \geq x_j + V_j(t) + \inf_{0 \leq s \leq t} ( X_{j+1}^M(s) - x_j - V_j(s)) \wedge 0 - x_M \\
& = X_j^M(t) - x_M.
\end{split}
\]
Together these two bounds complete the induction step and prove the claim.

If $0 < p < 1$, then \ref{it:bf3} follows from \cite[Theorem 3.2]{AS2}. To see why, note that if we set $U_j^1(t) = x_1 + \tilde{X}_j^M(t)$, $U_j^2(t) = X_j^M(t)$, and $U_j^3(t) = x_M + \tilde{X}_j^M(t)$, then each of the processes $\{U_j^k\}_{1 \leq j \leq M}$, for $1 \leq k \leq 3$, are systems of $M$ competing Brownian particles with the same parameters and driving Brownian motions, but different initial data, satisfying
$$
U_j^1(0) = x_1 \leq U_j^2(0) = x_j \leq U_j^3(0) = x_M, \; 1 \le j \le M.
$$
Hence, the hypotheses of the \cite[Theorem 3.2]{AS2} apply, and we obtain for $1 \leq i \leq j \leq k \leq M$,
$$
x_1 + \tilde{X}_i^M(t) \leq x_1 + \tilde{X}_j^M(t) \leq X_j^M(t) \leq x_M + \tilde{X}_j^M(t) \leq x_M + \tilde{X}_k^M(t).
$$
Here we use the monotonicity of the sequence $\{\tilde{X}_j(t)\}$ to obtain the first and last inequalities. This completes the proof of \ref{it:bf3}.
\end{proof}

\begin{proof}[Proof of Lemma \ref{lem:finK-unique}]
Fix $i,M \in \mathbb{N}$ and define a sequence of stopping times $\tau_{i + M+1} \leq \tau_{i + M} \leq \tau_{i + M - 1} \leq \cdots \leq \tau_1$ as follows. Let $\tau_{i + M+1} = 0$ and for $1 \leq k \leq i + M$, let 
$$
\tau_k = \inf\{t \geq \tau_{k + 1} : X_{k+1}(t) - X_k(t)=0\}.
$$
Observe that if $K^*(i,T) \leq M$, then $\tau_{i} > T$. This is clear since otherwise the sequence $s_k = \tau_{k}$, $i \leq k \leq i + M$,
would violate the definition of $K^*(i,T)$. Thus, we will be done if we can prove the following claim:

\vspace{0.2in}

\textbf{Claim.} For $1 \leq k \leq i + M$, 
$$
X_j(t) = X_j^{i+M}(t), \text{ for all $j \leq k$ and $t \in [\tau_{k+1}, \tau_{k})$, almost surely.}
$$

\vspace{0.2in}

The proof of the claim proceeds by backwards induction on $k$. 

\textit{Base case: $k = i + M$.} By definition of $\tau_{i + M}$, the paths $X_{i+M+1}$ and $X_{i+M}$ do not collide at times $t \in [0,\tau_{i+M})$ and consequently $L_{(i+M,i+M+1)}(t) = 0$ for all $t \in [0,\tau_{i+M})$. From this, it follows that the truncated system $\{X_j\}_{1 \leq j \leq i + M}$, with the associated local times $\{L_{(j,j+1)}, 0 \leq j \leq i + M\}$, satisfies \eqref{eq:loctimeeq} with $N = i + M$ on the interval $[0,\tau_{i + M})$, and the base case then follows by pathwise uniqueness of solutions to the finite-dimensional system (namely the $i+M$ competing particle system). See Remark \ref{rem:veryunique}.

\textit{Induction step: $k \rightarrow k-1$.} Suppose the claim has been established for some $2 \leq k \leq i + M$. By continuity of the solutions, 
\[
X_j(\tau_k) = X_j^{i + M}(\tau_k), \text{ for all $j \leq k$}.
\]
Let $\{\widetilde{X}_j^{k-1}\}_{1 \leq j \leq k-1}$ denote the system of $k-1$ competing Brownian particles with parameters $p,\{g_j\}_{1 \leq j \leq k-1}$, initial conditions $\{X_j(\tau_k)\}_{1 \leq j \leq k-1}$, and driving Brownian motions $\{\ntr_{\tau_k} B_j\}_{1 \leq j \leq k-1}$. Recall Lemma \ref{lem:translate}, which says that $\{\utr_{\tau_k} X_j\}_{j \in \mathbb{N}}$ is a solution to \eqref{eq:loctimeeq} with $N=\infty$, and  parameters $p,\{g_j\}_{j \in \mathbb{N}}$, initial condition  $\{X_j(\tau_k)\}_{j \in \NN}$ and driving Brownian motions $\{\ntr_{\tau_k} B_j\}_{j \in \mathbb{N}}$. The associated collection of local times for this solution is $\{\ntr_{\tau_k} L_{(j-1,j)}\}_{j \in \mathbb{N}}$. Moreover, by definition of $\tau_{k-1}$,
$$
\ntr_{\tau_k} L_{(k-1,k)}(t) = L_{(k-1,k)}(t + \tau_k) - L_{(k-1,k)}(\tau_k) = 0, \text{ for all } t \in [0, \tau_{k-1} - \tau_k).
$$
Hence, by the same reasoning as in the argument for the base case, for all $t \in [0,\tau_{k-1},\tau_k)$ and $1 \leq j \leq k-1$,
\begin{equation} \label{eq:eqstepk}
\utr_{\tau_k}X_j(t) = \widetilde{X}_j^{k-1}(t).
\end{equation}
On the other hand, from Lemma \ref{lem:translate} and Lemma \ref{lem:bd_by_fin}, we obtain
\begin{equation} \label{eq:bdstepk}
\utr_{\tau_k}X_j(t) \leq \utr_{\tau_k}X_j^{i + M}(t) \leq \widetilde{X}_j^{k-1}(t),
\end{equation}
for all $t \in [0,\tau_{k-1}-\tau_k)$ and $1 \leq j \leq k-1$. We conclude from \eqref{eq:eqstepk} that the inequalities in \eqref{eq:bdstepk} are equalities, and hence, for all $t \in [\tau_k, \tau_{k-1})$ and $1 \leq j \leq k-1$,
\[
X_j(t) = \utr_{\tau_k} X_j(t - \tau_{k}) = \utr_{\tau_k} X_j^{i+M}(t - \tau_{k}) = X_j^{i+M}(t).
\]
This completes the induction step, so the claim and the lemma are proved.
\end{proof}

\begin{proof}[Proof of Lemma \ref{lem:unique}]
Let $\{X_j^{(l)}\}$, $l=1,2$, be as in the statement of the lemma.
By Lemma \ref{lem:finK-unique}, for any $i \in \mathbb{N}$, $T \in (0,\infty)$, and $M\in (0, \infty)$, on the set $A_M = \{M \ge \max\{K^*(i,T, \{X_j^{(l)}\}, l= 1,2\}\}$,
 $\{X_j^{(l)}(t), t \in [0,T], 1 \leq j \leq i\}$ for $l=1,2$, is the same as
 $\{X_j^{i+M}(t), t \in [0,T], 1 \leq j \leq i\}$, which is the unique solution to the system of $M+i$ competing Brownian  particles with parameters $p,\bar{g}^{(N)} := (g_1, \dots, g_N)$, initial conditions $\{x_j\}$, and driving Brownian motions $\{B_j\}$. 
 This says that  $X_j^{(1)}(t) = X_j^{(2)}(t)$ for all $t\le  T$ and $1 \le j \le i$, on $A_M$. Since under \ref{h:a} $\PP(\cup_{M \in \NN} A_M) = 1$, we have $X_j^{(1)}(t) = X_j^{(2)}(t)$ a.s. for all $t\le T$ and $1\le j \le i$. Since $T\ge 0$ and $i \in \NN$ are arbitrary, this proves the first part of the lemma.

Suppose now that (b) holds. By a conditioning argument we assume without loss of generality that $\Delta$ is nonrandom. Then, considering $T=0$ in (b), by the same argument as used for (a), we see that
$X_j^{(1)}(t) = X_j^{(2)}(t)$ for all $t \in [0,\Delta]$ and $j \in \NN$. 
By Lemma \ref{lem:translate}, the system $\{\utr_{\Delta} X^{(l)}_j\}$ is also a solution to \eqref{eq:loctimeeq} (with respect to the initial conditions $\{X^{(1)}_j(\Delta)\} = \{X^{(2)}_j(\Delta)\}$ and the Brownian motions $\{\ntr_{\Delta} B_j\}$). Also, the condition in (b) says that
$K^*(i,[0,\Delta],\{\utr_{\Delta}X^{(l)}_j\}) < \infty$ a.s. for $l=1,2$. Applying the argument for part (a) again, we now see that 
$$X_j^{(1)}(t+\Delta) = \utr_{\Delta}X_j^{(1)}(t) = \utr_{\Delta}X_j^{(2)}(t) = X_j^{(2)}(t+\Delta) \mbox{ for all } t \in [0,\Delta] \mbox{ and } j \in \NN.$$ The proof is now completed by a recursive argument.
\end{proof}

\section{Pointwise estimates}
\label{ssec:ptwise}
{In Section \ref{sec:finprop}, pathwise uniqueness was reduced to proving that the length
$K^*(i,T)$ of collision chains reaching a fixed particle is almost surely
finite. The purpose of this section is to develop the deterministic,
pathwise estimates that will later be used to prove this finiteness. These
estimates convert information about possible chains of collisions into
inequalities involving the initial configuration and certain Brownian
last-passage percolation quantities.}

{The section is organized as follows. In Section \ref{sec:ptw1} we first study finite
systems started from packed initial conditions. Lemmas \ref{lem:lppcontrol} and \ref{lem:pkdubs} give
upper and lower bounds for such systems in terms of the quantities
$\cV^-_M$, $\cV^+_M$, and $\cW_M$, which are variants
of Brownian last-passage percolation functionals, introduced below. In Section \ref{sec:ptw2} we use the comparison estimates from Section \ref{sec:finprop} (see Lemma \ref{lem:bd_by_fin})
to transfer the finite-system bounds to the infinite particle system.
The resulting estimates, collected in Lemma \ref{lem:Xbds_plessq}, are the main pathwise
input for the probabilistic arguments of Sections \ref{ssec:lpp} and \ref{sec:pfthm2.6}.}
We remark that the estimates collected in Lemmas \ref{lem:lppcontrol}-\ref{lem:Xbds_plessq} (as well as Lemma \ref{lem:bd_by_fin}) hold 
with $\{B_j\}$  replaced with any collection of real continuous functions on $[0,\infty)$.

\subsection{Finite system with packed initial conditions}\label{sec:ptw1}

In this section we establish some estimates for the finite competing Brownian particle system 
 started from packed initial conditions. Recall from Section \ref{sec:pfssec2}, for $M \in \NN$, the collection $\{\tilde X^M_j, 1 \le j \le M\}$, given as the unique solution  to \eqref{eq:loctimeeq}, with $N=M$,  parameters $p$ and $\bar g$ and driving Brownian motions $\{B_j\}$,  with {\em packed initial conditions} $\tilde x^M = \{\tilde x_j\}_{1\le j \le M}$, where  $\tilde{x}_j = 0$ for $1 \leq j \leq M$. The associated local times are denoted as $\{\tilde L^M_{(j, j+1)}, 0 \le j\le M\}$. Also, note from Remark \ref{rem:veryunique} that such a system is well-defined for any $p \in [0,1]$.

 The following lemmas give us bounds on the particle trajectories $\{\tilde{X}^M_j\}$ in terms of \emph{Brownian last-passage percolation} quantities. These bounds, in conjunction with the estimates in \S\ref{ssec:lpp}, will be used crucially in the proofs of our main results. For $i,M \in \mathbb{N}$ and $T \in [0,\infty)$, define the quantities 
\begin{equation} \label{eq:VMdef}
\cV_M^-(i,T) = \inf_{0 \leq s_{i + M - 2} \leq \cdots \leq s_i \leq T} \sum_{j = i}^{i+M-1} (V_j(s_{j-1}) - V_j(s_j)),
\end{equation}
and 
\begin{equation} \label{eq:VMplusdef}
\cV_M^+(i,T) = \sup_{0 \leq t_i \leq \cdots \leq t_{i + M - 2} \leq T} \sum_{j = i}^{i+M-1} (V_j(t_{j}) - V_j(t_{j-1})),
\end{equation}
where by convention $s_{i + M - 1} = t_{i-1} = 0$ and $s_{i-1} = t_{i + M - 1} = T$ in the sum above. Here the $V_j$ are defined as in \eqref{eq:Vjdef}. Notice that in \eqref{eq:VMdef}, the sequence of times $\{s_j\}$ is descending, while in \eqref{eq:VMplusdef}, the sequence of times $\{t_j\}$ is ascending.

\begin{remark} \label{rem:maxeig} \normalfont
In the special case where $\bar{g} = \bar{0}$ (equivalently $V_j = B_j$ for all $j \in \mathbb{N}$), the quantity \eqref{eq:VMplusdef} is precisely the last-passage time from the point $(0,i)$ to the point $(T,i+M-1)$ in the space $[0,\infty) \times \mathbb{N}$ in Brownian last-passage percolation. 
{Brownian last-passage percolation is a continuous analogue of directed last-passage percolation models in which one considers collections of independent Brownian motions and defines passage times by maximizing accumulated increments over directed paths. The model plays a central role in the KPZ universality class and has close connections to random matrix theory, interacting particle systems, queueing theory, and stochastic growth processes.}
It is well known that \eqref{eq:VMplusdef} is equal in distribution to the largest eigenvalue of a random Hermitian matrix with independent (up to symmetry) complex $N(0,T)$ entries. This fact was originally obtained by Gravner, Tracy and Widom in \cite{GTWLimitThms2001} and Baryshnikov in \cite{bary2001}. Subsequently, a number of other proofs of this fact have been obtained. For additional discussion of this connection with Brownian last-passage percolation, see \cite[\S2.2.1-2.2.2]{WFSbook}.
\end{remark}

\begin{lemma} \label{lem:lppcontrol}
Let $p \in [0,1]$. For any $M \in \mathbb{N}$, $1 \leq i \leq M$, and $T \in [0,\infty)$, the following bounds hold:
\[
\cV_{M-i+1}^-(i,T) \leq \tilde{X}_i^M(T) \leq \cV_i^+(1,T).
\]
\end{lemma}

\begin{proof}
By symmetry, it is enough to just prove the lower bound. (The upper bound may be obtained by interchanging the roles of $p$ and $q$, and replacing $B_j$ with $-B_j$ and $g_j$ with $-g_j$ for $1 \leq j \leq M$.) To this end, for $1 \leq j \leq M$ and $t \in [0,\infty)$, we write 
\[
\tilde{X}_j^M(t) = U_j(t) - \hat{L}_{(j,j+1)}(t),
\]
where 
\[
U_j(t) = B_j(t) + g_j t + p \tilde{L}^M_{(j-1,j)}(t),
\]
and 
\[
\hat{L}_{(j,j+1)}(t) = q \tilde{L}^M_{(j,j+1)}(t).
\]
Since $\hat{L}_{(j,j+1)}$ is continuous, non-decreasing, and can only increase when $\tilde{X}_j^M = \tilde{X}_{j+1}^M$, we see that for each $1 \leq j \leq M-1$, $\tilde{X}_j^M$ is the Skorokhod reflection of the trajectory $U_j$ downwards from $\tilde{X}_{j+1}^M$. Thus the following recursive relationship holds:
\begin{equation} \label{eq:recurs}
\begin{split}
& \tilde{X}_M^M(t) = U_M(t), \\
& \tilde{X}_j^M(t) = \inf_{0 \leq s \leq t} (\tilde{X}_{j+1}^M(s) + U_j(t) - U_j(s)), \quad 1 \leq j \leq M-1.
\end{split}
\end{equation}
Here we have used the fact that, since $\tilde{X}_{j+1}^M(0)-U_j(0)=0$, 
$$\inf_{0\le s \le t} ( \tilde{X}_{j+1}^M(s)-U_j(s))\wedge 0=\inf_{0\le s \le t} ( \tilde{X}_{j+1}^M(s)-U_j(s)).$$
Iterating \eqref{eq:recurs}, we obtain that, for $t\ge 0$, 
\begin{equation} \label{eq:var1}
\tilde{X}_i^M(t) = \inf_{0 \leq s_{M - 1} \leq \cdots \leq s_i \leq t} \sum_{j = i}^{M} U_j(s_{j-1}) - U_j(s_j),
\end{equation}
where $s_M := 0$ and $s_{i-1} := t$. Observe that for each $j$, and $0\le s_j \le s_{j-1}$,
\begin{equation} \label{eq:wandwo}
\begin{split}
U_j(s_{j-1}) - U_j(s_j) & = V_j(s_{j-1}) - V_j(s_j) + p\left( \tilde{L}_{(j-1,j)}^M(s_{j-1}) - \tilde{L}_{(j-1,j)}^M(s_j) \right) \\
& \geq V_j(s_{j-1}) - V_j(s_j).
\end{split}
\end{equation}
Using \eqref{eq:wandwo} to lower bound \eqref{eq:var1}, we obtain the result.
\end{proof}

The next lemma gives some more precise upper bounds for particles in the packed system.
Let $p \in [0,1)$ and define $r:= p/q$.
Our main use for the result will be for the case when $p \leq q$, although bounds hold for all $p \in [0,1)$. For $1 \leq k < \infty$, let 
\begin{equation} \label{eq:altBMdef}
W_k(t) = \sum_{j = 1}^k r^{k - j}V_j(t), \quad t \in [0,\infty).
\end{equation}
For $i, M \in \mathbb{N}$ and $T \in [0,\infty)$, define 
\begin{equation} \label{eq:WMdef}
\cW_M(i,T) = \inf_{0 \leq s_{i + M - 2} \leq \cdots \leq s_i \leq T} \sum_{k = i}^{i+M-1} (W_k(s_{k-1}) - W_k(s_k)),
\end{equation}
where by convention $s_{i + M - 1} = 0$ and $s_{i-1} = T$ in the sum above. 

\begin{lemma} \label{lem:pkdubs}
For $k \in \mathbb{N} \cup \{0\}$, let 
\[
\alpha_k := \begin{cases}
\frac{1 - r^k}{1 - r} & \text{ if } p \neq q, \\
k & \text{ if } p = q.
\end{cases}
\]
For all $M \in \mathbb{N}$ and $T \in [0,\infty)$, the following bounds hold:
\begin{enumerate}[label = (\roman*)]
\item \label{pkdubs-i}
\begin{equation} \label{eq:lowub}
\inf_{0 \leq s \leq T} \tilde{X}_1^M(s) \leq \alpha_M^{-1}\cW_M(1,T).
\end{equation}

\item \label{pkdubs-ii}
\begin{equation} \label{eq:highub}
\tilde{X}_M^M(T) \leq W_M(T) - \sum_{j = 1}^{M-1} r^{M - j} \cV^-_{M - j + 1}(j,T).
\end{equation}

\item \label{pkdubs-iii} For $1 \leq i \leq M$,
\begin{equation} \label{eq:midub}
\inf_{0 \leq s \leq T} \tilde{X}_i^M(s) \leq \alpha_M^{-1}\cW_{M}(1,T) + \sqrt{T} \mathcal{J}(i,T),
\end{equation}
where 
\begin{equation} \label{eq:Jdef}
\mathcal{J}(i,T) := \frac{1}{\sqrt{T}} \sup_{0 \leq s \leq T} \left( -\cV_i^-(1,s) + W_i(s) - 
\sum_{j = 1}^{i-1} r^{i - j} \cV^-_{i - j + 1}(j,s) \right).
\end{equation}

\end{enumerate}
\end{lemma}

\begin{proof}
We will prove the result in the case $p \neq q$. It is easy to check that, with minor modifications, the arguments still go through in the case $p = q$ (i.e. $r = 1$), as well.

(i) For $1 \leq k \leq M$, we let 
$$
Y_k^M(t) := \sum_{j = 1}^k r^{k - j}\tilde{X}^M_j(t), \quad t \geq 0.
$$
We compute, using \eqref{eq:loctimeeq}, 
\begin{equation} \label{eq:cancelation}
\begin{split}
Y_k^M(t) & = W_k(t) + \sum_{j = 1}^k \left( r^{k - j} p \tilde L^M_{(j-1,j)}(t) - r^{k - j - 1}p \tilde L^M_{(j,j+1)}(t) \right) \\
& = W_k(t) - q\tilde L^M_{(k,k+1)}(t).
\end{split}
\end{equation}
Here $W_k$ is defined as in \eqref{eq:altBMdef}, and we use $p = rq$ to obtain the summand in the first equality. We obtain the second equality by evaluating the telescoping sum in the previous line. 

We will prove the following claim.

\vspace{0.2in}

\textbf{Claim.} For $1 \leq k \leq M$ and $T \in [0,\infty)$,
$$
Y_k^M(T) \leq -r^k \alpha_{M - k} \inf_{0 \leq s \leq T} \tilde{X}^M_1(s) + \cW_{M-k+1}(k,T).
$$

\vspace{0.2in}

To see that \ref{pkdubs-i} follows from the claim, note that taking $k = 1$ in the claim gives us 
\[
\tilde{X}_1^M(T) \leq - \left(\frac{r - r^M}{1 - r}\right)\inf_{0 \leq s \leq T} \tilde{X}_1^M(s) + \cW_M(1,T).
\]
Bounding $\tilde{X}_1^M(T)$ below by $\inf_{0 \leq s \leq T} \tilde{X}_1^M(s)$ and rearranging terms gives us 
\[
\cW_M(1,T) \geq \left( 1 + \frac{r - r^M}{1 - r} \right)\inf_{0 \leq s \leq T} \tilde{X}_1^M(s) = \alpha_M \inf_{0 \leq s \leq T} \tilde{X}_1^M(s),
\]
which gives us \ref{pkdubs-i}.

To prove the claim, we proceed as follows. First, observe that for $1 \leq k \leq M - 1$ and $t \geq 0$, 
\begin{equation} \label{eq:Ylb}
\begin{split}
Y_{k + 1}^M(t) & = r^k \tilde{X}_1^M(t) + \sum_{j = 2}^{k + 1} r^{k + 1 - j} \tilde{X}_j^M(t) \\
& \geq r^k \tilde{X}_1^M(t) + \sum_{j = 2}^{k+1} r^{k + 1 - j} \tilde{X}^M_{j-1}(t) \\
& = r^k \tilde{X}_1^M(t) + Y_k^M(t),
\end{split}
\end{equation}
where we use monotonicity of the sequence $\{\tilde{X}_j^M(t)\}$ to obtain the second line. Set 
$$
I_1^M(T) = \inf_{0 \leq s \leq T} \tilde{X}_1^M(s).
$$
From \eqref{eq:Ylb}, we obtain 
\begin{equation} \label{eq:Yrecurs1}
\begin{split}
\inf_{0 \leq s \leq T} \left( Y_{k+1}^M(s) - W_k(s) \right) & \geq r^k I_1^M(T) + \inf_{0 \leq s \leq T} \left(Y_k^M(s) - W_k(s) \right) \\
& = r^k I_1^M(T) + \inf_{0 \leq s \leq T} \left(-q \tilde L^M_{(k,k+1)}(s) \right) \\
& = r^k I_1^M(T) - q \tilde L^M_{(k,k+1)}(T),
\end{split}
\end{equation}
where the second line uses \eqref{eq:cancelation}, and the third line uses monotonicity of the local time. Adding $W_k(T)$ to both sides of \eqref{eq:Yrecurs1} and using \eqref{eq:cancelation} again, we obtain the recursive relationship
\begin{equation} \label{eq:Yrecurs2}
Y_k^M(T) \leq -r^k I_1(T) + \inf_{0 \leq s \leq T} \left( Y_{k+1}^M(s) + W_k(T) - W_k(s) \right).
\end{equation}
To prove the claim we first establish the following bound.
\begin{equation} \label{eq:Yub}
Y_k^M(T) \leq -\sum_{\ell = k}^{M-1} r^l I_1(T) + \inf_{0 \leq s_{M - 1} \leq \cdots \leq s_k \leq T} \sum_{\ell = k}^{M} (W_\ell(s_{\ell-1}) - W_\ell(s_\ell)), \quad 1 \leq k \leq M,
\end{equation}
where $s_{M} := 0$ and $s_{k-1} := T$, and by convention we take the first sum to be zero and the second term to be $W(T)$ when $k=M$. To see why this bound holds, we proceed by downward induction on $k$, starting from $M$. In the base case $k = M$, the inequality in \eqref{eq:Yub} just says that $Y_M^M(T) \leq W_M(T)$. And, in fact, $Y_M^M(T) = W_M(T)$ by \eqref{eq:cancelation}, recalling that $\tilde L^M_{(M,M+1)} \equiv 0$. We now proceed with the induction step: Suppose that the inequality in \eqref{eq:Yub} holds with $k$ replaced with $k+1$ for some $1\le k \le M-1$, and consider \eqref{eq:Yub} with such a $k$. 
 By \eqref{eq:Yrecurs2} and the induction assumption, 
\[
\begin{split}
Y_k^M(T) \leq -r^k I_1(T) + \inf_{0 \leq s \leq T} & \Bigg\{ -\sum_{\ell = k+1}^{M-1} r^\ell I_1(s) + \inf_{0 \leq s_{M - 1} \leq \cdots \leq s_{k+1} \leq s} \sum_{\ell = k+1}^{M} (W_\ell(s_{\ell-1}) - W_\ell(s_\ell)) \\
& \hspace{2.3in} + W_k(T) - W_k(s) \Bigg\}.
\end{split}
\]
Noting that $I_1(T) \leq I_1(s)$ for $s \in [0,T]$ completes the induction step $k + 1 \to k$. This proves the statement in \eqref{eq:Yub} which in turn proves the claim on observing that $\sum_{\ell = k}^{M-1} r^l = r^k \alpha_{M-k}$.
 This concludes the proof of (i).

(ii) By \eqref{eq:cancelation}, recalling that $\tilde L_{(M,M+1)}^M = 0$, 
$$
\sum_{j = 1}^M r^{M - j} \tilde{X}_j^M(t) = Y_M^M(t) = W_M(t).
$$
Hence, 
$$
\tilde{X}_M^M(t) = W_M(t) - \sum_{j = 1}^{M-1} r^{M - j} \tilde{X}_j^M(t).
$$
Moreover, by Lemma \ref{lem:lppcontrol}, we have $\tilde{X}_j^M(t) \geq \cV^-_{M-j+1}(j,t)$ for $1 \leq j \leq M-1$. Part \ref{pkdubs-ii} follows. 

(iii)
Fix $1\le i \le M$.
By \cite[Corollary 3.7]{AS2}, since $M \geq i$, we have 
\begin{equation} \label{eq:diffub}
\tilde{X}_i^{M}(s) - \tilde{X}_1^{M}(s) \leq \tilde{X}_i^{i}(s) - \tilde{X}_1^i(s), \text{ for $s \geq 0$.}
\end{equation}
On the other hand, taking $M = i$ in part (ii) gives us
\begin{equation} \label{eq:topub}
\tilde{X}_i^i(s) \leq W_i(s) - \sum_{j = 1}^{i-1} r^{i - j} \cV^-_{i - j + 1}(j,s),
\end{equation}
and by Lemma \ref{lem:lppcontrol},
\begin{equation} \label{eq:botlb}
\tilde{X}_1^i(s) \geq \cV_i^-(1,s).
\end{equation}
Using the bounds \eqref{eq:topub} and \eqref{eq:botlb} to the estimate the terms appearing in \eqref{eq:diffub}, we obtain
\begin{equation}
\tilde{X}_i^{M}(s) \leq \tilde{X}_1^{M}(s) + \left( W_i(s) - \sum_{j = 1}^{i-1} r^{i - j} \cV^-_{i - j + 1}(j,s) \right) - \cV^-_i(1,s).
\end{equation}
Taking the infimum of both sides over $s \in [0,T]$, and applying part \ref{pkdubs-i} of the lemma, we complete the proof of (iii).
\end{proof}

\subsection{The infinite particle system}
\label{sec:ptw2}
We now return to the infinite system.
 Fix parameters $p \in [0,1)$, $\bar{g}$ as before, and let a filtered probability space, Brownian motions $\{B_j\}$ and sequence $\bar x = \{x_j\}$ be as at the beginning of Section \ref{sec:sepu}.
Let $\{X_j(t), j \in \mathbb{N}, t \in [0,\infty)\}$ be a solution to \eqref{eq:loctimeeq}, with $N = \infty$, with driving noises $\{B_j\}$ and initial condition $\bar x$. Associated local times are denoted as $\{L_{(j,j+1)}\}$.
 For $i,M \in \mathbb{N}$ and $T \in [0,\infty)$, define 
\begin{equation} \label{eq:defIstar}
I_M^*(i,T) := \inf_{0 \leq s \leq T} X_{i + M}(s).
\end{equation}
Note that by definition $I^*(i,0) = X_{i+M}(0) = x_{i + M}$.

\begin{lemma}\label{lem:Xbds_plessq} For all $M \in \mathbb{N}$ and $T \in (0,\infty)$, the following bounds hold.
\begin{enumerate}[label = (\roman*)]
\item \label{lem:Xbds_plessq-i} For $i \in \mathbb{N}$,
\begin{equation} \label{eq:variX_plessq}
\begin{split}
X_i(T) \leq \inf_{0 \leq s_{i+M - 1} \leq \cdots \leq s_i \leq T} & \Bigg\{X_{i+M}(s_{i + M - 1}) + \sum_{j = i}^{i+M-1} (V_j(s_{j-1}) - V_j(s_j)) \\
& \hspace{0.5in} + \sum_{j = i}^{i+M-1} p\left(L_{(j-1,j)}(s_{j-1}) - L_{(j-1,j)}(s_j) \right) \Bigg\},
\end{split}
\end{equation}
where $s_{i-1} := T$. Moreover, on the event $K^*(i,T) \geq M$, \eqref{eq:variX_plessq} is an equality.

\item \label{lem:Xbds_plessq-ii} For $i \in \mathbb{N}$, on the event $K^*(i,T) \geq M$,
\begin{equation} \label{eq:varlb_plessq}
X_i(T) \geq I^*_M(i,T) + \inf_{0 \leq s \leq T} V_{i+M-1}(s) + \cV^-_M(i,T).
\end{equation}

\item \label{lem:Xbds_plessq-iii}
\begin{equation} \label{eq:XMub}
X_M(T) \leq x_M + W_M(T) - \sum_{j = 1}^{M-1} r^{M - j} \cV^-_{M - j + 1}(j,T).
\end{equation}

\item \label{lem:Xbds_plessq-iv} For $1 \leq i \leq M$,
\begin{equation} \label{eq:infXibd}
\inf_{0\leq s \leq T} X_i(s) \leq x_{M} + \alpha_M^{-1}\cW_{M}(1,T) + \sqrt{T} \mathcal{J}(i,T),
\end{equation}
where $\mathcal{J}(i,T)$ is as in Lemma \ref{lem:pkdubs}\ref{pkdubs-iii}.
\end{enumerate}
\end{lemma}

\begin{proof}
(i) Observe that for $j \in \mathbb{N}$ and $t \in [0,\infty)$, 
$$
X_j(t) = x_j + U_j(t) - \hat{L}_{(j,j+1)}(t),
$$
where 
\begin{equation} \label{eq:modifieddrivers}
U_j(t) := V_j(t) + pL_{(j-1,j)}(t), \quad\quad \hat{L}_{(j,j+1)}(t) := qL_{(j,j+1)}(t).
\end{equation}
In view of the fact that $\hat{L}_{(j,j+1)}$ is continuous, non-decreasing, and can only increase when $X_j(t) = X_{j+1}(t)$, we see that $X_j$ is the Skorokhod reflection from below of the trajectory $x_j + U_j$ from $X_{j+1}$; hence 
\begin{equation} \label{eq:one-sided_spread}
X_j(t) = x_j + U_j(t) + \inf_{0 \leq s \leq t} \left( X_{j+1}(s) - x_j - U_j(s) \right) \wedge 0, \quad \text{ for all } t \in [0,\infty), j \in \mathbb{N}.
\end{equation}
For $t \in [0,\infty)$, let $\hat{X}_{i + M}(t) = X_{i + M}(t)$, and for $1 \leq j \leq i + M - 1$, let $\hat{X}_j$ be the Skorokhod reflection from below of the trajectory $x_{i + M} + U_j$ from $\hat{X}_{j+1}$, i.e. 
\begin{equation} \label{eq:one-sided_packed}
\begin{split}
\hat{X}_j(t) & = x_{i + M} + U_j(t) + \inf_{0 \leq s \leq t} \left( \hat{X}_{j+1}(s) - x_{i + M} - U_j(s) \right) \wedge 0 \\
& = \inf_{0 \leq s \leq t} \left( \hat{X}_{j+1}(s) + U_j(t) - U_j(s) \right).
\end{split}
\end{equation}
It follows from \eqref{eq:one-sided_spread}, the first line of \eqref{eq:one-sided_packed}, and an easy induction argument that 
\begin{equation} \label{eq:packed_ubd}
X_j(t) \leq \hat{X}_j(t) \text{ for all } t \in [0,\infty), 1 \le j \le i+M.
\end{equation}
Moreover, iteratively applying the equality \eqref{eq:one-sided_packed} $M$ times to the right-hand side of \eqref{eq:packed_ubd} with $j = i$ and $t = T$, we obtain 
\[
X_i(T) \leq \inf_{0 \leq s_{i+M-1} \leq \cdots  s_{i+1} \leq s_i \leq T} \left\{ X_{i+M}(s_{i + M - 1}) + \sum_{j = i}^{i + M - 1} U_j(s_{j - 1}) - U_j(s_j) \right\},
\]
where $s_{i - 1} := T$, recalling that $X_{i+M} = \hat{X}_{i+M}$. In view of \eqref{eq:modifieddrivers}, the right-hand side of the inequality above agrees with that of \eqref{eq:variX_plessq}. This proves the inequality part of \ref{lem:Xbds_plessq-i}. 

To prove that if $K^*(i,T) \geq M$ then \eqref{eq:variX_plessq} is an equality, define a decreasing sequence $s^*_j$, $i \leq j \leq i + M$, as follows. Let $s^*_{i + M} = 0$, and let 
$$
s_j^* := \inf\{s \geq s^*_{j+1} : X_j(s) = X_{j+1}(s)\} \text{ for } i \leq j \leq i + M - 1.
$$
Note that if $K^*(i,T) \geq M$, then $s_i^* \leq T$. For if there exists a sequence $0 \leq s_{i+M-1} \leq \cdots \leq s_i \leq T$ such that $X_j(s_j) = X_{j+1}(s_j)$ for $i \leq j \leq i + M - 1$, then by definition $s^*_{i + M - 1} \leq s_{i + M - 1}$, and by induction $s^*_j \leq s_j$ for each $j$. In particular, $s^*_i \leq s_i \leq T$. Consequently, the rest of the statement \ref{lem:Xbds_plessq-i} is obtained by taking $j = i$  in the following claim and recalling that the right side of \eqref{eq:variX_plessq} equals $\hat X_i(T)$.

\vspace{0.2in}

\textbf{Claim.} For $i \leq j \leq i + M$ and $t \geq s_j^*$, $X_j(t) = \hat{X}_j(t)$.

\vspace{0.2in}

We prove the claim by induction. The case $j = i + M$ is immediate from the definitions of $s^*_{i+M}$ and of $\hat{X}_{i+M}$. The induction step $j+1 \to j$ is accomplished as follows. In view of \eqref{eq:packed_ubd}, $X_j(t) \leq \hat{X}_j(t) \leq \hat{X}_{j+1}(t)$ for all $t \geq 0$. Furthermore, by definition of $s_j^*$ and the induction assumption, $X_j(s_j^*) = X_{j+1}(s_j^*) = \hat{X}_{j+1}(s_j^*)$. Therefore, 
\begin{equation} \label{eq:coinatj}
X_j(s_j^*) = \hat{X}_j(s_j^*). 
\end{equation}
Moreover, by translation invariance of solutions to \eqref{eq:loctimeeq} (Lemma \ref{lem:translate}), for $t \geq s_j^*$, 
\[
\begin{split}
X_j(t) & = X_j(s_j^*) + U_j(t) - U_j(s_j^*) + \inf_{s_j^* \leq s \leq t} \left( X_{j+1}(s) - X_j(s_j^*) - U_j(s) + U_j(s_j^*) \right) \wedge 0 \\
& = \hat{X}_j(s_j^*) + U_j(t) - U_j(s_j^*) + \inf_{s_j^* \leq s \leq t} \left( \hat{X}_{j+1}(s) - \hat{X}_j(s_j^*) - U_j(s) + U_j(s_j^*) \right) \wedge 0 \\
& = \hat{X}_j(t),
\end{split}
\]
where we use \eqref{eq:coinatj} and the induction assumption again in the second line. This completes the proof of the claim, and statement \ref{lem:Xbds_plessq-i} of the lemma.

Part \ref{lem:Xbds_plessq-ii} follows by substituting the bounds $X_{i+M}(s_{i+M-1}) \geq I_M^*(i,T)$ and $L_{(j-1,j)}(s_{j-1}) - L_{(j-1,j)}(s_j) \geq 0$ into \eqref{eq:variX_plessq}, and recalling from part (i) that since $K^*(i,T) \geq M$, the inequality in \eqref{eq:variX_plessq} is in fact an equality. We obtain 
\[
\begin{split}
X_i(T) & \geq I_M^*(i,T) + \inf_{0 \leq s_{i+M - 1} \leq \cdots \leq s_i \leq T} \sum_{j = i}^{i+M-1} (V_j(s_{j-1}) - V_j(s_j)) \\
& \geq I_M^*(i,T) + \inf_{0 \leq s \leq T}V_{i+M-1}(s) + \cV^{-}_M(i,T).
\end{split}
\]
Note that in the first line, we take the infimum starting from $s_{i+M-1}$ unlike in \eqref{eq:VMdef} where we start from $s_{i+M-2}$ (and $s_{i+M-1} = 0$). This accounts for the extra term $\inf_{0 \leq s \leq T}V_{i+M-1}(s)$ in the last line.

To prove \ref{lem:Xbds_plessq-iii}, by Lemma \ref{lem:bd_by_fin}(ii), we have
\begin{equation} \label{eq:finbd1}
X_j(s) \leq X_j^M(s) \text{ for } 1 \leq j \leq M, s \in [0,\infty).
\end{equation}
Taking $j = M$ and $s = T$, \eqref{eq:finbd1}, Lemma \ref{lem:bd_by_fin}(iii) and Lemma \ref{lem:pkdubs}\ref{pkdubs-ii} give us (iii).

To establish \ref{lem:Xbds_plessq-iv}, we
apply Lemma \ref{lem:bd_by_fin}\ref{it:bf3} to \eqref{eq:finbd1} to obtain 
$$
X_j(s) \leq x_M + \tilde{X}_j^M(s), s \in [0,\infty).
$$
Taking the infimum of both sides over $s \in [0,T]$ and then applying Lemma \ref{lem:pkdubs}\ref{pkdubs-iii} to bound the right-hand side, we  obtain the result.
\end{proof}

\section{Bounds from Brownian last-passage percolation} \label{ssec:lpp}

{The estimates of Section \ref{ssec:ptwise} reduce the analysis of collision chains to
probability bounds for the last-passage type quantities $\cV^-_M$, $\cV^+_M$
and $\cW_M$ introduced, respectively, in \eqref{eq:VMdef}, \eqref{eq:VMplusdef} and \eqref{eq:WMdef}. The goal of this section is to collect the required bounds.
The estimates for $\cV^-_M$ and $\cV^+_M$ are concentration
bounds around their typical scale $2\sqrt{MT}$, and are obtained by using
the connection between Brownian last-passage percolation and the largest
eigenvalue of a GUE random matrix. The tail probability estimate for $\cW_M$, which is needed
only in the case $p<q$, proceeds via concentration inequalities for extrema of Gaussian processes around the mean and upper bounding the associated expectation by considering Gaussian quantities indexed by a tractable subset of up-right grid paths.}

\begin{lemma} \label{lem:lpp}
\begin{enumerate}[label = (\roman*)]
\item \label{it:lpp1} For every $\alpha_0,T \in (0,\infty)$, there exists $C = C(\alpha_0,T) \in (0,\infty)$ such that for all $M \in \mathbb{N}$, $i\in \NN$, and $\alpha \in [\alpha_0,\infty)$,
\begin{equation} \label{eq:vmlpp}
\mathbb{P}\left( \left| \frac{\cV^-_M(i,T)}{\sqrt{MT}} + 2 \right| \geq \alpha \right) \leq C \exp(-M (\alpha^{3/2} \wedge \alpha^3)/C),
\end{equation}
and
\begin{equation} \label{eq:vplpp}
\mathbb{P}\left( \left| \frac{\cV^+_M(i,T)}{\sqrt{MT}} - 2 \right| \geq \alpha \right) \leq C \exp(-M (\alpha^{3/2} \wedge \alpha^3)/C).
\end{equation}

\item \label{it:lpp2} 
Assume $p < q$. There exist $\delta_0 \in (0,\infty)$ and for each $T>0$, $M_0(T) \in \mathbb{N}$ such that, for all $i \in \mathbb{N}$, $t \in (0,T]$ and $M \geq M_0(T)$,
\begin{equation} \label{eq:modbmlpp}
\mathbb{P}\left( \frac{\cW_M(i,t)}{\sqrt{Mt}}  \geq -\delta_0 \right) \leq \delta_0^{-1} e^{-\delta_0 M}.
\end{equation}
\end{enumerate}
\end{lemma}

\begin{proof} In the following argument, $\mathcal{V}_M^+(i,T,\bar{g})$ denotes the quantity defined by \eqref{eq:VMplusdef} with drift parameters given by $\bar{g} = \{g_j\}_{j \in \mathbb{N}}$, and $\cV_M^-(i,T, \bar g)$ is defined similarly. To obtain (i), it suffices to prove \eqref{eq:vplpp}, noting that $\cV_M^+(i,T) = \cV_M^+(i,T, \bar g)$ is equal in distribution to $-\cV_M^-(i,T, -\bar g) $. Moreover, we may assume without loss of generality that $\bar{g} = \bar{0}$, where $\bar{0}$ is the vector of zeroes. Indeed,  setting $\bar{g} = \bar{0}$ in \eqref{eq:VMplusdef} amounts to the same thing as replacing each $V_j$ in the expression with $B_j$. Consequently, 
\begin{equation} \label{eq:vgv0comp}
\begin{split}
\cV_M^+(i,T,\bar{g}) & = \sup_{0 \leq t_i \leq \cdots \leq t_{i + M - 2} \leq T} \sum_{j = i}^{i+M-1} \left\{g_j(t_j - t_{j-1}) + (B_j(t_{j}) - B_j(t_{j-1})) \right\} \\
& \leq \sup_{0 \leq t_i \leq \cdots \leq t_{i + M - 2} \leq T} \sum_{j = i}^{i+M-1} \left\{|\bar{g}|_\infty(t_j - t_{j-1}) + (B_j(t_{j}) - B_j(t_{j-1})) \right\} \\
& = |\bar{g}|_\infty T + \cV_M^+(i,T,\bar{0}).
\end{split}
\end{equation}
A similar estimate in which we replace each $g_j$ with $-|\bar{g}|_\infty$ yields 
\[
\cV_M^+(i,T,\bar{g}) \geq -|\bar{g}|_\infty T  + \cV_M^+(i,T,\bar{0}).
\]
Combining the two estimates we see that 
\[
|\cV_M^+(i,T,\bar{g}) - \cV_M^+(i,T,\bar{0})| \leq |\bar{g}|_\infty T.
\]
Since the above quantity normalized by $\sqrt{Mt}$ goes to $0$ as $M\to \infty$, we see  that replacing $\cV_M^{\pm}(i,T) = \cV_M^{\pm}(i,T,\bar{g})$ by $\cV_M^{\pm}(i,T,\bar{0})$ on the left sides of \eqref{eq:vmlpp} and \eqref{eq:vplpp}
  has no effect on the bounds, other than potentially changing the constant $C$.

We now recall that $\cV^{+}_M(i,T,\bar{0})$ is equal in distribution to $\lambda_{\text{max}}(H_2)$, the largest eigenvalue of an $M \times M$ GUE matrix with complex $N(0,T)$ entries (see Remark \ref{rem:maxeig}). 
Consequently, we may apply large and small deviation bounds for $\lambda_{\text{max}}(H_2)$ given in \cite[Theorem 1 and Equation (1.4)]{smalldev2010} to obtain the following bound: there exists a universal constant $C \in (0,\infty)$ such that, for all $\alpha \in (0,\infty)$,
\begin{equation} \label{eq:VMplusbd}
\PP\left( \left| \frac{\cV_M^+(i,T,\bar{0})}{\sqrt{MT}} - 2 \right| \geq \alpha \right) \leq C \exp(-M(\alpha^{3/2} \wedge \alpha^3)/C).
\end{equation}
Indeed, for $\alpha \in (0,2]$, the exponent $\alpha^{3}$  arises from the left tail bound in 
\cite[Theorem 1]{smalldev2010} and the exponent $\alpha^{3/2}$ from the right tail bound in the same theorem. The case $\alpha>2$ follows on appealing to the right tail bound from Equation (1.4) in that paper (note that the left tail bound for $\alpha>2$ holds trivially in our case since $\cV_M^+(i,T,\bar{0})$ is nonnegative.)
This gives us the desired bound for $\cV_+(i,T,\bar{0})$, and part (i) is proved (the lower bound $\alpha_0$ on $\alpha$ and dependence of $C$ on $\alpha_0$ is required to translate the bound from $\cV_M^+(i,T,\bar{0})$ to $\cV_M^+(i,T,\bar{g})$).

To prove (ii), we may once again reduce the argument to the zero drift case. For $t \in (0,T]$, letting $\cW_M(i,t,\bar{g})$ denote the quantity defined by \eqref{eq:WMdef} with drift parameters $\bar{g} = \{g_j\}$, by a similar calculation to \eqref{eq:vgv0comp}, we obtain
\[
\begin{split}
\cW_M(i,t,\bar{g}) & \leq \inf_{0 \leq s_{i+M-2} \leq \cdots \leq s_i \leq t} \sum_{k = i}^{i+M-1} \sum_{j = 1}^k r^{k - j}\left\{|\bar{g}|_{\infty}(s_{k-1} - s_k) + (B_j(s_{j-1}) - B_j(s_j))\right\} \\
& \leq \frac{1}{1-r}|\bar{g}|_\infty T + \cW_M(i,t,\bar{0}),
\end{split}
\]
where the factor of $1/(1-r)$ in the second bound is obtained by summing the geometric series. We obtain a lower bound on $\cW_M(i,t,\bar{g})$ by replacing $|\bar{g}|_\infty$ with $-|\bar{g}|_\infty$ in estimate above, and we conclude that 
\[
|\cW_M(i,t,\bar{g}) - \cW_M(i,t,\bar{0})| \leq \frac{1}{1-r}|g|_{\infty}T.
\]
Since the above quantity normalized by $\sqrt{Mt}$ goes to $0$ as $M\to \infty$, we see  that replacing $\cW_M(i,t) = \cW_M(i,t,\bar{g})$ with $\cW_M(i,t,\bar{0})$ on the left side of \eqref{eq:modbmlpp}
  has no effect on the bounds, other than potentially changing  $\delta_0$ and $M_0(T)$.
  Thus, for the rest of the argument we may take $g_j = 0$ (and $V_j = B_j$) for all $j \in \mathbb{N}$.

For $M \in \mathbb{N}$, define the quantity 
\[
\mathcal{W}_M^* = \mathcal{W}_M(1, M) = \inf_{0 \leq s_{M-1} \leq \cdots \leq s_1 \leq M} \sum_{k = 1}^M (W_k(s_{k-1}) - W_k(s_k)),
\]
where $s_M = 0$ and $s_0 = M$. In view of Brownian scaling, it suffices to prove that there exist constants $c, \delta \in (0,\infty)$ such that for all $M \in \mathbb{N}$,
\begin{equation}
\mathbb{P}(\mathcal{W}_M^* \geq -\delta M) \leq 2 e^{-cM}.
\end{equation}
To this end, let $\Pi := \{(s
_M, s_{M-1},\dots,s_1, s_0) : 0 = s_M \leq s_{M-1} \leq \cdots \leq s_1 \leq s_0 = M\}$ be the topological space of partitions of $[0,M]$ with $M-1$ points. For each $\pi \in \Pi$, define the Gaussian random variable 
$$
L(\pi) = \sum_{i = 1}^M (W_i(s_{i-1}) - W_i(s_i)).
$$
Note that $\mathbb{E}[L(\pi)] = 0$, and 
\[
\begin{split}
\Var(L(\pi)) & = \Var\left( \sum_{i = 1}^M \sum_{j = 1}^i r^{i - j}(B_j(s_{i-1}) - B_j(s_i)) \right) \\
& = \sum_{i = 1}^M \sum_{j = 1}^i r^{2(i - j)}(s_{i-1} - s_i) \leq \frac{M}{1 - r^2}.
\end{split}
\]
Hence, by the Borell-TIS Inequality \cite[Theorem 2.1.1]{AdlerTaylor}, for $u \in (0,\infty)$,
\[
\mathbb{P}\left( \left| \inf_{\pi \in \Pi} L(\pi) - \mathbb{E}\left[ \inf_{\pi \in \Pi} L(\pi) \right] \right| > uM \right) \leq 2 \exp( -\frac{u^2(1 - r^2)M}{2} ).
\]
Thus, it suffices to show there exists $M_0 > 0, \eta > 0$ such that for all $M \geq M_0$,
\begin{equation} \label{eq:expbd}
\mathbb{E}\left[ \inf_{\pi \in \Pi} L(\pi) \right] \leq -\eta M.
\end{equation}

Consider the random partition 
$$\pi^* = \{0 = s_M^* \leq s^*_{M-1} \leq \cdots \leq s^*_1 \leq s_0^* = M\},$$
where for $j= 1, \ldots, M-1$, with $A_j = \{W_{M - j + 1}(j+1) - W_{M - j + 1}(j) \le W_{M - j}(j+1) - 
W_{M - j}(j)\}$,
$$s^*_{M-j} = (j+1) 1_{A_j} + j 1_{A_j^c}.$$ 
Then
\begin{equation} \label{eq:sumrep}
L(\pi^*) = W_M(1) + \sum_{j = 1}^{M - 1} \Psi_j,
\end{equation}
where
\[
\begin{split}
& \Psi_j = G_{1,j} \wedge G_{2,j}, \\
& G_{1,j} = W_{M - j + 1}(j+1) - W_{M - j + 1}(j), \\ 
& G_{2,j} = W_{M - j}(j+1) - W_{M - j}(j).
\end{split}
\]

We compute 
\[
\begin{split}
\Var(G_{1,j}) & = \Var\left( \sum_{ \ell = 1}^{M - j + 1} r^{M - j + 1 - \ell} (B_{\ell}(j+1) - B_{\ell}(j)) \right) \\
& = \sum_{ \ell = 1}^{M - j + 1} r^{2(M - j + 1 - \ell)} = \frac{1 - r^{2(M-j+1)}}{1 - r^2},
\end{split}
\]
and similarly
\[
\Var(G_{2,j}) = \frac{1 - r^{2(M-j)}}{1 - r^2}.
\]
And 
\[
\begin{split}
\text{Cov}(G_{1,j},G_{2,j}) & = \text{Cov}\left( \sum_{ \ell = 1}^{M - j + 1} r^{M - j + 1 - \ell} (B_{\ell}(j+1) - B_{\ell}(j)), \sum_{ \ell = 1}^{M - j} r^{M - j - \ell} (B_{\ell}(j+1) - B_{\ell}(j)) \right) \\
& = \sum_{\ell = 1}^{M - j} r \cdot r^{2(M - j - \ell)} = \frac{r(1 - r^{2(M-j)})}{1 - r^2}.
\end{split}
\]
Therefore, we can write 
\[
\begin{split}
& G_{1,j} = a_{1,j}U_{1,j} + b_j V, \\
& G_{2,j} = a_{2,j}U_{2,j} + b_j V,
\end{split}
\]
where $U_{1,j}, U_{2,j},$ and $V$ are i.i.d.\ $N(0,1)$, and 
\[
\begin{split}
& a_{1,j}^2 = \frac{1 - r^{2(M-j+1)} - r(1 - r^{2(M-j)})}{1 - r^2}, \\
& a_{2,j}^2 = \frac{1 - r^{2(M-j)} - r(1 - r^{2(M-j)})}{1 - r^2}, \\
& b^2_j = \frac{r(1 - r^{2(M-j)})}{1 - r^2}.
\end{split}
\]
Hence, for all $1 \leq j \leq M - 1$,
\[
\begin{split}
\mathbb{E}[\Psi_j] & = \mathbb{E}[G_{1,j} \wedge G_{2,j}] = \mathbb{E}\left[\frac{1}{2}(G_{1,j} + G_{2,j} - |G_{1,j} - G_{2,j}|)\right] \\
& = -\frac{1}{2}\mathbb{E}|G_{1,j} - G_{2,j}| = -\frac{1}{\sqrt{2\pi}}\sqrt{a_{1,j}^2 + a_{2,j}^2} \leq -\frac{1}{\sqrt{2\pi}} \sqrt{\frac{1-r}{1-r^2}}.
\end{split}
\]
Hence, by \eqref{eq:sumrep}, 
\[
\mathbb{E}\left[ \inf_{\pi \in \Pi} L(\pi) \right] \leq \mathbb{E}L(\pi^*)= \sum_{j = 1}^{M-1} \mathbb{E}[\Psi_j] \leq -\frac{1}{\sqrt{2\pi}} \sqrt{\frac{1-r}{1-r^2}}(M-1).
\]
This proves \eqref{eq:expbd} and completes the proof of part \ref{it:lpp2} of the lemma.
\end{proof}

\section{Proof of Theorems \ref{thm:unique1}, \ref{thm:unique2}, and \ref{thm:unique3}}\label{sec:pfthm2.6}

We will now prove our main results on pathwise uniqueness of solutions to \eqref{eq:loctimeeq0}. Throughout the section we assume that $\bar{g} \in \ell^\infty(\mathbb{N})$ and that $p \in [0,1)$.
{We combine the deterministic estimates from Section \ref{ssec:ptwise} with the
probability bounds from Section \ref{ssec:lpp} to prove the pathwise uniqueness
theorems, by exploiting the connection with chains of collisions described in Section \ref{sec:finprop}.}

\subsection{Proof of Theorems \ref{thm:unique1} and \ref{thm:unique2}}

For $0 \leq u < v < \infty$ and $i,M \in \mathbb{N}$, define the quantities
\begin{equation} \label{eq:shiftedlpp}
\cV_M^-(i,[u,v]) = \inf_{u \leq s_{i + M - 2} \leq \cdots \leq s_i \leq v} \sum_{j = i}^{i+M-1} (V_j(s_{j-1}) - V_j(s_j)),
\end{equation}
\begin{equation} \label{eq:shiftedlpp+}
\cV_M^+(i,[u,v]) = \sup_{u \leq t_i \leq \cdots \leq t_{i+M-2} \leq v} \sum_{j = i}^{i+M-1} (V_j(t_{j}) - V_j(t_{j-1})),
\end{equation}
and
\begin{equation} \label{eq:shiftedaltlpp}
\cW_M(i,[u,v]) = \inf_{u \leq s_{i + M - 2} \leq \cdots \leq s_i \leq v} \sum_{j = i}^{i+M-1} (W_j(s_{j-1}) - W_j(s_j)),
\end{equation}
where $s_{i + M - 1} = t_{i-1} := u$ and $s_{i-1} = t_{i + M - 1} := v$. Note 
\[
\begin{split}
\cV_M^{\pm}(i,[u,v]) = \theta_u \cV_M^{\pm}(i,v - u), \quad\quad \cW_M(i,[u,v]) = \theta_u \cW_M(i,v - u),
\end{split}
\]
where $\theta_u$ is the translation map defined by \eqref{eq:translate}. In fact, by translation invariance of Brownian motion, $\cV_M^{\pm}(i,[u,v])$ is equal in distribution to $\cV_M^{\pm}(i,v-u)$, $\cW_M(i,[u,v])$ is equal in distribution to $\cW_M(i,v-u)$, and they are both measurable with respect to the $\sigma$-field 
\begin{equation} \label{eq:futurefield}
\mathcal{F}_{[u,\infty)} := \sigma(B_j(s) - B_j(u) : u \leq s < \infty, j \in \mathbb{N}).
\end{equation} 
Let 
\begin{equation} \label{eq:Rstar_def}
R^*_M(i,T) = \sup_{0 \leq s \leq T} \left\{W_{i+M}(T) - W_{i+M}(s) - \sum_{j = 1}^{{i+M}-1} r^{{i+M} - j}\cV^-_{i+M - j + 1}(j,[s,T]) \right\}.
\end{equation}

\begin{lemma}\label{lem:rstar}
Fix 
 $T \in (0,\infty)$ and let $p < q$. Then there exists $C(T) \in (0,\infty)$, such that for all $\alpha > 0$, and $i,M \in \mathbb{N}$,
\begin{equation} \label{eq:Rstarprobbd}
\mathbb{P}\left( R_M^*(i,T) \geq \alpha \right) \leq C(T)e^{-\alpha}. 
\end{equation}
\end{lemma}

\begin{proof}
Note that
\begin{equation} \label{eq:Rstar_bd1}
R_{M}^*(i,T) \leq \sup_{0 \leq s \leq T} \left( W_{i+M}(T) - W_{i+M}(s) \right) - \sum_{j = 1}^{{i+M} - 1} r^{{i+M} - j} \inf_{0 \leq s \leq T} \cV^-_{i+M - j + 1}(j,[s,T]).
\end{equation}
Furthermore, in view of \eqref{eq:altBMdef},
\begin{equation} \label{eq:modBM_bd1}
\begin{split}
\sup_{0 \leq s \leq T} \left( W_{i+M}(T) - W_{i+M}(s) \right) & \leq \sum_{j = 1}^{i+M} r^{{i+M} - j}\sup_{0 \leq s \leq T} (V_j(T) - V_j(s)) \\
& \leq \frac{|\bar{g}|_\infty T}{1 - r} + \sum_{j = 1}^{i+M} r^{{i+M} - j}\sup_{0 \leq s \leq T} (B_j(T) - B_j(s)).
\end{split}
\end{equation}
And, in view of \eqref{eq:shiftedlpp},
\begin{equation} \label{eq:VM_bd1}
\begin{split}
\inf_{0 \leq s \leq T} \cV^-_{i+M - j + 1}(j,[s,T]) & \geq \sum_{k = j}^{{i+M}} \inf_{0 \leq s \leq t \leq T} \left( V_k(t) - V_k(s) \right) \\
& \geq - (i + M - j+1)|\bar{g}|_\infty T + \sum_{k = j}^{{i+M}} \inf_{0 \leq s \leq t \leq T} \left( B_k(t) - B_k(s) \right).
\end{split}
\end{equation}
From \eqref{eq:Rstar_bd1}, \eqref{eq:modBM_bd1}, and \eqref{eq:VM_bd1}, we obtain
\begin{equation} \label{eq:Rstar_bd2}
\begin{split}
R_{M}^*(i,T) & \leq \frac{|\bar{g}|_\infty T}{1 - r} + \sum_{j = 1}^{i+M} r^{{i+M} - j}\sup_{0 \leq s \leq T} (B_j(T) - B_j(s)) \\
& \hspace{0.5in} + \sum_{j = 1}^{{i+M}-1} r^{{i+M} - j} (i + M - j+1)|\bar{g}|_\infty T \\
& \hspace{1in} - \sum_{j = 1}^{{i+M}-1} r^{{i+M} - j} \sum_{k = j}^{{i+M}} \inf_{0 \leq s \leq t \leq T} \left( B_k(t) - B_k(s) \right) \\
& \leq c_1 |\bar{g}|_\infty T + \sum_{j = 1}^{i+M} r^{{i+M} - j}\sup_{0 \leq s \leq T} (B_j(T) - B_j(s)) \\
& \hspace{0.5in} - c_2 \sum_{k = 1}^{{i+M}} r^{{i+M} - k} \inf_{0 \leq s \leq t \leq T} \left( B_k(t) - B_k(s) \right) \\
& \leq c_3 \left( |\bar{g}|_\infty T + \sum_{j = 1}^{i+M} r^{{i+M} - j}\Big(\sup_{0 \leq s \leq T} B_j(s) - \inf_{0 \leq s \leq T} B_j(s) \Big) \right),
\end{split}
\end{equation}
where $c_1, c_2, c_3 \in (0,\infty)$ are constants that depend only on $r$. For $j \in \mathbb{N}$ and a standard Brownian motion $B$, let
$$
C_0(T) := \mathbb{E} \exp( c_3\Big( \sup_{0 \leq s \leq T} B(s) - \inf_{0 \leq s \leq T} B(s)\Big) ).
$$
Note that $C_0(T)<\infty$. Hence, using \eqref{eq:Rstar_bd2} and independence of the $B_j$'s, for $\alpha > 0$, 
\[
\begin{split}
& \mathbb{P}(R_M^*(i,T) \geq \alpha) = \mathbb{P}\left( e^{R_M^*(i,T)} \geq e^\alpha \right) \leq e^{-\alpha}\mathbb{E}\left[ e^{R_M^*(i,T)} \right] \\
& \hspace{0.6in} \leq e^{-\alpha} \exp(c_3|\bar{g}|_\infty T) \prod_{j = 1}^{i+M} \mathbb{E} \exp( c_3 r^{{i+M} - j} \Big(\sup_{0 \leq s \leq T} B_j(s) - \inf_{0 \leq s \leq T} B_j(s) \Big) ) \\
& \hspace{0.6in} \leq e^{-\alpha} \exp(c_3|\bar{g}|_\infty T) \prod_{j = 1}^{i+M} C_0(T)^{r^{{i+M} - j}} \\
& \hspace{0.6in} = e^{-\alpha} \exp(c_3|\bar{g}|_\infty T) C_0(T)^{ \sum_{j = 1}^{i+M} r^{{i+M} - j}} \\
& \hspace{0.6in} \leq C(T)e^{-\alpha},
\end{split}
\]
as desired. Here the third line follows by Jensen's inequality, noting that $u \mapsto u^{r^{{i+M} - j}}$ is a concave function for $1 \leq j \leq i+M$.
\end{proof}

We now consider a solution to \eqref{eq:loctimeeq}.
 Fix parameters $p \in [0,1)$, $\bar{g}$ as before and let a filtered probability space, Brownian motions $\{B_j\}$ and sequence $\bar x = \{x_j\}$ be as at the beginning of Section \ref{sec:sepu}.
Let $\{X_j(t), j \in \mathbb{N}, t \in [0,\infty)\}$ be a solution to \eqref{eq:loctimeeq}, with $N = \infty$, with driving noises $\{B_j\}$ and initial condition $\bar x$. Associated local times are denoted as $\{L_{(j,j+1)}\}$.
Recall that for $T, M>0$ and $i \in \NN$, $I^*_M(i,T)$ is defined in \eqref{eq:defIstar}.
\begin{lemma} \label{lem:biginf2}
Assume $p < q$. For $i \in \mathbb{N}$ and $T \in (0,\infty)$, 
$$
\liminf_{M \to \infty} \frac{I^*_M(i,T)}{\sqrt{M}} = \infty \text{ a.s.}
$$
Furthermore, $\liminf_{M \to \infty} x_M/\sqrt{M} = \infty$ a.s.
\end{lemma}

\begin{remark} \normalfont
Note that we have made no \textit{a priori} assumption on the initial data $\{x_i\}_{i \in \mathbb{N}}$ in this section, other than their being drawn from a distribution $\gamma \in \clp_0(\mathbb{R}^\infty)$. Thus the second statement of Lemma \ref{lem:biginf2} gives us a necessary condition on the initial data for a solution to \eqref{eq:loctimeeq0} to exist when $p < q$ and $\bar{g} \in \ell^\infty(\mathbb{N})$.
\end{remark}

\begin{proof}
The second statement follows from the first upon observing that, by definition, $I^*_M(1,T) \leq X_{M+1}(0) = x_{M + 1}$ for all $T > 0$ and $M \in \NN$. To prove the first statement, fix $i, T$ and $p$ as in the statement of the lemma. By Lemma \ref{lem:Xbds_plessq}\ref{lem:Xbds_plessq-iv} and translation invariance (Lemma \ref{lem:translate}), for all $0 \leq u \leq v < \infty$,
\begin{equation} \label{eq:xiub2}
\inf_{u \leq s \leq v} X_i(v) \leq X_{i+M}(u) + \alpha_{i + M}^{-1} \cW_{i+M}(1,[u,v]) + \sqrt{v - u} \mathcal{J}(i,[u,v]),
\end{equation}
where  
$$
\mathcal{J}(i,[u,v]) := \frac{1}{\sqrt{v - u}} \sup_{u \leq s \leq v} \left(-\cV^-_{i}(1,[u,s]) + W_i(s) - W_i(u) - \sum_{j = 1}^{i-1} r^{i - j} \cV^-_{i - j + 1}(j,[u,s]) \right).
$$

Let $\tau^* \in [0,T]$ be the first time at which the infimum in the definition \eqref{eq:defIstar} of $I^*_M(i,T)$ is attained.  By Lemmas \ref{lem:translate} and  \ref{lem:Xbds_plessq}\ref{lem:Xbds_plessq-iii},
\begin{equation} \label{eq:diff2bd}
X_{i+M}(T) - X_{i + M}(\tau^*) = \sup_{0 \leq s \leq T} X_{i+M}(T) - X_{i + M}(s) \leq R^*_M(i,T),
\end{equation}
where $R^*_M(i,T)$ is defined as in \eqref{eq:Rstar_def}.

For $K \in (T, \infty)$, choose $\delta_0$ and $M_1(K) = M_0(K-T)$ as in Lemma \ref{lem:lpp}(ii) with $T$ there replaced by $K-T$. Define 
$$
\rho^* = 0 \vee \liminf_{M \to \infty} \frac{I^*_M(i,T)}{\sqrt{M}},
$$
$$
T' = [\delta_0^{-1}(1-r)^{-1}(\rho^* + 2)]^2 + T,
$$ 
Let $F$ be the event that $\rho^* < \infty$ (which is equivalent to $T' < \infty$). We want to show that $\mathbb{P}(F) = 0$. 
By Lemma \ref{lem:translate}, Lemma \ref{lem:Xbds_plessq}\ref{lem:Xbds_plessq-iv}, and \eqref{eq:diff2bd}, 
\begin{equation} \label{eq:infXibd3}
\begin{split}
& \inf_{T \leq s \leq T' \wedge K} X_i(s) = ( \inf_{T \leq s \leq T' \wedge K} X_i(s) - X_{i+M}(T)) + (X_{i + M}(T) - X_{i + M}(\tau^*)) + I_M^*(i,T) \\
& \quad\quad \leq (1-r)\cW_{i+M}(1,[T,T' \wedge K]) + Q(i,T,T',K) + R^*_M(i,T) + I_M^*(i,T),
\end{split}
\end{equation}
where 
$$
Q(i,T,T',K) := \sqrt{T' \wedge K - T} \mathcal{J}(i,[T,T' \wedge K]).
$$
Here  we appeal to the fact that  $\rho^*$ and $T'$ are independent of the $\sigma$-field $\mathcal{F}_{[T,\infty)}$, and we can apply Lemma \ref{lem:translate} and Lemma \ref{lem:Xbds_plessq}\ref{lem:Xbds_plessq-iv} after a suitable conditioning.

Since the left-hand side of \eqref{eq:infXibd3} and the quantity $Q(i,T,T',K)$ are almost surely finite and do not depend on $M$, we see that 
\begin{equation} \label{eq:AKiszero2}
\mathbb{P}(A_K) = 0,
\end{equation}
where 
$$
A_K := \left\{\liminf_{M \to \infty } ( (1-r)\cW_{M+i}(1,[T,T' \wedge K]) + R^*_M(i,T) + I_M^*(i,T) ) = -\infty \right\}.
$$
Next, consider that 
\begin{equation} \label{eq:contains3_2}
\begin{split}
A_K \cap F & \supset \left\{ \liminf_{M \to \infty} \frac{(1-r)\cW_{M+i}(1,[T,T' \wedge K]) + R^*_M(i,T) + I_M^*(i,T)}{\sqrt{M+i}} < 0 \right\} \cap F \\
& \supset \left\{ \rho^* + \limsup_{M \to \infty} \frac{(1-r)\cW_{M+i}(1,[T,T' \wedge K])}{\sqrt{M + i}} + \limsup_{M \to \infty} \frac{R^*_M(i,T)}{\sqrt{M + i}} < 0 \right\} \cap F \\
& \supset B_K \cap C_K \cap D_K,
\end{split}
\end{equation}
where 
\[
\begin{split}
& B_K := \left\{ \rho^* - \delta_0(1-r)\sqrt{T' \wedge K - T} + 1 < 0 \right\} \cap F, \\
& C_K := \left\{ \limsup_{M \to \infty} \frac{\cW_{i+M}(1,[T,T' \wedge K])}{\sqrt{(M+i)(T' \wedge K - T)}} \leq -\delta_0 \right\}, \\
& D_K := \left\{ \limsup_{M \to \infty} \frac{R^*_M(i,T)}{\sqrt{M+i}} \leq 1 \right\}.
\end{split}
\]
Observe that 
\begin{equation} \label{eq:BKlim2}
\mathbb{P}(B_K) = \mathbb{P}\left( \left\{ \rho^* - \delta_0(1-r)\sqrt{T' \wedge K - T} + 1 < 0 \right\} \cap F\right) \to \mathbb{P}(F) \text{ as } K \to \infty.
\end{equation}
Furthermore, for $M^* \in \mathbb{N}$,  with $M^* > M_1(K)$,
\begin{equation} \label{eq:Ckcbd2}
\begin{split}
\mathbb{P}(C_K^c) & \leq \mathbb{P}\left(\exists M \geq M^*, \frac{\cW_{i+M}(1,[T,T' \wedge K])}{\sqrt{(M+i)(T' \wedge K - T)}} > -\delta_0\right) \\
& \leq \sum_{M = M^*}^\infty \mathbb{P}\left( \frac{\cW_{i + M}(1,[T,T' \wedge K])}{\sqrt{(M+i)(T' \wedge K - T)}} > -\delta_0 \right) \\
& \leq \sum_{M = M^*}^\infty \delta_0^{-1} e^{-\delta_0 M},
\end{split}
\end{equation}
where the last line uses Lemma \ref{lem:lpp}\ref{it:lpp2}. Note that although $T'$ is a random variable, we are allowed to apply the lemma since $T'$ is independent of the $\sigma$-field $\mathcal{F}_{[T,\infty)}$, as noted above. Also, \begin{equation} \label{eq:Dkcbd2}
\begin{split}
\mathbb{P}(D_K^c) & \leq \mathbb{P}\left(\exists M \geq M^*, \frac{R^*_M(i,T)}{\sqrt{M+i}} > 1 \right) \\
& \leq \sum_{M = M^*}^\infty \mathbb{P}\left(  R^*_M(i,T) > \sqrt{M+i} \right) \leq \sum_{M = M^*}^\infty C(T) e^{-\sqrt{M+i}},
\end{split}
\end{equation}
where the last inequality uses Lemma \ref{lem:rstar}.
Since the bounds \eqref{eq:Ckcbd2} and \eqref{eq:Dkcbd2} converge to zero as $M^* \to \infty$, this shows that $\mathbb{P}(C_K) = \mathbb{P}(D_K) = 1$. Hence, in view of \eqref{eq:AKiszero2}, \eqref{eq:contains3_2}, and \eqref{eq:BKlim2}, 
$$
0 = \limsup_{K \to \infty} \mathbb{P}(A_K \cap F) \geq \limsup_{K \to \infty} \mathbb{P}(B_K) = \mathbb{P}(F),
$$
and the result follows.
\end{proof}

\begin{lemma} \label{lem:biginf4}
Assume that the condition \eqref{eq:cncstar} holds for the solution $\{X_j(t), j \in \mathbb{N}, t \in [0,\infty)\}$ with local times $\{L_{(j,j+1)}\}$, introduced above Lemma \ref{lem:biginf2}. Assume that the initial conditions $\{x_j\}$ satisfy \eqref{eq:init_cond1}. Then for any $T \in (0,\infty)$, 
\begin{equation}
\liminf_{M \to \infty} \frac{I^*_M(i,T)}{\sqrt{M}} \geq \liminf_{M \to \infty}\frac{x_M}{\sqrt{M}} \text{ a.s.}
\end{equation}
\end{lemma}

\begin{proof}
From \eqref{eq:loctimeeq}, we may obtain for all $i,M \in \mathbb{N}$ and $T \in [0,\infty)$,
\begin{equation} \label{eq:Istarlb}
\begin{split}
\frac{I^*_M(i,T)}{\sqrt{M}} & \geq \frac{x_{i+M} \vee 1}{\sqrt{M}} \Bigg\{\frac{x_{i+M}}{x_{i+M} \vee 1} + \frac{\inf_{s \in [0,T]}V_{i+M}(s)}{x_{i+M} \vee 1} \\
& \hspace{1in} - \sup_{s \in [0,T]} \left( \frac{q L_{(i+M,i+M+1)}(s) - p L_{(i+M-1,i+M)}(s)}{x_{i+M} \vee 1} \right) \Bigg\}.
\end{split}
\end{equation}
The first term inside the braces converges to 1 as $M \to \infty$ by the assumption \eqref{eq:init_cond1} on the initial conditions, and the second converges to zero almost surely by standard Gaussian tail bounds. The $\liminf$, as $M \to \infty$, of the third term in the braces is non-negative by \eqref{eq:cncstar}. Thus the liminf of the overall quantity in the braces can be bounded below by $1$. We thus obtain the result by taking the limit infimum as $M \to \infty$ of both sides of \eqref{eq:Istarlb}.
\end{proof}

\begin{proof}[Proof of Theorems \ref{thm:unique1} and \ref{thm:unique2}] Consider a strong solution $\{X_j\}$ of \eqref{eq:loctimeeq}, with $N = \infty$, with driving noises $\{B_j\}$ and initial distribution $\gamma$. We denote $\{X_j(0)\} = \{x_j\} = \bar x$ and assume that $\bar{g} \in \ell^\infty(\mathbb{N})$. Associated local times are denoted as $\{L_{(j,j+1)}\}$. Assume the hypothesis of either Theorem \ref{thm:unique1} or of Theorem \ref{thm:unique2} holds.

We claim that under the hypothesis of Theorem \ref{thm:unique1} or Theorem \ref{thm:unique2}, we may choose some
$\clf_0$-measurable random variable $\Delta \in (0,1)$
such that 
\begin{equation} \label{eq:Istarlb2}
\textit{for any $T \in [0,\infty)$}, \quad \liminf_{M \to \infty} \frac{I^*_M(i,T)}{\sqrt{M}} > (2 + \Delta)\sqrt{\Delta} \text{ a.s.}
\end{equation}
In the first case this follows by Lemma \ref{lem:biginf2}, while in the second, this follows by Lemma \ref{lem:biginf4} and the assumption on the initial conditions in \eqref{eq:init_cond1}. By Lemma \ref{lem:unique}, 
we will be done if we can show that, for all $i \in \mathbb{N}$ and $t \in (0,\infty)$, $K^*(i,[t, t + \Delta], \{X_j\}) < \infty$ a.s. 

To this end, fix $t \in (0,\infty)$, and let $\hat{X}_j = \theta_t X_j$, $j \in \mathbb{N}$. Recall by Lemma \ref{lem:translate} that $\{\hat{X}_j\}$ solves \eqref{eq:loctimeeq} with Brownian motions $\hat{B}_j := \hat{\theta}_t B_j$, initial conditions $\hat{x}_j := X_j(t)$, and the same parameters $p,\bar{g}$ as before. For $T \in (0,\infty)$, let $\hat{I}_M^*(i,T)$ and $\hat{\mathcal{V}}_M^-(i,T)$ be defined as in \eqref{eq:defIstar} and \eqref{eq:VMdef}, but with $\hat{x}_j, \hat{X}_j, \hat{B}_j$ replacing $x_j,X_j,B_j$, respectively. We also denote $\hat{K}^*(i,\Delta) = K^*(i,[t,t+\Delta], \{X_j\})$. 

By translation invariance of Brownian motion, $\hat{\mathcal{V}}_M^-(i,T)$ is equal in distribution to $\mathcal{V}_M^-(i,T)$. Moreover, the condition \eqref{eq:Istarlb2} 
clearly holds if  $I_M^*(i,T)$ is replaced by $\hat{I}_M^*(i,T)$. Since $T \mapsto \hat{I}^*_M(i,T)$ is non-increasing, and $\Delta < 1$, we have 
\begin{equation} \label{eq:Istarlb3}
\liminf_{M \to \infty} \frac{\hat{I}^*_M(i,\Delta)}{\sqrt{M}} \geq \liminf_{M \to \infty} \frac{\hat{I}^*_M(i,1)}{\sqrt{M}} > (2 + \Delta)\sqrt{\Delta} \text{ a.s.}
\end{equation}

Lemma \ref{lem:Xbds_plessq}\ref{lem:Xbds_plessq-ii} implies that if $\hat{K}^*(i,\Delta) = \infty$, then for all $M \in \mathbb{N}$, $\hat{X}_i(\Delta) \geq \hat{I}_M^*(i,\Delta) + \inf_{s \in [t,t+\Delta]} V_{i+M-1}(s) + \hat{\cV}^-_M(i,\Delta)$. Since $\hat{X}_i(\Delta)$ is finite and does not depend on $M$, 
\begin{equation} \label{eq:0geG}
\begin{split}
0 & = \mathbb{P}\left( \left\{ \limsup_{M \to \infty} (\hat{I}^*_M(i,\Delta) + \inf_{s \in [t,t+\Delta]} V_{i+M-1}(s) + \hat{\cV}^-_M(i,\Delta)) = \infty \right\} \cap \{\hat{K}^*(i,\Delta) = \infty\} \right) \\
& \geq \mathbb{P}\left(\left\{ \limsup_{M \to \infty} \frac{\hat{I}^*_M(i,\Delta) + \inf_{s \in [t,t+\Delta]} V_{i+M-1}(s) + \hat{\cV}^-_M(i,\Delta)}{\sqrt{M}} > 0 \right\} \cap \{\hat{K}^*(i,\Delta) = \infty\} \right) \\
& \geq \mathbb{P}\left( \left\{ \limsup_{M \to \infty} \frac{\hat{\cV}^-_M(i,\Delta)}{\sqrt{M}} > -(2 + \Delta) \sqrt{\Delta} \right\} \cap \{\hat{K}^*(i,\Delta) = \infty\} \right),
\end{split}
\end{equation}
where the last line follows by \eqref{eq:Istarlb3}, and the fact that  $\inf_{s \in [t,t+\Delta]} V_{i+M-1}(s)/\sqrt{M} \to 0 \text{ a.s.}$ as $M \to \infty$, which follows by straightforward argument using Gaussian tail bounds.
Denote the conditional distribution given $\clf_0$ as $\PP_{\clf_0}$. Recall that $\clf_0$ is independent of the Brownian motions $\{B_j\}$, and observe that for all $M_0 \in \mathbb{N}$, a.s.,
\[
\begin{split}
\mathbb{P}_{\clf_0}\left( \limsup_{M \to \infty} \frac{\hat{\cV}^-_M(i,\Delta)}{\sqrt{M}} > -(2 + \Delta) \sqrt{\Delta} \right) &\geq 1 - \mathbb{P}_{\clf_0}\left(\exists M \geq M_0, \frac{\hat{\cV}^-_M(i,\Delta)}{\sqrt{M\Delta}} \leq -2 - \Delta \right) \\
& \geq 1 - \sum_{M = M_0}^\infty \mathbb{P}_{\clf_0}\left( \frac{\hat{\cV}^-_M(i,\Delta)}{\sqrt{M\Delta}} \leq -2 - \Delta \right) \\
& \geq 1 - \sum_{M = M_0}^\infty C \exp(-M (\Delta^{3/2} \wedge \Delta^3)/C),
\end{split}
\]
applying Lemma \ref{lem:lpp}\ref{it:lpp1} in the last line with the constant $C= C(\Delta,\Delta)$ as in the statement of the lemma. Since the quantity above converges to $1$ as $M_0 \to \infty$, this shows that $\mathbb{P}(\hat{K}^*(i,\Delta) = \infty) = 0$, as desired. 
\end{proof}

\subsection{Proof of Theorem \ref{thm:unique3}}

Throughout this subsection, we assume $p > q$. We let 
$$
\sigma = r^{-1} = \frac{q}{p} \in (0,1).
$$
Consider now a solution $\{X_j(t), j \in \mathbb{N}, t \in [0,\infty)\}$ of \eqref{eq:loctimeeq} with driving noises $\{B_j\}$, initial condition $\bar x$ and associated local times  $\{L_{(j,j+1)}\}$.

Fix $i,M \in \mathbb{N}$ and $T \in (0,\infty)$, and define a sequence of stopping times $\tau_{i+M} \leq \tau_{i+M-1} \leq \cdots \leq \tau_i$, as follows: $\tau_{i+M} := 0$, and for $i \leq k \leq i+M-1$,  
\begin{equation} \label{eq:nextcol}
\tau_k := \inf\{t \geq \tau_{k+1} : X_{k+1}(t) = X_k(t)\} \wedge T.
\end{equation}
Also, let $\tau_{i-1} := T$. 

Recall the definitions \eqref{eq:shiftedlpp}-\eqref{eq:shiftedlpp+} of $\cV_M^-(i,[u,v]), \cV_M^+(i,[u,v])$, respectively.

Let $D_{M,j}(i,T)$ denote the set of all doubly indexed sequences 
$$
\mathbf{t} = \{t_\ell^{(k)} : i \leq k \leq i+M, k \leq \ell \leq k + j\} \in \mathbb{R}^{(j+1)(M+1)},
$$
satisfying
$$
t_\ell^{(k)} \leq t_{\ell'}^{(k')} \text{ whenever $k \geq k'$, or $k = k'$ and $\ell \leq \ell'$.}, \quad t^{(i+M)}_{i+M} = 0, \quad t^{(i)}_{i+j} = T.
$$
namely
\begin{align*}
0 &= t_{i+M}^{(i+M)} \le t_{i+M+1}^{(i+M)} \le \cdots \le t_{i+M+j}^{(i+M)}
\le t_{i+M-1}^{(i+M-1)}\\
&\le \cdots \le t_{i+M-1+j}^{(i+M-1)} \le t_{i+M-2}^{(i+M-2)} \le \cdots \le t^{(i)}_i \le \cdots \le t_{i+j-1}^{(i)} \le t_{i+j}^{(i)} = T.
\end{align*}

\begin{lemma} \label{lem:localtimebds}
Assume the condition \eqref{eq:cncstar2} holds. 
\begin{enumerate}[label = (\roman*)]
\item \label{it:lbds1} Almost surely, for all $k \in \mathbb{N}$ and $t \in [0,\infty)$,
\begin{equation} \label{eq:ltk_bd2}
L_{(k,k+1)}(t) \leq q^{-1}\sum_{j = 1}^\infty \sigma^j\left\{\cV^+_{j+k}(1,t) - V_{j+k}(t)\right\}.
\end{equation}
\item \label{it:lbds2} Let $i \leq k \leq i + M$. Almost surely, for $t \in [\tau_k,\tau_{k-1}]$,
\begin{equation} \label{eq:bdaftertau}
L_{(k,k+1)}(t) - L_{(k,k+1)}(\tau_k) \leq q^{-1}\sum_{j = 1}^\infty \sigma^j\left\{\cV^+_{j+1}(k,[\tau_k,t]) - V_{j+k}(t) + V_{j+k}(\tau_k)\right\}.
\end{equation}
\item \label{it:lbds3} Almost surely, the following bound holds:
\begin{equation}
\begin{split}
\sum_{k = i}^{i+M} & (L_{(k,k+1)}(\tau_{k-1}) - L_{(k,k+1)}(\tau_k)) \\
& \leq q^{-1}\sum_{j = 1}^\infty \sigma^j \left\{|\bar g|_\infty T + \cU(j,M,T) - \cV_{M+1}^-(j+i,T) \right\},
\end{split}
\end{equation}
where
\begin{equation} \label{eq:udef}
\cU(j,M,T) := \sup_{\{t_\ell^{(k)}\} \in D_{M,j}(i,T)} \sum_{k = i}^{i+M} \ \sum_{\ell = k}^{j+k} \ B_{\ell}(t_\ell^{(k)}) - B_{\ell}(t_{\ell - 1}^{(k)}).
\end{equation}
\end{enumerate}
\end{lemma}

\begin{proof}
(i) By equation \eqref{eq:loctimeeq}, for all $k,M \in \mathbb{N}$, with $k \leq M$, and $t \in [0,\infty)$, 
\begin{equation} \label{eq:collapse}
\begin{split}
\sum_{j = 1}^M \sigma^j \{X_{j+k}(t) - x_{j+k}\} & = \sum_{j = 1}^M \sigma^j V_{j+k}(t) + \sum_{j = 1}^M \{\sigma^{j-1} q L_{(j+k-1,j+k)}(t) - \sigma^j q L_{(j+k,j+k+1)}(t)\} \\
& = \sum_{j = 1}^M \sigma^j V_{j+k}(t) + qL_{(k,k+1)}(t) - \sigma^M q L_{(M+k,M+k+1)}(t).
\end{split}
\end{equation}
Here in the first line we make the substitution $p = \sigma^{-1}q$ to obtain a telescoping sum, which we then evaluate in the second line. Letting $M \to \infty$, almost surely, the sum in the last line converges absolutely and uniformly on compact time intervals (follows by standard Gaussian tail bounds applied to the quantities $\sup_{s \in [0,T]}|V_{j+k}(s)|$), and the last term $\sigma^M q L_{(M+k,M+k+1)}(t)$ converges uniformly on compact sets to zero by \eqref{eq:cncstar2} and monotonicity. Thus, the following equality holds, almost surely, for all $t \in [0,\infty)$,
\begin{equation} \label{eq:locrep}
q L_{(k,k+1)}(t) = \sum_{j = 1}^\infty \sigma^j \{X_{j+k}(t) - x_{j+k}\} - \sum_{j = 1}^\infty \sigma^j V_{j+k}(t).
\end{equation}
For $M \in \mathbb{N}$, let $\{X_j^M\}_{1 \leq j \leq M}$ be defined as in the paragraph preceding Lemma \ref{lem:bd_by_fin}. By Lemma \ref{lem:bd_by_fin}\ref{it:bf2} and Lemma \ref{lem:lppcontrol}, we have 
\begin{equation} \label{eq:xjk_bd}
X_{j+k}(t) \leq X_{j+k}^{j+k}(t) \leq x_{j+k} + \cV_{j+k}^+(1,t), \quad t \in [0,\infty).
\end{equation}
Applying the bound \eqref{eq:xjk_bd} to each term of the first sum in \ref{eq:locrep}, we obtain, almost surely, for all $t \in [0,\infty)$,
\begin{equation} \label{eq:finalloc_bd}
qL_{(k,k+1)}(t) \leq \sum_{j = 1}^\infty \sigma^j\{\mathcal{V}^+_{j+k}(1,t) - V_{j+k}(t)\},
\end{equation}
and this proves \ref{it:lbds1}.

(ii) We obtain \ref{it:lbds2} as a consequence of the previous result as follows. Fix $k \in \mathbb{N}$, and recall the collision times $\tau_k$, defined by \eqref{eq:nextcol}, and the translation maps $\utr_t$ and $\ntr_t$, defined by \eqref{eq:translate}. For $j \in \mathbb{N}$, $t \in [0,\infty)$, define
\begin{equation} \label{eq:spacetimeshift}
\begin{split}
x_j' & = X_{j + k - 1}(\tau_k), \\
g_j' & = g_{j + k - 1}, \\
X_j'(t) & = \utr_{\tau_k} X_{j+k-1}(t) = X_{j + k - 1}(t + \tau_k), \\
B_j'(t) & = \ntr_{\tau_k}B_{j + k - 1}(t) = B_{j + k - 1}(t +\tau_k) - B_{j + k - 1}(\tau_k), \\
L_{(j,j+1)}'(t) & = \ntr_{\tau_k}L_{(j + k - 1,j + k)}(t) = L_{(j + k - 1,j + k)}(t + \tau_k) - L_{(j + k - 1,j + k)}(\tau_k), \\
L_{(0,1)}'(t) & = 0.
\end{split}
\end{equation}
Also, define
\begin{equation} \label{eq:shiftedVs}
\begin{split}
V'_j(t) & = g_j' t + B_j'(t) = V_{j+k-1}(t + \tau_k) - V_{j + k-1}(\tau_k) = \ntr_{\tau_k}V_{j+k-1}(t). 
\end{split}
\end{equation}
Note that on the time interval $[\tau_k,\tau_{k-1})$, the particle $X_k$ never collides with $X_{k-1}$, and thus 
$$
L_{(k-1,k)}(t) - L_{(k-1,k)}(\tau_k) = 0 \text{ for all } t \in [\tau_k,\tau_{k-1}]
$$
(we obtain the equality at the endpoint $t = \tau_{k-1}$ by continuity). Using this, one may show that, on the interval $[0,\tau_{k-1} - \tau_k]$, $\{X_j', j \in \mathbb{N}\}$ is an infinite system of competing Brownian particles with parameters $p,\{g_j'\}$, with initial data $\{x_j'\}$. Moreover, the associated collection of local times is $\{L_{(j-1,j)}', j \in \mathbb{N}\}$, and this collection satisfies the condition \eqref{eq:cncstar2}. To see this, note that since $\tau_k \leq T$, and the local times are non-negative and non-decreasing, for any $t \in [0,T)$,
$$
L'_{(j,j+1)}(t) = L_{(j+k-1,j+k)}(t + \tau_k) - L_{(j+k-1,j+k)}(\tau_k) \leq L_{j+k-1}(T + t),
$$
and $\sigma^M L_{M-1+k}(T + t) \to 0$ as $M \to \infty$ by \eqref{eq:cncstar2}. Consequently, we may apply the result \ref{it:lbds1} to the new system to obtain, almost surely, for $t \in [0, \tau_{k-1} - \tau_k]$, 
$$
L_{(k,k+1)}(t + \tau_k) - L_{(k,k+1)}(\tau_k) = L'_{(1,2)}(t) \leq q^{-1}\sum_{j = 1}^\infty \sigma^j\left\{\cV^+_{j+1}(k,[\tau_k,t+\tau_k]) - V'_{j+1}(t)\right\}.
$$
The last inequality follows on observing that, in the definition \eqref{eq:VMplusdef} of $\cV_M^+(i,t)$, if we replace each $V_j$ with $V_j'$, we obtain $\cV_M^+(i+k-1,[\tau_k,t + \tau_k])$. The above bound is equivalent to \eqref{eq:bdaftertau}, proving \ref{it:lbds2}.

(iii) By the result \ref{it:lbds2}, almost surely,
\begin{equation} \label{eq:sumdiff_lbd}
\begin{split}
\sum_{k = i}^{i+M} (L_{(k,k+1)}(\tau_{k-1}) - L_{(k,k+1)}(\tau_k)) & \leq q^{-1} \sum_{j = 1}^\infty \sigma^j \Bigg\{ \sum_{k = i}^{i+M}  \cV^+_{j+1}(k,[\tau_k,\tau_{k-1}]) \\
& \hspace{0.7in} - \sum_{k = i}^{i+M} \left( V_{j+k}(\tau_{k-1}) - V_{j+k}(\tau_k) \right) \Bigg\}.
\end{split}
\end{equation}
The right-hand side of \eqref{eq:sumdiff_lbd} is bounded above by 
\[
\begin{split}
q^{-1} \sum_{j = 1}^\infty \sigma^j \Bigg\{\sup_{\substack{0 = s_{i+M} \leq s_{i+M-1} \\
 \quad \leq \cdots \leq s_i \leq s_{i-1} = T}} & \sum_{k = i}^{i+M}  \cV^+_{j+1}(k,[s_k,s_{k-1}]) \\
& - \inf_{\substack{0 = s_{i+M} \leq s_{i+M-1} \\
 \quad \leq \cdots \leq s_i \leq s_{i-1} = T}} \sum_{k = i}^{i+M} \left( V_{j+k}(s_{k-1}) - V_{j+k}(s_k) \right)\Bigg\}.
\end{split}
\]
Notice that the second quantity appearing in the braces is by definition $\cV^-_{M+1}(j+i,T)$. It remains to show that the first quantity in the braces is bounded by $|\bar{g}|_\infty + \cU(j,M,T)$. Expanding the definition of $\cV^+_{j+1}(k,[s_k,s_{k-1}])$ gives us
\[
\begin{split}
\sup_{\substack{0 = s_{i+M} \leq s_{i+M-1} \\
 \quad \leq \cdots \leq s_i \leq s_{i-1} = T}} & \sum_{k = i}^{i+M}  \cV^+_{j+1}(k,[s_k,s_{k-1}]) \\
& = \sup_{\substack{0 = s_{i+M} \leq s_{i+M-1} \\
 \quad \leq \cdots \leq s_i \leq s_{i-1} = T}} \sum_{k = i}^{i+M} \ \sup_{\substack{s_k = t_{k-1}^{(k)} \leq t_k^{(k)}
 \leq \cdots \\ \quad\quad \leq t_{j+k-1}^{(k)} \leq t_{j+k}^{(k)} = s_{k-1}}} \sum_{\ell = k}^{j+k} \ (V_{\ell}(t_\ell^{(k)}) - V_{\ell}(t_{\ell - 1}^{(k)})) \\
& = \sup_{\{t_\ell^{(k)}\} \in D_{M,j}(i,T)} \sum_{k = i}^{i+M} \ \sum_{\ell = k}^{j+k} \ (V_{\ell}(t_\ell^{(k)}) - V_{\ell}(t_{\ell - 1}^{(k)})).
\end{split}
\]
Thus
\[
\begin{split}
\sup_{\substack{0 = s_{i+M} \leq s_{i+M-1} \\
 \quad \leq \cdots \leq s_i \leq s_{i-1} = T}} & \sum_{k = i}^{i+M}  \cV^+_{j+1}(k,[s_k,s_{k-1}]) \\
& \leq \sup_{\{t_\ell^{(k)}\} \in D_{M,j}(i,T)} \sum_{k = i}^{i+M} \ \sum_{\ell = k}^{j+k} \ \left\{|\bar g|_\infty(t_{\ell}^{(k)} - t_{\ell - 1}^{(k)}) + (B_{\ell}(t_{\ell}^{(k)}) - B_{\ell}(t_{\ell-1}^{(k)})) \right\} \\
& = |\bar g|_\infty T + \sup_{\{t_\ell^{(k)}\} \in D_{M,j}(i,T)} \sum_{k = i}^{i+M} \ \sum_{\ell = k}^{j+k} \ B_{\ell}(t_{\ell - 1}^{(k)}) - B_{\ell}(t_{\ell}^{(k)}) \\
& = |\bar g|_\infty T + \cU(j,M,T),
\end{split}
\]
as desired.
\end{proof}

\begin{lemma} \label{lem:Uconc}
There exists a constant $C$ such that, for any $j,M \in \mathbb{N}$, $T \in (0,\infty)$, and $\delta \in (0,\infty)$, 
\[
\mathbb{P}\left( \frac{\cU(j,M,T)}{\sqrt{(j+1)(M+1)T}} \geq 2 + \delta \right) \leq  C \exp( -(j+1)(M+1)(\delta^{3/2} \wedge \delta^3)/C).
\]
\end{lemma}

\begin{proof}
Fix $i,j,M \in \mathbb{N}$ and $T \in (0,\infty)$. We define centered Gaussian processes $G$ and $H$ on $D_{M,j}(i,T)$ as follows: For $\mathbf{t} = \{t_\ell^{(k)}\} \in D_{M,j}(i,T)$, 
\[
G(\mathbf{t}) := \sum_{k = i}^{i+M} \ \sum_{\ell = k}^{j+k} \ B_{\ell}(t_\ell^{(k)}) - B_{\ell}(t_{\ell - 1}^{(k)}),
\]
and
\[
H(\mathbf{t}) := \sum_{k = i}^{i+M} \ \sum_{\ell = k}^{j+k} \ \tilde B_{\ell}^{(k)}(t_\ell^{(k)}) - \tilde B_{\ell}^{(k)}(t_{\ell - 1}^{(k)}),
\]
where $\{\tilde B_\ell^{(k)}, i \leq k \leq i+M, k \leq \ell \leq j + k\}$ is a collection of i.i.d.\ standard Brownian motions, independent of the Brownian motions $\{B_\ell\}$.

We will prove the following inequality: For all real $u$,
\begin{equation} \label{eq:dist_comp}
\mathbb{P}\left( \sup_{\mathbf{t} \in D_{M,j}(i,T)} G(\mathbf{t}) \geq u \right) \leq \mathbb{P}\left( \sup_{\mathbf{t} \in D_{M,j}(i,T)} H(\mathbf{t}) \geq u \right).
\end{equation}
To see how the lemma may be obtained from this result, observe that 
\begin{equation} \label{eq:supG}
\sup_{\mathbf{t} \in D_{M,j}(i,T)} G(\mathbf{t}) = \cU(j,M,T).
\end{equation}
Furthermore,
\begin{equation} \label{eq:suptG}
\sup_{\mathbf{t} \in D_{M,j}(i,T)} H(\mathbf{t}) \stackrel{\mbox{\tiny{dist}}}{=} \cV^+_{(j+1)(M+1)}(i,T,\bar{0}).
\end{equation}
Here $\cV_M^+(i,T,\bar{0})$ denotes the quantity \eqref{eq:VMplusdef} with $\bar{g} = \bar{0}$ (equivalently, with $V_j = B_j$ for $j \in \mathbb{N}$). The inequality \eqref{eq:dist_comp} therefore implies that, for all $\delta > 0$,
\begin{equation} \label{eq:dist_comp2}
\mathbb{P}\left( \frac{\cU(j,M,T)}{\sqrt{(j+1)(M+1)T}} \geq 2 + \delta \right) \leq \mathbb{P}\left( \frac{\cV^+_{(j+1)(M+1)}(i,T,\bar{0})}{\sqrt{(j+1)(M+1)T}} \geq 2 + \delta \right).
\end{equation}
By the proof of Lemma \ref{lem:lpp}\ref{it:lpp1} (see \eqref{eq:VMplusbd}), we also have the following bound:
\begin{equation} \label{eq:cheby}
\mathbb{P}\left( \left| \frac{\cV^+_{(j+1)(M+1)}(i,T,\bar{0})}{\sqrt{(j+1)(M+1)T}} - 2\right| \geq \delta \right) \leq c_1 \exp( -(j+1)(M+1)(\delta^{3/2} \wedge \delta^3)/c_1),
\end{equation}
where $c_1 \in (0,\infty)$ is a universal constant. The result then follows by combining the bounds \eqref{eq:dist_comp2} and \eqref{eq:cheby}.

We now set out to prove \eqref{eq:dist_comp}. By Slepian's Inequality \cite[Theorem 2.2.1]{AdlerTaylor}, it suffices to show that for all $\mathbf{s} = \{s_\ell^{(k)}\}$ and $\mathbf{t} = \{t_\ell^{(k)}\}$ in $D_{M,j}(i,T)$,  
\begin{equation} \label{eq:samevar}
\mathbb{E}[G(\mathbf{t})^2] = \mathbb{E}[H(\mathbf{t})^2],
\end{equation}
and 
\begin{equation} \label{eq:L2_comp}
\mathbb{E}[(G(\mathbf{s}) - G(\mathbf{t}))^2] \leq \mathbb{E}[(H(\mathbf{s}) - H(\mathbf{t}))^2].
\end{equation}
We compute 
\begin{equation} \label{eq:sigmasq}
\begin{split}
\mathbb{E}[G(\mathbf{t})^2]
& = \mathbb{E}\left[\left( \sum_{k = i}^{i+M} \sum_{\ell = k}^{j + k} \ B_{\ell}(t_\ell^{(k)}) - B_{\ell}(t_{\ell - 1}^{(k)})\right)^2\right] \\
& = \sum_{k = i}^{i+M} \sum_{\ell = k}^{j + k} \mathbb{E}\left[\left( B_{\ell}(t_\ell^{(k)}) - B_{\ell}(t_{\ell - 1}^{(k)})\right)^2\right] \\
&= \sum_{k = i}^{i+M} \sum_{\ell = k}^{j + k} \mathbb{E}\left[\left( \tilde B^{(k)}_{\ell}(t_\ell^{(k)}) - \tilde B^{(k)}_{\ell}(t_{\ell - 1}^{(k)})\right)^2\right]\\
&= \mathbb{E}[H(\mathbf{t})^2],
\end{split}
\end{equation}
where the last two lines follow on recalling the properties of the collection $\{\tilde B_\ell^{(k)}, i \leq k \leq i+M, k \leq \ell \leq j + k\}$. This proves \eqref{eq:samevar}.
We also compute
\begin{equation} \label{eq:Hidff}
\begin{split}
\mathbb{E}[(H(\mathbf{s}) & - H(\mathbf{t}))^2] \\
& = \mathbb{E}\left[ \left( \sum_{k = i}^{i+M} \ \sum_{\ell = k}^{j+k} \ \tilde B_{\ell}^{(k)}(s_\ell^{(k)}) - \tilde B_{\ell}^{(k)}(s_{\ell - 1}^{(k)}) - \tilde B_{\ell}^{(k)}(t_\ell^{(k)}) + \tilde B_{\ell}^{(k)}(t_{\ell - 1}^{(k)}) \right)^2 \right] \\
& = \sum_{k = i}^{i+M} \ \sum_{\ell = k}^{j+k} \mathbb{E}\left[ \left( \tilde B_{\ell}^{(k)}(s_\ell^{(k)}) - \tilde B_{\ell}^{(k)}(s_{\ell - 1}^{(k)}) - \tilde B_{\ell}^{(k)}(t_\ell^{(k)}) + \tilde B_{\ell}^{(k)}(t_{\ell - 1}^{(k)}) \right)^2 \right] \\
& = \sum_{k = i}^{i+M} \ \sum_{\ell = k}^{j+k} (s_\ell^{(k)} - s_{\ell - 1}^{(k)}) + (t_\ell^{(k)} - t_{\ell - 1}^{(k)}) - 2|[s_{\ell-1}^{(k)},s_\ell^{(k)}] \cap [t_{\ell-1}^{(k)},t_\ell^{(k)}]| \\
& = \sum_{k = i}^{i+M} \ \sum_{\ell = k}^{j+k} |[s_{\ell-1}^{(k)},s_\ell^{(k)}] \triangle [t_{\ell-1}^{(k)},t_\ell^{(k)}]|,
\end{split}
\end{equation}
where $\triangle$ denotes symmetric difference. The corresponding calculation for $G(\mathbf{t})$ is more complicated due to the additional dependencies among the terms. Interchanging the order of summation in the definition of $G(\mathbf{t})$, we may write
\[
G(\mathbf{t}) = \sum_{\ell = i}^{j+i+M} \ \sum_{k = (\ell - j) \vee i}^{\ell \wedge (i+M)} \ (B_{\ell}(t_\ell^{(k)}) - B_{\ell}(t_{\ell - 1}^{(k)})).
\]
Hence, with $\mathbf{u} \in D_{M,j}(i,T)$, denoting 
$\Delta^{(k)}_{\ell}(\mathbf{u}) = \ B_{\ell}(u_\ell^{(k)}) - B_{\ell}(u_{\ell - 1}^{(k)})$,
\begingroup
\allowdisplaybreaks
\begin{align*}
&\mathbb{E}[(G(\mathbf{s}) - G(\mathbf{t}))^2] = \mathbb{E}\left[\left( \sum_{\ell = i}^{j+i+M} \ \sum_{k = (\ell - j) \vee i}^{\ell \wedge (i+M)} \ (\Delta^{(k)}_{\ell}(\mathbf{s}) -\Delta^{(k)}_{\ell}(\mathbf{t}))\right)^2\right] \\
& = \sum_{\ell = i}^{j+i+M} \mathbb{E}\left[\left( \sum_{k = (\ell - j) \vee i}^{\ell \wedge (i+M)} \ (\Delta^{(k)}_{\ell}(\mathbf{s}) -\Delta^{(k)}_{\ell}(\mathbf{t}))\right)^2 \right] \\
& = \sum_{\ell = i}^{j+i+M} \Bigg\{ \sum_{k = (\ell - j) \vee i}^{\ell \wedge (i+M)} \mathbb{E}\left[\left( \Delta^{(k)}_{\ell}(\mathbf{s}) -\Delta^{(k)}_{\ell}(\mathbf{t}) \right)^2\right] \\
& \hspace{0.4in} + 2\cdot \sum_{(\ell - j) \vee i \leq k < m \leq \ell \wedge (i+M)} \mathbb{E}\left[ \left( \Delta^{(k)}_{\ell}(\mathbf{s}) -\Delta^{(k)}_{\ell}(\mathbf{t}) \right)  \left( \Delta^{(m)}_{\ell}(\mathbf{s}) -\Delta^{(m)}_{\ell}(\mathbf{t}) \right) \right] \Bigg\} \\
& = \sum_{\ell = i}^{j+i+M} \Bigg\{ \sum_{k = (\ell - j) \vee i}^{\ell \wedge (i+M)} |[s_{\ell-1}^{(k)},s_\ell^{(k)}] \triangle [t_{\ell-1}^{(k)},t_\ell^{(k)}]| \\
& \hspace{0.4in}+ 2 \cdot \sum_{(\ell - j) \vee i \leq k < m \leq \ell \wedge (i+M)} \left( \mathbb{E}\left[ -\Delta^{(k)}_{\ell}(\mathbf{s})\Delta^{(m)}_{\ell}(\mathbf{t}) \right]  + \mathbb{E}\left[ -\Delta^{(m)}_{\ell}(\mathbf{s})\Delta^{(k)}_{\ell}(\mathbf{t}) \right]  \right) \Bigg\} \\
& = \sum_{\ell = i}^{j+i+M} \Bigg\{ \sum_{k = (\ell - j) \vee i}^{\ell \wedge (i+M)} |[s_{\ell-1}^{(k)},s_\ell^{(k)}] \triangle [t_{\ell-1}^{(k)},t_\ell^{(k)}]| \\
& \hspace{0.4in} - 2 \cdot \sum_{(\ell - j) \vee i \leq k < m \leq \ell \wedge (i+M)} \left(|[s_{\ell-1}^{(k)},s_{\ell}^{(k)}] \cap [t_{\ell-1}^{(m)},t_{\ell}^{(m)}]| + |[s_{\ell-1}^{(m)},s_{\ell}^{(m)}] \cap [t_{\ell-1}^{(k)},t_{\ell}^{(k)}]|\right) \Bigg\}.
\end{align*}
\endgroup

Noting that the second quantity in the last expression is nonpositive, we obtain the bound
\begin{equation} \label{eq:Gdiffbd}
\begin{split}
\mathbb{E}[(G(\mathbf{s}) - G(\mathbf{t}))^2] & \leq \sum_{\ell = i}^{j+i+M} \sum_{k = (\ell - j) \vee i}^{\ell \wedge (i+M)} |[s_{\ell-1}^{(k)},s_\ell^{(k)}] \triangle [t_{\ell-1}^{(k)},t_\ell^{(k)}]| \\
& = \sum_{k = i}^{i+M} \sum_{\ell = k}^{j + k} |[s_{\ell-1}^{(k)},s_\ell^{(k)}] \triangle [t_{\ell-1}^{(k)},t_\ell^{(k)}]|.
\end{split}
\end{equation}
The equality \eqref{eq:Hidff} and the bound \eqref{eq:Gdiffbd} imply \eqref{eq:L2_comp}. This completes the proof of \eqref{eq:dist_comp} and the lemma.
\end{proof}

\begin{proof}[Proof of Theorem \ref{thm:unique3}]
Let $p>q$. Consider a strong solution $\{X_j\}$ of \eqref{eq:loctimeeq}, with $N = \infty$, with driving noises $\{B_j\}$ and initial distribution $\gamma$. We denote $\{X_j(0)\} = \{x_j\} = \bar x$ and suppose that \eqref{eq:init_cond2} holds and suppose that $\bar{g} \in \ell^\infty(\mathbb{N})$. Associated local times are denoted as $\{L_{(j,j+1)}\}$ which are assumed to satisfy the condition \eqref{eq:cncstar2}.
Fix $T>0$. In view of Lemma \ref{lem:unique}, it suffices to show that for all $i \in \mathbb{N}$  $K^*(i,T,\{X_j\}) < \infty$ a.s.

 Consider a sequence $t_{i+M} = 0 \leq t_{i +M - 1} \leq t_{i+M-2} \leq \cdots \leq t_i \leq t_{i-1} = T$. By \eqref{eq:loctimeeq}, for $i \leq k \leq i + M - 1$, 
\begin{equation} \label{eq:xdiffbd}
\begin{split}
X_k(t_{k-1}) - X_k(t_k) & = V_k(t_{k-1}) - V_k(t_k) + p(L_{(k-1,k)}(t_{k-1}) - L_{(k-1,k)}(t_k)) \\
& \hspace{1in} - q(L_{(k,k+1)}(t_{k-1}) - L_{(k,k+1)}(t_k)) \\
& \geq V_k(t_{k-1}) - V_k(t_k) - q(L_{(k,k+1)}(t_{k-1}) - L_{(k,k+1)}(t_k)),
\end{split}
\end{equation}
here noting that $L_{(k-1,k)}$ is increasing to obtain the inequality. Summing over both sides of \eqref{eq:xdiffbd} and rearranging terms yields
\begin{equation} \label{eq:sumdiffbd}
\begin{split}
& q \sum_{k = i}^{i + M} \left( L_{(k,k+1)}(t_{k-1}) - L_{(k,k+1)}(t_k) \right) \geq - \sum_{k = i}^{i + M} \left(X_k(t_{k-1}) - X_k(t_k)\right) \\
& \hspace{2.8in} + \sum_{k = i}^{i + M} \left( V_k(t_{k-1}) - V_k(t_k) \right) \\
& \hspace{1in} \geq -X_i(T) + x_{i + M - 1} - \sum_{k = i}^{i+M-1} \left( X_{k+1}(t_k) - X_k(t_k) \right) + \cV_{M+1}^-(i,T).
\end{split}
\end{equation}
where to obtain the second line, we rearrange terms in the first sum, and we bound the second sum  by the quantity \eqref{eq:VMdef}.

Recall the quantity $\cU(j,M,T)$ from Lemma \ref{lem:localtimebds}\ref{it:lbds3}. We will next prove the following claim:

\vspace{0.1in}

\noindent\textbf{Claim:} For all $M \in \mathbb{N}$, on the event $K^*(i,T) \geq M$, the following bound holds almost surely:
\begin{equation} \label{eq:init_ubd}
x_{i+M-1} \leq X_i(T) - \cV^-_{M+1}(i,T) + \sum_{j = 1}^\infty \sigma^j \left\{|\bar{g}|_\infty T + \cU(j,M,T) - \cV_{M+1}^-(j+i,T) \right\}.
\end{equation}

To prove the claim, assume that $K^*(i,T) \geq M$. We must have $X_{k+1}(\tau_k) = X_k(\tau_k)$ for $i \leq k \leq i + M - 1$ (i.e.\ the infimum appearing in the right-hand side of \eqref{eq:nextcol} is bounded above by $T$). Indeed, if $K^*(i,T) \geq M$, then there exists a sequence $0 \leq s_{i+M-1} \leq \cdots \leq s_i \leq T$ such that $X_k(s_k) = X_{k+1}(s_k)$ for $i \leq k \leq i + M - 1$; consequently, the infimum in \eqref{eq:nextcol} is bounded above by $s_k \leq T$ for $i \leq k \leq i + M - 1$. Setting $t_k = \tau_k$ for $i-1 \leq k \leq i + M$, we see that the  sum in the last line of \eqref{eq:sumdiffbd} vanishes, and thus
\begin{equation} \label{eq:sumdiffbd2}
q \sum_{k = i}^{i + M} \left( L_{(k,k+1)}(\tau_{k-1}) - L_{(k,k+1)}(\tau_k) \right)  \geq -X_i(T) + x_{i + M - 1} + \cV^-_{M+1}(i,T).  
\end{equation}
Applying the bound from Lemma \ref{lem:localtimebds}\ref{it:lbds3} to the left-hand side and rearranging terms yields \eqref{eq:init_ubd}, proving the claim.

By Lemma \ref{lem:lpp}\ref{it:lpp1}, for any $\varepsilon_0 \in (0,\infty)$, there exists $C = \tilde C(\varepsilon_0,T) \in (0,\infty)$ such that the following bound holds: For all $\varepsilon \in [\varepsilon_0,\infty)$, and $i, j,M \in \mathbb{N}$,
\begin{equation} \label{eq:probbds1}
\mathbb{P}\left( \left| \frac{\cV^-_{M+1}(i+j,T)}{\sqrt{(M+1)T}} + 2 \right| \geq \varepsilon \right) \leq C \exp(-(M+1)(\varepsilon^{3/2} \wedge \varepsilon^3)/C).
\end{equation}
Furthermore, by Lemma \ref{lem:Uconc}, there exists $\tilde{C} \in (0,\infty)$ such that for all  $j,M \in \mathbb{N}$,
\begin{equation} \label{eq:probbds2}
\mathbb{P}\left( \frac{|\cU(j,M,T)|}{\sqrt{(j+1)(M+1)T}} \geq 2 + \varepsilon \right) \leq  \tilde C \exp( -(j+1)(M+1)(\varepsilon^{3/2} \wedge \varepsilon^3)/\tilde C).
\end{equation}
To understand the second bound, recall \eqref{eq:supG}. Note that $G(\mathbf{t}) \stackrel{\text{dist}}{=} -G(\mathbf{t})$, and consequently $\cU(j,M,T) = \sup_{\mathbf{t}} G(\mathbf{t})$ stochastically dominates $-\cU(j,M,T)$. Therefore, the same bound from Lemma \ref{lem:Uconc} holds if we replace $\cU(j,M,T)$ with $-\cU(j,M,T)$, and we then obtain \eqref{eq:probbds2} through a union bound.

In \eqref{eq:probbds1} taking $\varepsilon = \varepsilon_0(j+1)$ for some fixed $\varepsilon_0>0$, and using the Borel-Cantelli lemma we see that, for a.e. $\omega$, there exist $M_0, j_0 \in \NN$ such that for all $M \ge M_0$ and $j \ge j_0$
$$\frac{|\cV_{M+1}^-(j+i,T)|}{(M+1)^{1/2}} \le \sqrt{T}(2+\varepsilon_0(j+1)).
$$
Similarly, applying the Borel-Cantelli lemma in \eqref{eq:probbds2}, we see that for a.e. $\omega$, there exist $M_1, j_1 \in \NN$ such that for all $M \ge M_1$ and $j \ge j_1$
$$\frac{|\cU(j,M,T)|}{(M+1)^{1/2}} \le \sqrt{T}(2+\varepsilon_0(j+1)).
$$
A similar use of the Borel-Cantelli lemma also show that for all fixed $i,j\in \NN$,
$$\limsup_{M\to \infty} \frac{|\cU(j,M,T)|+|\cV_{M+1}^-(j+i,T)|}{(M+1)^{1/2}}  <\infty.$$
Consequently, we have, a.s., 
\begin{equation} \label{eq:RHSlim}
\begin{split}
& \limsup_{M \to \infty} \left\{ - \frac{\cV^-_{M+1}(i,T)}{\sqrt{M+1}} + \sum_{j = 1}^\infty \sigma^j \frac{1}{\sqrt{M+1}}\left\{|\bar{g}|_\infty T + \cU(j,M,T) - \cV_{M+1}^-(j+i,T) \right\}\right\} \\
&< \infty.
\end{split}
\end{equation}

\noindent Combining the claim and \eqref{eq:RHSlim}, we obtain
\[
\begin{split}
\mathbb{P}\left( K^*(i,T) = \infty \right) & \leq \mathbb{P}\left( \text{$\forall M \in \mathbb{N}$}, \frac{x_{i+M-1}}{\sqrt{M}} \leq \frac{1}{\sqrt{M}} \cdot \text{RHS of \eqref{eq:init_ubd}} \right) \\
& \leq \mathbb{P}\left( \limsup_{M \to \infty} \frac{x_{i+M - 1}}{\sqrt{M}} < \infty \right) = 0,
\end{split}
\]
where the last equality follows by the assumption \eqref{eq:init_cond2} on the initial data. This concludes the proof.
\end{proof}

\section{Approximative solutions} \label{ssec:approximative}

Recall that Theorems \ref{thm:unique2} and \ref{thm:unique3} establish pathwise uniqueness under different conditions on the local times, namely \eqref{eq:cncstar} and \eqref{eq:cncstar2}, respectively. The goal of this section is to show that these conditions are satisfied for the most commonly investigated type of strong solution of \eqref{eq:loctimeeq}, known as the {\em (strong) approximative version}. 
Specifically, we will verify that the approximative version of the system \eqref{eq:loctimeeq} satisfies \eqref{eq:cncstar} if $p \ge  q$ and \eqref{eq:init_cond1} holds,  and it satisfies \eqref{eq:cncstar2} if $p > q$. We begin by  recalling from \cite{sarantsev2017infinite} the definition of an approximative version of an infinite system of competing Brownian particles. Fix parameters $p, \bar g$, $p \ge q$, and let a filtered probability space, Brownian motions $\{B_j\}_{j \in \mathbb{N}}$, and initial data  $\bar{x} = \{x_j\}_{j \in \mathbb{N}}$ be as at the beginning of Section \ref{sec:sepu}.

Given $M \in \mathbb{N}$, let $\bar{x}^M = \{x_j\}_{1 \leq j \leq M}$, $\bar{g}^M = \{g_j\}_{1 \leq j \leq M}$, and $\bar{B}^M = \{B_j\}_{1 \leq j \leq M}$. We let $\bar{X}^M = \{X_j^M\}_{1 \leq j \leq M}$ denote the solution to the system \eqref{eq:loctimeeq} with $N = M$ particles, relative to the initial conditions $\bar{x}^M$, drifts $\bar{g}^M$, and Brownian motions $\bar{B}^M$. We say that a solution to the infinite system $\bar{X} = \{X_j\}_{j \in \mathbb{N}}$ is a \textit{(strong) approximative version} if, for all $T \in (0,\infty)$ and $j, M_0 \in \mathbb{N}$,
\begin{equation} \label{eq:uniform}
\lim_{M \geq M_0, M \to \infty} \sup_{s \in [0,T]} |X_j(s) - X_j^M(s)| = 0 \text{ a.s.}
\end{equation}
Note that given the initial conditions, drifts, and Brownian motions, there is at most one approximative version by uniqueness of solutions to the finite-dimensional system. The paper \cite[Theorem 3.7]{sarantsev2017infinite} proves that approximative versions exist when $p \geq q$, the initial data $\bar x$ satisfy \eqref{eq:scon}, and the drift parameters satisfy $\inf_{n \in \mathbb{N}} g_n > -\infty$.

Let $\bar{L} = \{L_{(j-1,j)}\}_{j \in \mathbb{N}}$ and $\bar{L}^M = \{L_{(j-1,j)}^M\}_{1 \leq j \leq M+1}$ denote the collections of local times for the infinite approximative version solution $\bar{X}$ and the finite system $\bar{X}^M$, respectively. Note that by \eqref{eq:loctimeeq}, if \eqref{eq:uniform} holds, then the for each $j$, $L_{(j-1,j)}^M$ converges uniformly on compact intervals to $L_{(j-1,j)}$ almost surely.

The following proposition shows that the local times associated with approximative version solutions satisfy the growth conditions in Section \ref{sec:sepu}. The proofs for the two parts of the proposition require somewhat different arguments.

\begin{proposition} \label{prop:lcond}
Let $\bar{X} = \{X_j\}$ be an approximative version of the system \eqref{eq:loctimeeq}, and with initial conditions, drift parameters, and Brownian motions as above. 

\begin{enumerate}[label = (\roman*)]

\item \label{it:lcond1} If $p \ge q$, and the initial conditions $\bar x$ satisfy \eqref{eq:init_cond1}, then the collection of local times $\bar{L}$ satisfies condition \eqref{eq:cncstar}.

\item \label{it:lcond2} If $p > q$, then the collection of local times $\bar{L}$ satisfies the condition \eqref{eq:cncstar2}.

\end{enumerate}
\end{proposition}

\begin{proof}[Proof of Proposition \ref{prop:lcond}\ref{it:lcond1}]
Assume the conditions in the statement of part (i). Since the driving Brownian motions are independent of the initial conditions $\{x_j\}_{j \in \mathbb{N}}$, we may assume without loss of generality that the initial conditions are deterministic. Note that by \eqref{eq:loctimeeq}, the condition \eqref{eq:cncstar} is equivalent to the following: For every $T \in (0,\infty)$
$$
\limsup_{M \to \infty} \sup_{s \in [0,T]} \frac{x_M + B_M(s) - X_M(s)}{x_M} \leq 0 \text{ a.s.}
$$
In fact, since $x_M$ satisfies Condition \eqref{eq:init_cond1}, standard Gaussian tail bounds show that 
$$
\limsup_{M \to \infty} \sup_{s \in [0,T]} B_M(s)/x_M = 0 \text{ a.s.,}
$$ 
and it is enough to check that 
$$
\limsup_{M \to \infty} \sup_{s \in [0,T]} \frac{x_M - X_M(s)}{x_M} \leq 0 \text{ a.s.}
$$
By the Borel-Cantelli Lemma, it is sufficient to show that for all $\varepsilon > 0$, there exists $M_0 \in \mathbb{N}$ such that
\begin{equation} \label{eq:TPfinsum}
\sum_{M \geq M_0} \mathbb{P}\left( \sup_{s \in [0,T]} \frac{x_M - X_M(s)}{x_M} \geq \varepsilon \right) < \infty.
\end{equation}
Since $\bar{X}$ is an approximative version, we have 
\begin{equation}\label{eq:1250}
\mathbb{P}\left( \sup_{s \in [0,T]} \frac{x_M - X_M(s)}{x_M} \geq \varepsilon \right)  = \lim_{N \to \infty} \mathbb{P}\left( \sup_{s \in [0,T]} \frac{x_M - X_M^N(s)}{x_M} \geq \varepsilon \right).
\end{equation}
Denote by $\{\tilde X_j^M\}_{1 \leq j \leq M}$  the solution to the system \eqref{eq:loctimeeq} with $N = M$ particles, relative to the initial conditions $\bar{x}^M$, drifts $\bar{g}^M$,  Brownian motions $\bar{B}^M$ and collision parameters $p=q=1/2$. From \cite[Theorem 3.2]{AS2}, almost surely, for all $N \in \NN$, $1 \le j \le N$, and $t \ge 0$, $\tilde X^N_j(t) \le X^N_j(t)$. Thus, for $N \ge M$,
\begin{equation}\label{eq:1251}
\mathbb{P}\left( \sup_{s \in [0,T]} \frac{x_M - X_M^N(s)}{x_M} \geq \varepsilon \right) \le \mathbb{P}\left( \sup_{s \in [0,T]} \frac{x_M - \tilde X_M^N(s)}{x_M} \geq \varepsilon \right).
\end{equation}
We now briefly review some facts about (finite) rank-based diffusions, in connection with the system \eqref{eq:loctimeeq} with finitely many particles. Let $\bar{W} = \{W_j\}_{j \in \mathbb{N}}$ be an i.i.d.\ collection of Brownian motions, and suppose that, for some $M \in \NN$, $\bar{Y}^M = \{Y_j\}_{1 \leq j \leq M}$ is a collection of continuous processes which satisfy the system of equations
\begin{equation} \label{eq:rbd}
Y_j(t) = x_j + \int_0^t g_{\pi_{\bar{Y}^M(s)}(j)} \dd s + W_j(t), \quad 1 \leq j \leq M, \quad t \in [0,\infty).
\end{equation}
Here, for each $\bar{y} \in \mathbb{R}^M$, $\pi_{\bar{y}} : [M] \to [M]$ is the unique permutation such that $y_{\pi_{\bar{y}}(i)} \leq y_{\pi_{\bar{y}}(j)}$ whenever $1 \leq i < j \leq M$, and ties are broken by lexicographical ordering. The existence of a strong solution to \eqref{eq:rbd} and pathwise uniqueness of solutions follows from \cite[Theorem 2]{ichiba2013strong}. Furthermore, by Proposition 3 in the same paper, the ranked trajectories 
$$
\hat{X}_k(t) := Y_{\pi_{\bar{Y}^M(t)}(k)}(t), \quad \quad 1 \leq k \leq M, \quad t \in [0,\infty),
$$
solve \eqref{eq:loctimeeq} with $N = M$, initial conditions $\bar{x}^M$, drifts $\bar{g}^M$, and Brownian motions given as follows:
$$
B_j(t) = \int_0^t \sum_{k = 1}^M \mathbf{1}\{\pi_{\bar{Y}^M(s)}(j) = k\} \dd W_k(s), \quad \quad 1 \leq j \leq M, \quad t \in [0,\infty).
$$
By uniqueness in distribution of solutions to \eqref{eq:loctimeeq} in the finite-dimensional case, we have 
\begin{equation} \label{eq:ed_rtraj}
\{\tilde{X}_k\}_{1 \leq k \leq M} \stackrel{d}{=} \{\hat{X}_k\}_{1 \leq k \leq M}.
\end{equation}
Combining this fact with \eqref{eq:1250} and \eqref{eq:1251},
we have 
\begin{equation}\label{eq:103}
\mathbb{P}\left( \sup_{s \in [0,T]} \frac{x_M - X_M(s)}{x_M} \geq \varepsilon \right) 
 \le \lim_{N \to \infty} \mathbb{P}\left( \sup_{s \in [0,T]} \frac{x_M - \hat{X}_M^N(s)}{x_M} \geq \varepsilon \right).
\end{equation}
We bound the last quantity as follows. First, observe that for $1 \leq M \leq N < \infty$,
\[
\hat{X}_M^N(t) \geq \min_{M \leq j \leq N} Y_j(t) \geq \min_{M \leq j \leq N} \left( x_j - |\bar{g}|_\infty t + W_j(t) \right).
\]
Consequently, for $M$ sufficiently large, 
\[
\begin{split}
\mathbb{P}\left( \sup_{s \in [0,T]} \frac{x_M - \hat{X}_M^N(s)}{x_M} \geq \varepsilon \right) & \leq \mathbb{P}\left( \frac{|\bar{g}|_\infty T}{x_M} + \max_{M \leq j \leq N} \frac{x_M - x_j - W^*_j(T)}{x_M} \geq \varepsilon \right) \\
& \leq \mathbb{P}\left( \max_{M \leq j \leq N} \frac{x_M - x_j - W^*_j(T)}{x_M} \geq \varepsilon' \right),
\end{split}
\]
where $W_j^*(T) = \inf_{s \in [0,T]} W_j(s)$ and $\varepsilon' = \varepsilon/2$. The last quantity above is equal to
\[
\begin{split}
& 1 - \mathbb{P}\left( \max_{M \leq j \leq N} \frac{x_M - x_j - W^*_j(T)}{x_M} < \varepsilon' \right) \\
& \leq 1 - \prod_{M \leq j \leq N} \mathbb{P}\left( \frac{x_M - x_j - W^*_j(T)}{x_M} < \varepsilon' \right) \\
& = 1 - \prod_{M \leq j \leq N} \mathbb{P}\left( \frac{-W^*_j(T)}{\sqrt{T}} < \frac{x_j - (1 - \varepsilon')x_M}{\sqrt{T}} \right) \\
& \leq 1 - \prod_{M \leq j \leq N} \left( 1 - c_1 \exp( -c_2 T^{-1}(x_j - (1 - \varepsilon')x_M)^2 ) \right) \\
& \leq 1 - \prod_{M \leq j \leq N} \left( 1 - c_1 \exp( -c_3 T^{-1} \varepsilon^2 x_j^2 ) \right),
\end{split}
\]
where $c_1, c_2, c_3$ are positive constants, the second to last line follows by standard Gaussian tail bounds, and the last line follows by noting that $x_j \geq x_M$. By the assumption on the initial conditions, there exist $c_4>0$ and  $M_0 \in \NN$ such that for $j \geq M_0$, $x_j \geq c_4 \sqrt{j}$. Hence, for $M$ sufficiently large, we obtain the following upper bound for the quantity in the last line above: 
\[
\begin{split}
& 1 - \prod_{M \leq j \leq N} \left( 1 - c_1 \exp(-c_5 T^{-1}\varepsilon^2 j) \right) \\
& = 1 - \exp( \sum_{M \leq j \leq N} \log(1 - c_1 \exp(-c_5 T^{-1}\varepsilon^2 j)) ) \\
& \leq 1 - \exp( \sum_{M \leq j \leq N} - 2c_1\exp(-c_5 T^{-1}\varepsilon^2 j)) \\
& \leq 1 - \exp\left( -c_6 \exp(-c_5 T^{-1}\varepsilon^2 M) \right) \\
& \leq c_6 \exp(-c_5 T^{-1}\varepsilon^2 M).
\end{split}
\]
Here the first inequality follows because $\log(1 - a) \geq -2a$ for $0 \leq a \leq 1/2$, and the last inequality follows since $1 - e^{-a} \leq a$ for $a > 0$. As the last quantity is summable in $M$, this, in view of \eqref{eq:103}, proves \eqref{eq:TPfinsum} and completes the proof of the first part of Proposition \ref{prop:lcond}.
\end{proof}

\begin{proof}[Proof of Proposition \ref{prop:lcond}\ref{it:lcond2}] 
Let $\bar{X}^M = \{X_j^M\}_{1 \leq j \leq M}$ denote the approximating sequence as in the definition \eqref{eq:uniform} for the approximative version $\bar{X}$, and let $\bar{L}^M = \{L_{(j,j+1)}^M\}_{1 \leq j \leq M-1}$ denote the associated collection of local times. Recall that as $M \to \infty$, almost surely these local times will converge uniformly on compact intervals to the corresponding local times for the infinite system $\bar{X}$. By equation \eqref{eq:loctimeeq}, for all $k,M \in \mathbb{N}$, with $k \leq M$ and $t \in \mathbb{N}$, 
\begin{equation} \label{eq:collapse2}
\begin{split}
\sum_{j = k+1}^M \sigma^j \{X_j^M(t) - x_j\} & = \sum_{j = k+1}^M \sigma^j V_j(t) + \sum_{j = k+1}^M \{\sigma^{j-1} q L_{(j-1,j)}^M(t) - \sigma^j q L_{(j,j+1)}^M(t)\} \\
& = \sum_{j = k+1}^M \sigma^j V_j(t) + \sigma^k qL_{(k,k+1)}^M(t),
\end{split}
\end{equation}
here recalling that $L_{(M,M+1)} \equiv 0$. Note that by the bounds in Lemma \ref{lem:bd_by_fin}\ref{it:bf1} and \ref{it:bf3} and in Lemma \ref{lem:lppcontrol}, for each $j \leq M$,
\begin{equation} \label{eq:pushup_bd}
X_j^M(t) \leq X_j^j(t) \leq x_j + \tilde{X}_j^j(t) \leq x_j + \cV_j^+(1,t).
\end{equation}
Rearranging terms in \eqref{eq:collapse2} and applying the bound \eqref{eq:pushup_bd} yields
\begin{equation} \label{eq:finltbd}
\sigma^k q L_{(k,k+1)}^M(t) \leq \sum_{j = k+1}^M \sigma^j \{\cV_j^+(1,t) - V_j(t)\}.
\end{equation}
With probability 1, the sum
$$
\sum_{j = 1}^\infty \sigma^j \{\cV_j^+(1,t) - V_j(t)\}
$$
converges absolutely. This fact may easily be obtained from the probability bounds in Lemma \ref{lem:lpp} for $\cV_j^+(1,t)$ and standard Gaussian tail bounds applied to the Brownian motion $V_j(t)$. (Alternatively, one can use the following more elementary bound for $\cV_j^+(1,t)$:
$
|\cV_j^+(1,T)| \leq \sum_{k = 1}^j \sup_{t \in [0,T]} 2|B_k(t)|,
$
and then appeal to Gaussian tail bounds to control the terms.) Taking the limit as $M \to \infty$ in \eqref{eq:finltbd}, we obtain
\[
\sigma^k q L_{(k,k+1)}(t) \leq \sum_{j = k+1}^\infty \sigma^j \{\cV_j^+(1,t) - V_j(t)\}.
\]
Since the right-hand side converges to zero as $k \to \infty$, the proposition follows from this bound.
\end{proof}

\section{Proof of Theorem \ref{thm:exist}}\label{proofexist}

To establish the existence result Theorem \ref{thm:exist}, we will first prove the theorem for when the drift parameters $g_j$ are equal to zero (or equivalently $V_j = B_j$) for all $j \in \mathbb{N}$. To obtain the general case we will apply Girsanov's Theorem.

Fix a collection of independent standard Brownian motions $\{B_j(t), t \in [0,\infty), j \in \mathbb{N}\}$ given on a filtered probability space $(\Omega, \clf, \{\clf_t\}_{t\ge 0}, \PP)$, satisfying the assumptions stated in \S\ref{sec:sepu}. Also fix $\clf_0$-measurable initial data $\bar{x} = \{x_j\}_{j \in \mathbb{N}}$. The following explicit solution to the infinite system in the case $(p,q) = (0,1)$ is presented in \cite[\S2.3.1]{WFSbook}. For $n,k \in \mathbb{N}$ with $k \geq n$, let 
\[
\tilde{Y}_{n,k}(t) = \inf_{0 \leq s_{k-1} \leq s_{k-2} \leq \cdots \leq s_n \leq t} \sum_{j = n}^{k} (B_j(s_{j-1}) - B_j(s_j)),
\]
where, in the sum, $s_k := 0$ and $s_{n-1} := t$. Let 
\begin{equation} \label{eq:varform}
X_n^{0,\infty}(t) = \inf_{n \leq k < \infty} \left(x_k + \tilde{Y}_{n,k}(t) \right).
\end{equation}
By Proposition 2.4 in \cite{WFSbook}, the infimum in \eqref{eq:varform} is achieved at some finite $k \in \mathbb{N}$ almost surely, provided that the distribution of the initial condition satisfies \eqref{eq:admissible}. (Our notation differs slightly from that of \cite{WFSbook} insofar as they consider a system in which there are infinitely many particles to the \textit{left} and particles reflect off of the trajectories to their left rather than to their right, but the two systems are easily seen to be equivalent, up to a sign change and replacing infimums with supremums.) Furthermore, one may show that, for each $n \in \mathbb{N}$, the following Skorokhod representation holds:
\[
X_n^{0,\infty}(t) = x_n + B_n(t) - L_{(n,n+1)}^{0,\infty}(t),
\]
where 
\[
L_{(n,n+1)}^{0,\infty}(t) = - \inf_{0 \leq s \leq t} \left( X_{n+1}^{0,\infty}(s) - x_n - B_n(s) \right) \wedge 0.
\]
In particular, $\{X_{n+1}^{0,\infty}\}_{n \in \mathbb{N}}$ is a system of competing Brownian particles, in the sense of Definition \ref{def:cbp}, with $p = 0$, and $\bar{g} = \bar{0}$. 

\begin{proof}[Proof of Theorem \ref{thm:exist}(i)]
\textbf{Step 1.} Suppose that $g_j=0$ for all $j \in \NN$. Fix $p \in [0,1)$, $q = 1-p$, and let the initial data $\{x_j\}$, the Brownian motions $\{B_j\}$, and the filtration $\{\clf_t\}$ be as in Section \ref{sec:sepu}. Assume that the distribution $\gamma$ of the initial data satisfies \eqref{eq:admissible}. We will show that there exists a strong solution to \eqref{eq:loctimeeq} driven by Brownian motions $\{B_j\}$ and initial data $\bar x$. We may assume $p > 0$, since when $p = 0$ such a solution is given by \eqref{eq:varform}.
The following argument is adapted from the proof of  \cite[Theorem 3.7]{AS2}.
 
For each $M \geq 2$, let $\{X_j^M(t), t \in [0,\infty), 1 \leq j \leq M\}$ be the unique solution to \eqref{eq:loctimeeq} when $N = M$, with respect to the initial data $\{x_j\}_{1 \leq j \leq M}$, Brownian motions $\{B_j\}_{1 \leq j \leq M}$, collision parameters $p$ and $q$, and drifts $g_j = 0$ for $1 \leq j \leq M$. Denote the associated collection of local times $\{L^M_{(j,j+1)}\}_{0 \leq j \leq M}$.

Similarly, let $\{X_j^{0,M}(t), t \in [0,\infty), 1 \leq j \leq M\}$ denote the unique solution to \eqref{eq:loctimeeq} when $N = M$ with respect to these same initial data and Brownian motions and drifts, but with collision parameters $p = 0$ and $q = 1$.

Recall the solution \eqref{eq:varform} to the infinite system in the $p = 0, q = 1$ case. We claim that, for all $t \geq 0$, $M \geq 2$, and $1 \leq k \leq M$,
\begin{equation} \label{eq:doubbd}
X_k^{0,\infty}(t) \leq X_k^{0,M}(t) \leq X_k^M(t) \leq X_k^{M+1}(t).
\end{equation}
Here the first third inequalities follow from Lemma \ref{lem:bd_by_fin}\ref{it:bf1} and \ref{it:bf2}. To obtain the middle inequality, we proceed by downward induction on $k$. The base case $k = M$ is clear since, by \eqref{eq:loctimeeq}, for all $t \in [0,\infty)$,
$$
X_M^{0,M}(t) = x_M + B_M(t) \leq x_M + B_M(t) + p L_{(M-1,M)}^M(t) = X_M^M(t).
$$
The induction step $k+1 \to k$ proceeds as follows. We may write 
\begin{equation} \label{eq:2to1sided}
X_k^M(t) = x_k + U_k(t) - q L_{(k,k+1)}^M(t), 
\end{equation}
where 
$$
U_k(t) := B_k(t) + p L_{(k-1,k)}^M(t).
$$
By the properties of the local time $L_{(k,k+1)}^M(t)$, $X_k^M(t)$ is the Skorokhod reflection from below of the trajectory $x_k + U_k(t)$ from $X_{k+1}^M(t)$, and by Skorokhod representation, we may write 
\begin{equation} \label{eq:pskor_rep}
q L^M_{(k,k+1)}(t) = -\inf_{0 \leq s \leq t} \left( X_{k+1}^{M}(s) - x_k - U_k(s) \right) \wedge 0.
\end{equation}
On the other hand, by the Skorokhod representation for $X_k^{0,M}$ and the induction assumption, 
\begin{equation} \label{eq:0system_lb}
\begin{split}
X_k^{0,M}(t) & = x_k + B_k(t) + \inf_{0 \leq s \leq t} \left( X_{k+1}^{0,M}(s) - x_k - B_k(s) \right) \wedge 0 \\
& \leq  x_k + B_k(t) + \inf_{0 \leq s \leq t} \left( X_{k+1}^{M}(s) - x_k - B_k(s) \right) \wedge 0 \\
& \leq x_k + U_k(t) + \inf_{0 \leq s \leq t} \left( X_{k+1}^{M}(s) - x_k - U_k(s) \right) \wedge 0 \\
& = x_k + U_k(t) - q L_{(k,k+1)}^M(t) = X_k^M(t).
\end{split}
\end{equation}
Here the third line follows from the fact that the local time $L_{(k-1,k)}^M$ in the definition of $U_k$ is nondecreasing, and we use \eqref{eq:pskor_rep} and \eqref{eq:2to1sided} in the last line. This proves \eqref{eq:doubbd}.

We see from \eqref{eq:doubbd} that for each fixed $t$ and $k$, the sequence $\{X_k^M(t)\}_{M \geq 2}$ is decreasing and bounded from below, and therefore the limit 
\begin{equation} \label{eq:ptwiselim}
X_k(t) := \lim_{M \to \infty} X_k^M(t)
\end{equation}
exists and is finite. From equation \eqref{eq:loctimeeq}, for $t \geq 0$, $M \geq 2$, and $1 \leq k \leq M$, we have 
\begin{equation} \label{eq:Mloceq}
X_k^M(t) = x_k + B_k(t) + pL_{(k-1,k)}^M(t) - qL_{(k,k+1)}^M(t).
\end{equation}
Since $L_{(0,1)}^M \equiv 0$ by definition, taking $k = 1$ in \eqref{eq:Mloceq} and applying the result \eqref{eq:ptwiselim} yields that $\lim_{M \to \infty} L^M_{(1,2)}(t)$ exists and is finite for each fixed $t$. An induction argument then shows that, for each fixed $k$ and $t$, the limit
\begin{equation} \label{eq:L_ptw_lim}
L_{(k,k+1)}(t) := \lim_{M \to \infty} L_{(k,k+1)}^M(t) 
\end{equation}
exists and is finite. We define $L_{(0,1)} \equiv 0$.

We will next show that the limit \eqref{eq:L_ptw_lim} holds uniformly on compact intervals $[0,T]$. By \cite[Corollary 3.9]{AS2}, for any $2 \leq M \leq M'$, $1 \leq k \leq M$, and $0 < s \leq t$,
\begin{equation} \label{eq:comp1}
L_{(k,k+1)}^{M}(t) - L_{(k,k+1)}^{M}(s) \leq L_{(k,k+1)}^{M'}(t) - L_{(k,k+1)}^{M'}(s).
\end{equation}
Letting $M' \to \infty$, we obtain 
\begin{equation} \label{eq:comp2}
L_{(k,k+1)}^{M}(t) - L_{(k,k+1)}^{M}(s) \leq L_{(k,k+1)}(t) - L_{(k,k+1)}(s).
\end{equation}
Taking $s = 0$ in \eqref{eq:comp1} shows that 
$0 \le L_{(k,k+1)}(t) - L_{(k,k+1)}^M(t)$ for all $M \in \NN$, $1\le k \le M$ and $t\ge 0$.
This observation, combined with \eqref{eq:comp2}, yields
\begin{equation} \label{eq:comp3}
0 \leq L_{(k,k+1)}(s) - L_{(k,k+1)}^{M}(s) \leq L_{(k,k+1)}(t) - L_{(k,k+1)}^{M}(t).
\end{equation}
Since the right-hand side converges to zero, this bound shows that the limit \eqref{eq:L_ptw_lim} is uniform on compact intervals $[0,T]$. Note that by \eqref{eq:Mloceq}, this implies that the convergence \eqref{eq:ptwiselim} is uniform on compact time intervals as well.

We will now show that the process $\{X_k\}_{k \in \mathbb{N}}$, with associated local times $\{L_{(k,k+1)}\}_{k \in \mathbb{N}}$, is an infinite system of competing Brownian particles, in the sense of Definition \ref{def:cbp}. Since the limits \eqref{eq:ptwiselim} and \eqref{eq:L_ptw_lim} hold uniformly on compact intervals, continuity follows from continuity of the prelimiting processes $\{X_k^M\}$ and $\{L_{(k,k+1)}^M\}$. The fact that these processes are adapted to $\{\clf_t\}$ and that Properties \ref{pr:p1}-\ref{pr:p2} hold are also easy consequences of the fact that this limit holds and these properties hold for the prelimiting processes. 

The only part of Property \ref{pr:p3} which needs discussion is that the local times $L_{(k,k+1)}$ can only increase at times $t \in [0,\infty)$ when $X_{k+1}(t) = X_k(t)$. 
Fix $\omega \in \Omega$ for which the convergence of $X^M_k(\omega)$, $X^M_{k+1}(\omega)$ and $L^M_{(k,k+1)}(\omega)$ to limit values $X_k(\omega)$, $X_{k+1}(\omega)$ and $L_{(k,k+1)}(\omega)$ holds uniformly on compacts. We suppress $\omega$ from the notation.
Suppose that for some $0 \leq t_1 < t_2 < \infty$, $X_{k+1}(t) > X_k(t)$ for all $t \in [t_1,t_2]$. It suffices to show that $L_{(k,k+1)}(t) = 0$ for all $t \in [t_1,t_2]$.
By continuity, there exists $\varepsilon > 0$ such that $X_{k+1}(t) - X_k(t) \geq \varepsilon$ for all $t \in [t_1,t_2]$, and by uniform convergence, there exists $M_0 \in \mathbb{N}$ such that for all $M \geq M_0$,
$$
X_{k+1}^M(t) - X_k^M(t) \geq \varepsilon/2 \text{ for all } t \in [t_1,t_2].
$$
Then $L^M_{(k,k+1)}(t) = 0$ for all $t \in [t_1,t_2]$ and $M \geq M_0$. Taking the limit as $M \to \infty$ shows that $L_{(k,k+1)}(t) = 0$ for all $t \in [t_1,t_2]$, as desired. This completes Step 1.

\textbf{Step 2.} We will now extend this existence result to the case where the drift parameters $\bar{g} = \{g_j\}_{j \in \mathbb{N}}$ are allowed to be non-zero and lie in $\ell^2(\mathbb{N})$. It suffices to prove existence of solutions over compact time intervals $[0,T]$ for each $T>0$. Fix $T>0$ and define for $0\le t \le T$,
\[
Z(t) = \exp( \sum_{j = 1}^\infty g_j B_j(t) - \frac{1}{2}\sum_{j = 1}^\infty g_j^2 t).
\]
Since $\bar g \in \ell^2$, the first sum in the exponential  converges a.s.\ and $\{Z(t), 0 \le t \le T\}$ is a $\clf_t$-martingale with mean $1$.
Define a probability measure $\tilde \PP_T$ on $(\Omega,\mathcal{F}_T)$ by 
\[
\tilde \PP_T(A) = \mathbb{E}\left[ Z(T)\mathbf{1}_A \right], \;  A \in \clf_T.
\]
Note that, since $\bar x$ is $\clf_0$-measurable, $\PP \circ(\bar x)^{-1} = \tilde \PP_T\circ(\bar x)^{-1}$.
By Girsanov's theorem, under $\tilde \PP_T$,  $\{V_j\}_{j \in \mathbb{N}}$ is a sequence of i.i.d standard Brownian motions on $(\Omega,\mathcal{F}_T)$ and $\{V_i(t+s)- V_i(t), i \in \NN, 0 \le s \le T-t\}$ is independent of $\clf_t$ for all $t\in [0,T]$.
By  Step 1, there exists a strong solution $\{X_j\}_{j \in \mathbb{N}}$ to \eqref{eq:loctimeeq} (over $[0,T]$) with $g_j=0$, $j \in \NN$, driving Brownian motions $\{V_j\}$, and initial data $\bar x$, on the filtered probability space 
$(\Omega,\mathcal{F}_T, \{\clf_t\}_{0\le t \le T}, \tilde \PP_T)$. Since $\tilde \PP_T$ and $\PP$ (restricted to $\clf_T$) are mutually absolutely continuous, \eqref{eq:loctimeeq} is  also satisfied $\PP$-a.s. and we see that under $\PP$, $\{X_j\}_{j \in \mathbb{N}}$ defines a solution of \eqref{eq:loctimeeq} on $(\Omega,\mathcal{F}_T, \{\clf_t\}_{0\le t \le T}, \PP)$ with the given drift sequence $\bar g$, initial condition $\bar x$ and driving Brownian motions $\{B_j\}$. The result follows.
\end{proof}

\begin{acks}[Acknowledgments]
Banerjee was partially supported by NSF-CAREER award DMS-2141621.  Budhiraja was partially supported by NSF DMS-2152577 and DMS-2506010. Banerjee, Budhiraja and Rudzis were partially funded by NSF RTG grant DMS-2134107. Part of this work was carried out while Rudzis was a postdoctoral fellow at the University of North Carolina at Chapel Hill. Budhiraja would also like to thank the Isaac Newton Institute for Mathematical Sciences, Cambridge, for support and hospitality during the programme Stochastic systems for anomalous diffusion, where part of the work on this paper was undertaken. This  was supported by EPSRC grant EP/Z000580/1.
\end{acks}

\bibliographystyle{imsart-number} 
\bibliography{asym_ref}

\end{document}